\documentclass[a4paper,11pt,twoside]{article}

\usepackage{graphicx}
\usepackage{caption}
\usepackage{subcaption}
\usepackage{multirow}
\usepackage[utf8]{inputenc}
\usepackage[T1]{fontenc}
\usepackage{ dsfont }
\usepackage[english]{babel}

\usepackage{charter}

\usepackage{indentfirst}
\usepackage{xcolor}
\usepackage{verbatim}

\usepackage{hyperref}

\usepackage{pifont}

\usepackage[top=3.cm, bottom=4.0cm, left=2.2cm, right=2.2cm]{geometry}
\usepackage{amsmath}
\usepackage{amsthm}	
\usepackage{amsfonts}	
\usepackage{amssymb	
            ,bbm
            ,units 
            ,stmaryrd
           }
\usepackage{enumerate}

\usepackage[square,numbers,sort&compress]{natbib}

\usepackage{
            ulem		
           ,soul		
} 

\numberwithin{equation}{section}
\numberwithin{figure}{section}
 \usepackage[nodayofweek]{datetime}

\renewcommand*{\thefootnote}{\fnsymbol{footnote}}

\title{Euler simulation of interacting particle systems and McKean-Vlasov SDEs with fully superlinear growth drifts in space and interaction
}

\author{
\normalsize Xingyuan Chen\textit{$^{a}$} \\
        \small   X.Chen-176@sms.ed.ac.uk 
\and 
 \normalsize Gon\c calo dos Reis\textit{$^{a,b,}$}\footnote{G.d.R. acknowledges support from the \emph{Funda{\c c}$\tilde{\text{a}}$o para a Ci$\hat{e}$ncia e a Tecnologia} (Portuguese Foundation for Science and Technology) through the project UIDB/00297/2020 and UIDP/00297/2020 (Centro de Matem\'atica e Aplica\c c$\tilde{\text{o}}$es CMA/FCT/UNL).} \\
        \small  G.dosReis@ed.ac.uk
    }

\date{%
    \footnotesize 
    $^{a}$~School of Mathematics, University of Edinburgh, The King's Buildings,  UK
    \\
    $^{b}$~Centro de Matem\'atica e Aplica\c c$\tilde{\text{o}}$es (CMA), FCT, UNL, Portugal
    \\
    \longdate \today \ (\currenttime)
    \vspace{-1.0cm}
}

\theoremstyle{plain}
\newtheorem{theorem}{Theorem}[section]
\newtheorem{lemma}[theorem]{Lemma}
\newtheorem{proposition}[theorem]{Proposition}

\newtheorem{definition}[theorem]{Definition}

\newtheorem{remark}[theorem]{Remark}

\newtheorem{assumption}[theorem]{Assumption}

\newcommand{\bE}{\mathbb{E}}
\newcommand{\bF}{\mathbb{F}}

\newcommand{\bN}{\mathbb{N}}
\newcommand{\bP}{\mathbb{P}}

\newcommand{\bR}{\mathbb{R}}
\newcommand{\bS}{\mathbb{S}}

\newcommand{\cF}{\mathcal{F}}

\newcommand{\cN}{\mathcal{N}}
\newcommand{\cO}{\mathcal{O}}
\newcommand{\cP}{\mathcal{P}}

\newcommand{\trace}{\textrm{Trace}}

\definecolor{darkgreen}{rgb}{0,0.35,0}

\newcommand{\1}{\mathbbm{1}}

\newcommand{\sign}{\text{sign}}

\newcommand{\tnp}{{t_{n+1}}}

\newcommand{\dd}{\mathrm{d}}

\newcommand{\bx}{ \overline{X} }
\newcommand{\by}{ \overline{Y} }
\newcommand{\hx}{ \hat{X} }

\newcommand{\kt}{ \kappa(t) }
\newcommand{\ks}{ {\kappa(s)} }
\newcommand{\hm}{ \hat{\mu} }
\newcommand{\bm}{ \overline{\mu}}

\newcommand{\hp}{ \hat{\Psi}}
\newcommand{\bp}{ \overline{\Psi}}

\begin{document}

\selectlanguage{english}

\maketitle
\renewcommand*{\thefootnote}{\arabic{footnote}}

\begin{abstract} 
We consider in this work the convergence of a split-step Euler type scheme (SSM) for the numerical simulation of interacting particle Stochastic Differential Equation (SDE) systems  and McKean-Vlasov Stochastic Differential Equations (MV-SDEs) with full super-linear growth in the spatial and the interaction component in the drift, and non-constant Lipschitz diffusion coefficient. 

The super-linear growth in the interaction (or measure) component stems from convolution operations with super-linear growth functions allowing in particular application to the granular media equation with multi-well confining potentials. From a methodological point of view, we avoid altogether functional inequality arguments (as we allow for non-constant non-bounded diffusion maps).

The scheme attains, in stepsize, a near-optimal classical (path-space) root mean-square error rate of $1/2-\varepsilon$ for $\varepsilon>0$ and an optimal rate $1/2$ in the non-path-space mean-square error metric. All findings are illustrated by numerical examples. In particular, the testing raises doubts if taming is a suitable methodology for this type of problem (with convolution terms and non-constant diffusion coefficients).
\end{abstract}


\footnotesize
\setcounter{tocdepth}{2}
\tableofcontents
\normalsize

\section{Introduction}

Interactions of organisms, humans, and objects are common phenomena seen easily in collective behaviour within natural and social sciences. Models for interacting particle systems (IPS) and their mesoscopic limits, as the number of particles grows to infinity, receive presently enormous attention given their applicability in areas such as finance, mathematical neuroscience, biology, machine learning, and physics: animal swarming, cell movement induced by chemotaxis, opinion dynamics, particle movement in porous media, electrical battery modelling, self-assembly of particles (see for example \cite{stevens1997aggregation,Holm2006,Carrillo2010,bolley2011stochastic,Dreyer2011phasetransition,Baladron2012,marchioro2012mathematical,Kolokolnikov2013,Bossy2015,Guhlke2018,giesecke2020inference,Carrillo2019,li2019mean,jin2020rbm1} and references). 
In this work, we address the numerical approximation of interacting particle systems given by stochastic differential equations (SDE) and their mesoscopic limit equations (or a class thereof) called McKean--Vlasov Stochastic Differential Equations (MV-SDE) that follow as the scaling limit of an infinite number of particles.     

We understand the IPS as an $N$-dimensional system of $\bR^d$-valued interacting particles where each particle is governed by a Stochastic Differential Equation (SDE). Let $i=1,\cdots,N$ and consider $N$ particles $(X^{i,N}_t)_{t\in[0,T]}$ with independent and identically distributed ${X}_{0}^{i,N}=X_{0}^{i}$ (the initial condition is random, but independent of other particles) and satisfying the $(\bR^d)^N$-valued SDE \eqref{Eq:MV-SDE Propagation} 
\begin{align}
\label{Eq:MV-SDE Propagation}
& \dd {X}_{t}^{i,N} 
 = \big( v (X_t^{i,N}, \mu^{X,N}_{t}  )+ b (t,{X}_{t}^{i,N}, \mu^{X,N}_{t}  )\big) \dd t 
+ \sigma (t,{X}_{t}^{i,N} , \mu^{X,N}_{t} ) \dd W_{t}^{i}
, \quad X^{i,N}_0=X_0^i\in L_{0}^{m}( \bR^{d}) , 
\\
\label{Eq:MV-SDE Propagation-drifts}
& 
\textrm{for } v ( {X}_{t}^{i,N} ,  \mu^{X,N}_{t}  )
= \Big(\frac1N \sum_{j=1}^N f({X}_{t}^{i,N}-{X}_{t}^{j,N}) \Big)  +u ( {X}_{t}^{i,N} ,  \mu^{X,N}_{t}  ) \textrm{ with } \mu^{X,N}_{t}(\dd x) := \frac{1}{N} \sum_{j=1}^N \delta_{X_{t}^{j,N}}(\dd x),
\end{align}
where $\delta_{{X}_{t}^{j,N}}$ is the Dirac measure at point ${X}_{t}^{j,N}$,  $\{W^{i}\}_{ i=1,\cdots,N}$ are independent Brownian motions and $L^m_0(\bR^d)$ denotes the usual $m$th-moment integrable space of $\bR^d$  random variables. 

For the IPS class \eqref{Eq:MV-SDE Propagation}, the limiting class as $N\to \infty$ are called McKean-Vlasov SDEs and the passage to the limit operation is known as ``Propagation of Chaos''. 
This class was first described by McKean \cite{McKean1966}, where he introduced the convolution type interaction (the $v$ in \eqref{Eq:MV-SDE Propagation-drifts}). This is a class of Markov processes associated with nonlinear parabolic equations where the map $v$ in \eqref{Eq:MV-SDE Propagation-drifts} is also called ``self-stabilizing''. The IPS underpinning our work \eqref{Eq:MV-SDE Propagation}-\eqref{Eq:MV-SDE Propagation-drifts}  has been studied widely, from a variety of points of view and as early as \cite{Sznitman1991} (for a general survey under global Lipschitz conditions and boundedness).  

McKean-Vlasov Stochastic Differential Equations (MV-SDEs) with convolution type drifts have general dynamics given by
\begin{align}
\label{Eq:General MVSDE}
\dd X_{t} &= \big( v(X_{t},\mu_{t}^{X}) + b(t,X_{t}, \mu_{t}^{X})\big)\dd t + \sigma(t,X_{t}, \mu_{t}^{X})\dd W_{t},  \quad X_{0} \in L_{0}^{m}( \bR^{d}),
\\
\label{Eq:General MVSDE shape of v}
&
\textrm{where }~ v(x,\mu)
= \int_{\bR^{d}  } f(x-y) \mu(\dd y) +u(x,\mu)
\quad\textrm{with}\quad \mu_{t}^{X}=\textrm{Law}(X_t),
\end{align}
where $\mu_{t}^{X}$ denotes the law of the solution process $X$ at time $t$, $W$ is a Brownian motion in $\bR^d$, $v,f,u,b, \sigma$ are measurable maps along with a sufficiently integrable initial condition $X_0$.

An embodiment (among many) for this typology of models is particle motion modelling that encapsulates three sources of forcing. Namely, the particle moves through a multi-well landscape potential gradient (the map $u$ and $b$), the trajectories are affected by a Brownian motion (and associated diffusion coefficient $\sigma$), and the convolution self-stabilisation forcing characterises the influence of a large population of identical particles (under the same laws of motion $v$ and $f$) on the particle. 
In effect, $v$ acts on the particle as an average attractive/repulsive force exerted on the said particle by a population of similar particles (through the potential $f$), see \cite{2013doublewell,adams2020large} and further examples in \cite{jin2020rbm1}. 
For instance, under certain constraints on $f$ the map $v$ adds inertia to the particle's motion, which in turn delays exit times from the domain of attraction and alters exit locations \cite{HerrmannImkellerPeithmann2008,dosReisSalkeldTugaut2017,adams2020large}. 
The self-stabilisation term in the system induces in the corresponding Fokker-Plank equation a nonlinear term of the form  $\nabla[\rho . \nabla(f \star \rho)]$ (where $\rho$ stands for the processes density while `$\star$' is the usual convolution operator) \cite{Carrillo2010,Carrillo2019,jin2020rbm1}. The granular media Fokker-Plank equation from biochemistry is a good example of an equation featuring this kind of structure  \cite{malrieu2006concentration,cattiaux2008probabilistic,adams2020large}. 
The literature on MV-SDE is growing explosively with many contributions addressing well-posedness, regularity, ergodicity, nonlinear Fokker-Planck equations, large deviations \cite{HuangRenWang2021DDSDE-wellposedness,dosReisSalkeldTugaut2017,adams2021operatorsplitting,adams2021entropic}. The convolution framework has been given particular attention as it underpins many settings of interest \cite{malrieu2006concentration,cattiaux2008probabilistic,harang2020pathwise,2013doublewell}. 
The literature is even richer under the restriction to a constant diffusion term, $\sigma=\textrm{const}$, as it gives access to methodologies based on Langevin-type dynamics but also to the machinery of Functional inequalities (e.g., log-Sobolev and  Poincare inequalities). We point to \cite{harang2020pathwise} for a nice overview on several \textit{open} problems of interest where $f$ is a singular kernel (and $\sigma$ is a constant): including Coulomb interaction $f(x)=x/|x|^d$, Bio-Savart law $f(x)=x^{\bot}/|x|^d$; Cucker-Smale models $f(x)=(1+|x|^2)^{-\alpha}$ for $\alpha>0$; crystallisation $f(x) = |x|^{-2p} - 2|x|^{-p}$ and take $p\to \infty$; 2D viscous vortex model with $f(x)=x/|x|^2$  \cite{fournier2014propagation}.

\smallskip

\emph{Super-linear interaction forces.} 
For the IPS \eqref{Eq:MV-SDE Propagation}-\eqref{Eq:MV-SDE Propagation-drifts} or the MV-SDE \eqref{Eq:General MVSDE}-\eqref{Eq:General MVSDE shape of v},  we focus on the class where the involved functions are not (necessarily) globally Lipschitz functions. Concretely, the map $v$ is a super-linear growth function in both space and measure component --- we assume that $f$ and $u$ in \eqref{Eq:General MVSDE shape of v} behave like a general polynomial but also satisfy a one-sided Lipschitz condition to control for radial growth (the specific details are given in Assumption \ref{Ass:Monotone Assumption} below); the maps $b$ and $\sigma$ are assumed globally Lipschitz functions.  

From the theoretical point of view, this class is presently well understood. Well-posedness was generally established in \cite{adams2020large}; \cite{herrmann2012self} investigate different properties of the invariant measures for particles in double-well confining potential and later \cite{2013doublewell} investigate the convergence to stationary states. Large deviations and exit times for such self-stabilising diffusions are established in  \cite{HerrmannImkellerPeithmann2008,adams2020large}. The study of probabilistic properties and parametric inference (under constant diffusion) for this class is given in \cite{GENONCATALOT2021}. Two recent studies on parametric inference \cite{belomestny2021semiparametric,comte2022nonparametric} include numerical studies for the particle interaction (\cite{GENONCATALOT2021} does not) but do not tackle super-linear growth in the interaction component (\cite{GENONCATALOT2021} does).



To the best of our knowledge and except for \cite{Malrieu2003}, no numerical methods exists for this class as no general method allows for super-linear growth interaction kernels. For emphasis, standard SDE results for super-linear growth drifts do not yield convergence results independent of the number of particles $N$. In other words, by treating the interacting particle system \eqref{Eq:MV-SDE Propagation} as an $(\mathbb R^{d})^N$-dimensional SDE known results from SDE numerics with coefficients with super-linear growth can be applied directly.  \textit{However}, all estimates would depend on the system's dimension, $Nd$, and hence ``explode'' as $N$ tends to infinity. In this work, we introduce new technical elements to overcome this difficulty, which, to the best of our knowledge, are new. It's noteworthy to observe that the direct numerical discretization of the IPS system \eqref{Eq:MV-SDE Propagation}-\eqref{Eq:MV-SDE Propagation-drifts} leads to a costly computational cost of $\cO(N^2)$ and hence care is needed.

Many of the current numerical methods in the literature of MV-SDEs rely on the particle approximation given by the IPS, and the known quantified rate for the propagation of chaos \cite{adams2020large,chaintron2021propagation,lacker2021hierarchies,Lacker2018}: taming \cite{reis2018simulation,kumar2020explicit}, time-adaptive \cite{reisinger2020adaptive}, early Split-Step Methods (SSM) methods \cite{2021SSM} -- all these contributions allow for super-linear growth in space only. 
Further noteworthy contributions include  \cite{belomestny2018projected,hutzenthaler2022multilevel,beck2021full,gobetpagliarani2018,agarwalstefano2021,bossytalay1997,crisanmcmurray2019,reis2018importance,talayvaillant2003}. 
Within the existing literature, no method can deal with a super-linear growth $f$ component; all cited works make the assumption of a Lipschitz behaviour in $\mu \mapsto v(\cdot,\mu)$ (which, in essence, entail that $\nabla f$ is bounded).

 \smallskip
 
\textbf{Our contribution.} 
\emph{The results of this manuscript provide for both the numerical approximation of interacting particle SDE systems  \eqref{Eq:MV-SDE Propagation}-\eqref{Eq:MV-SDE Propagation-drifts}, and McKean--Vlasov SDEs \eqref{Eq:General MVSDE}-\eqref{Eq:General MVSDE shape of v}.} 

The main contribution of this work is the numerical scheme and its convergence analysis. We present a particle approximation SSM algorithm inspired in \cite{2021SSM} for the numerical approximation of MV-SDEs and associated particle systems with drifts featuring super-linear growth in space and measure, and where the diffusion coefficient satisfies a general Lipschitz condition. The well-posedness result (Theorem \ref{Thm:MV Monotone Existence} below) and Propagation of Chaos (Proposition \ref{Prop:Propagation of Chaos} below) follow from known literature \cite{adams2020large} -- in fact, our Proposition \ref{Prop:Propagation of Chaos} establishes the well-posedness of the particle system hence closing the small gap present in \cite[Theorem 3.14]{adams2020large}. The only existing work tackling this involved setting via a fully implicit scheme is \cite{Malrieu2003}. They rely on (Bakry-Emery) functional inequalities methodologies under specific structural assumptions (constant elliptic diffusion, $u=b=0$ and differentiability) that we do not make. 

The scheme we propose is a split-step scheme inspired in \cite{2021SSM} (see Definition \ref{def:definition of the ssm} below) that first solves an implicit equation given by the SDE's drift component only then takes that outcome and feeds it to the remaining dynamics of the SDE via a standard Euler step. 
The idea is that the implicit step deals with the problematic super-linear growth part, and the elements passed to the Euler step are better behaved. In \cite{2021SSM}, there is only super-linear growth in the space variables, and the measure component is assumed Lipschitz; here both space and measure component have super-linear growth. From a practical point of view, the implicit step in \cite{2021SSM} for a particle $i$ only depended on the elements of particle $i$ (the measure being fixed to the previous time step); hence one solves $N$ decoupled equations in $\bR^d$. In this manuscript, the implicit step for particle $i$ involves the whole system of particles entailing that one needs to solve one-single system but in $(\bR^d)^N$ and the solution depends on all terms. This change in the scheme makes it much harder to obtain moment estimates for the scheme. For the setting of \cite{2021SSM} there were already several competitive schemes present in the literature, e.g., taming \cite{reis2018simulation,kumar2020explicit} and time-adaptive \cite{reisinger2020adaptive} and the numerical study there was comparative. For this work, no alternative numerical scheme exists -- see below for further discussion regarding the implementation of taming for this class.

Results-wise, we provide two convergence results in the strong-error\footnote{We understand a ``strong'' error metric as a metric that depends on the joint distribution of the true solution and the numerical approximation. In contrast to the weak error where one needs only the marginals separately. Theorem \ref{theorem:SSM: strong error 1} and \ref{theorem:SSM: strong error 2} showcase two ``strong'' but different error metrics.} sense. For the classical (path-space) root mean-square error, see Theorem \ref{theorem:SSM: strong error 2},  we achieve a nearly-optimal convergence rate of $1/2-\varepsilon$ with $\varepsilon>0$. 
The main difficulty, also where one of our main contributions lie, is in establishing higher-order moment bounds for the numerical scheme in a way that is compatible with the convolution component in \eqref{Eq:MV-SDE Propagation-drifts} or \eqref{Eq:General MVSDE shape of v} and It\^o-type arguments -- see Theorem \ref{theorem：moment bound for the big theorm time extensions }. We provide a second strong (non-path-space) mean-square error criteria, see Theorem \ref{eq: scheme continous extension in SDE form}, that attains the optimal rate $1/2$. This 2nd result requires only the higher moments of the IPS' solution process and the 2nd-moments of the numerical approximation \cite{2015ssmBandC} (which are easier to obtain). 
We emphasise that this 2nd notion of strong convergence (see Theorem \ref{eq: scheme continous extension in SDE form}) is also standard (albeit less) within Monte Carlo literature. It also controls the variance of the approximation error (simply not in path-space). Hence, it is sufficient for the many uses one can give to the simulation output -- as one would do given any other Monte Carlo estimators (e.g., confidence intervals). Lastly, we show that with a constant diffusion coefficient, one attains the higher convergence rate of $1.0$ (see Theorem \ref{theorem：gm strong error}). 

We illustrate our findings with extended numerical tests showing agreement with the theoretical results and discussing other properties for schemes: periodicity in phase-space, the impact of the number of particles and numerical rate of Propagation of Chaos, and complexity versus runtime. For comparison, we implement the taming algorithm \cite{reis2018simulation} for the setting (without proof) and find that in the example with constant diffusion, taming performs similarly to the SSM. In the non-constant diffusion example, it performs very poorly. This latter finding raises questions (for future research) if taming is a suitable methodology for this class.

\smallskip

\textbf{Organisation of the paper.} In Section \ref{sec:two} we set the notation and framework. In Section \ref{section:SSM scheme and main results.}, we state the SSM scheme and the two main convergence results. Section \ref{sec:examples} provides numerical illustrations (for the granular media model and a double-well model with non-constant diffusion). All proofs are given in Section \ref{sec:theSSMresults}.


%
%
%

%
%
%
\section{The split-step method for MV-SDEs and interacting particle systems}
\label{sec:two}

We follow the notation and framework set in  \cite{adams2020large,2021SSM}.

\subsection{Notation and Spaces}
\label{section: notations and space}
Let $\bN$ be the set of natural numbers starting at $0$, $\bR$ denotes the real numbers. For $a,b\in \bN$ with $a\leq b$,  define $\llbracket a,b\rrbracket:= [a,b] \cap \bN =  \{a,\cdots,b\}$. 
For $x,y \in \bR^d$ denote the scalar product of vectors by $x \cdot y$; and $|x|=(\sum_{j=1}^d x_j^2)^{1/2}$ the Euclidean distance. The $\mathbf{0}$ denotes the origin in $\bR^d$. Let $\1_A$ be the indicator function of set $A\subset \bR^d$. For a matrix $A \in \bR^{d\times n}$ we denote by $A^\intercal$ its transpose and its Frobenius norm by $|A|=\trace\{A A^\intercal\}^{1/2}$. Let $I_d:\bR^d\to \bR^d$ be the identity map. For collections of vectors, let the upper indices denote the distinct vectors, whereas the lower index is a vector component, i.e., $x^l_j$ denote the $j$-th component of $l$-th vector. $\nabla$ denotes the vector differential operator, $\partial$ denotes the partial differential operator.

We introduce over $\bR^d$ the space of probability measures $\cP(\bR^d)$ and its subset $\cP_2(\bR^d)$ of those with finite second moment. The space $\cP_2(\bR^d)$ is Polish under the Wasserstein distance
\begin{align}
\label{eq:def of wasserstein distance}
W^{(2)}(\mu,\nu) = \inf_{\pi\in\Pi(\mu,\nu)} \Big(\int_{\bR^d\times \bR^d} |x-y|^2\pi(\dd x,\dd y)\Big)^\frac12, \quad \mu,\nu\in \cP_2(\bR^d) .
\end{align}   
where $\Pi(\mu,\nu)$ is the set of couplings for $\mu$ and $\nu$ such that $\pi\in\Pi(\mu,\nu)$ is a probability measure on $\bR^d\times \bR^d$ such that $\pi(\cdot\times \bR^d)=\mu$ and $\pi(\bR^d \times \cdot)=\nu$. 

Let our probability space be a completion of $(\Omega, \bF, \cF,\bP)$ with $\bF=(\cF_t)_{t\geq 0}$ carrying an $\bR^l$-valued Brownian motion $W=(W^1,\cdots,W^l)$ and generating the probability space's filtration, augmented by all $\bP$-null sets, and with an additionally sufficiently rich sub $\sigma$-algebra $\cF_0$ independent of $W$. We denote by $\bE[\cdot]=\bE^\bP[\cdot]$ the usual expectation operator with respect to $\bP$.

We consider some finite terminal time $T<\infty$ and use the following notation for spaces (standard in the (McKean-Vlasov) SDE literature \cite{reis2018simulation,2021SSM}).  
For $0\leq t\leq T$, let $L_{t}^{p}\left(\mathbb{R}^{d}\right)$ define the space of $\mathbb{R}^{d}$-valued, $\mathcal{F}_{t}$-measurable random variables $X$, that satisfy $\mathbb{E}\left[|X|^{p}\right]^{1 / p}<\infty$. 
Define $\mathbb{S}^{m}([0,T])$ to be, for $m \geqslant 1$, the space of $\mathbb{R}^{d}$ -valued, $\cF_\cdot$-adapted processes $Z$, that satisfy $\mathbb{E}\left[\sup _{0 \leqslant t \leqslant T}|Z_t|^{m}\right]^{1 / m}<\infty$.

Throughout the text, $C$ denotes a generic constant positive real number that may depend on the problem's data, may change from line to line but is always independent of the constants $h,M,N$ (associated with the numerical scheme and specified below) but possibly depend on the terminal time $T$ (and other fixed problem data).

\subsection{Framework}

Let $W$ be an $l$-dimensional Brownian motion and take the measurable maps $v:\bR^d \times\cP_2(\bR^d) \to \bR^d$, $f:\bR^d \to \bR^d$, $b:[0,T] \times \bR^d \times\cP_2(\bR^d) \to \bR^d$ and $\sigma:[0,T] \times \bR^d \times \cP_2(\bR^d) \to \bR^{d\times l}$.
The MV-SDE of interest for this work is Equation \eqref{Eq:General MVSDE} (for some $m \geq 1$), where $\mu_{t}^{X}$ denotes the law of the process $X$ at time $t$, i.e., $\mu_{t}^{X}=\bP\circ X_t^{-1}$. 
We make the following assumptions on the coefficients.
\begin{assumption}
\label{Ass:Monotone Assumption} Let $b$ and $\sigma$ $1/2$-H\"{o}lder continuous in time, uniformly in $x\in \bR^d$ and $\mu\in \cP_2(\bR^d)$. Assume that $b,\sigma$ are uniformly Lipschitz in the sense that there exists $L_b, L_\sigma\ge0$ such that for all $t \in[0,T]$ and all $x, x'\in \bR^d$ and $\forall \mu, \mu'\in \cP_2(\bR^d)$ we have that 
\begin{align*}
&(\mathbf{A}^b)
&|b(t, x, \mu)-b(t, x', \mu')|^2
\leq L_b \big(|x-x'|^2 + W^{(2)}(\mu, \mu')^2 \big),
\\
&(\mathbf{A}^\sigma)
&|\sigma(t, x, \mu)-\sigma(t, x', \mu')|^2\leq 
L_\sigma \big(|x-x'|^2 + W^{(2)}(\mu, \mu')^2 \big).
\end{align*}
$(\mathbf{A}^u)~$ Let $u$  satisfy: there exist $ L_{u} \in \bR$, $L_{\hat{u}}>0$,  $L_{\tilde{u} } \ge 0$, 
$q_1>0$ 
such that for all $t\in[0,T]$, $ x, x'\in \bR^d$ and  $\forall \mu, \mu'\in \cP_2(\bR^d)$, it holds that 
\begin{align*}
\langle x-x', u(x,\mu)-u(x',\mu) \rangle 
& 
\leq L_{u}|x-x'|^{2}
& \textrm{(One-sided Lipschitz in space)},
\\
|u(x,\mu)-u( x',\mu)| 
& 
\leq L_{\hat{u}}(1+ |x|^{q_1} + |x'|^{q_1}) |x-x'| 
& \textrm{(Locally Lipschitz in space)},
\\
|u(x,\mu)-u( x,\mu')|^2 
&\leq L_{\tilde{u} } W^{(2)}(\mu, \mu')^2
& \textrm{(Lipschitz in measure)}.
\end{align*}

$(\mathbf{A}^f)~$ Let $f$  satisfy: there exist $ L_{f} \in \bR$, $L_{\hat{f}}>0$, 
$q_2>0$ 
such that for all $t\in[0,T]$, $ x, x'\in \bR^d$, it holds that 
\begin{align*}
\langle x-x', f(x)-f(x') \rangle 
& 
\leq L_{f}|x-x'|^{2}
& \textrm{(One-sided Lipschitz)},
\\
|f(x)-f( x')| 
& 
\leq L_{\hat{f}}(1+ |x|^{q_2} + |x'|^{q_2}) |x-x'| 
& \textrm{(Locally Lipschitz)},\\
f(x)&=-f(-x), 
& \textrm{(Odd function)}.
\end{align*}
Assume the normalisation\footnote{This constraint is a soft as the framework allows to easily redefine $f$ as $\hat f(x):=f(x)-f(\mathbf{0})$ with $f(\mathbf{0})$ merged into $b$.} $f(\mathbf{0})=\mathbf{0}$. Lastly, and for convenience, we set $q=\max\{q_1,q_2\}$ (and we have $q>0$).
\end{assumption} 
The benefits of choosing drift=$v+b$ with $b$ being uniformly Lipschitz are discussed below in Remark \ref{rem:Constraint on h is soft} (see also \cite{2021SSM}). Certain useful properties can be derived from these assumptions.
\begin{remark}[Implied properties]
\label{remark:ImpliedProperties}
Under Assumption \ref{Ass:Monotone Assumption}, take some $C>0$. Then for all $t \in [0,T]$, $x,x',z\in \bR^{d}$ and $\mu\in \cP_2(\bR^d)$, since $f$ is a normalised odd function (i.e., $f(\mathbf{0})=\mathbf{0}$), we have 
\begin{align*}
&\langle x,f(x)\rangle = \langle x-\mathbf{0},f(x)-f(\mathbf{0})\rangle+\langle x,f(\mathbf{0})\rangle \le L_f|x|^2+|x||f(\mathbf{0})|= L_f|x|^2.
\end{align*}
Also, for the function $u$, define $ \widehat L_u=L_u+1/2,~C_u= |u(0,\delta_0) |^2$, and thus by Young's inequality
\begin{align*}
\langle x,u(x,\mu)\rangle &
\le
C_u +\widehat L_u|x|^2+  L_{\tilde{u} } W^{(2)}(\mu,\delta_0)^2
,~\langle x-x',u(x,\mu)-u(x',\mu') \rangle
 \le \widehat L_u|x-x'|^2 + \frac{ L_{\tilde{u} }}{2} W^{(2)}(\mu,\delta_0)^2.
\end{align*}
Using the properties of the convolution, $v$ of \eqref{Eq:General MVSDE} also satisfies a one-sided Lipschitz condition in space 
\begin{align*}
    \langle x-x',v(x,\mu)-v(x',\mu) \rangle
    &\le \int_{\bR^{d} } L_f |x-x'|^2  \mu(dz) +  L_u |x-x'|^2
    ~=~ (L_f+L_u) |x-x'|^2.
\end{align*}
Moreover, for $\psi\in\{b,\sigma\}$, by Young's inequality, we have 
\begin{align*}
\langle x,\psi(t,x,\mu) \rangle 
\le
C(1+|x|^2+W^{(2)}(\mu,\delta_0)^2 )
\quad \textrm{and}\quad 
  |\psi(t,x,\mu)|^2
   \le
   C(1+|x|^2+W^{(2)}(\mu,\delta_0)^2 ).
\end{align*}
\end{remark}

We first recall a result from \cite{adams2020large} establishing well-posedness of the MV-SDE \eqref{Eq:General MVSDE}-\eqref{Eq:General MVSDE shape of v}. 
\begin{theorem}[Theorem 3.5 in \cite{adams2020large}]
\label{Thm:MV Monotone Existence}
	Let Assumption \ref{Ass:Monotone Assumption} hold and assume for some $m >2(q+1)$, $X_{0} \in L_{0}^{m}(\bR^{d})$.
	Then, there exists a unique solution $X$ to MV-SDE \eqref{Eq:General MVSDE} in $\bS^{m}([0,T])$.  
	For some constant $C>0$ (depending on $T$ and $m$) we have
	\begin{align*}
	\mathbb E \big[ \sup_{0 \leq t \leq T} |X_{t}|^{\widehat m} \big] 
	\leq C \big(1+ \bE\big[|X_0|^{\widehat m}\big]\big) e^{C T},\qquad \textrm{for any }~ \widehat m \in [2,m].
	\end{align*} 
\end{theorem}
\begin{proof}
Our Assumption \ref{Ass:Monotone Assumption} is a particularisation of  \cite[Assumption 3.4]{adams2020large} and hence our theorem follows directly from  \cite[Theorem 3.5]{adams2020large}. 
\end{proof}

\textbf{The interacting particle system \eqref{Eq:MV-SDE Propagation}.} 
As mentioned earlier, the numerical approximation results of this work apply directly if either one's starting point is the interacting particle system \eqref{Eq:MV-SDE Propagation} or if one's starting point is the MV-SDE  \eqref{Eq:General MVSDE}. On the latter, one can approximate the MV-SDE \eqref{Eq:General MVSDE} (driven by the Brownian motion $W$) by the $N$-dimensional system $\bR^d$-valued interacting particle system given in \eqref{Eq:MV-SDE Propagation} and approximate it numerically with the gap closed by the Propagation of Chaos  \cite{2021SSM,reisinger2020adaptive,reis2018simulation}.

For completeness we recall the setup of  \eqref{Eq:MV-SDE Propagation}. Let $i\in \llbracket 1,N\rrbracket$ and consider $N$ particles $(X^{i,N})_{t\in[0,T]}$ with independent and identically distributed (i.i.d.) initial conditions ${X}_{0}^{i,N}=X_{0}^{i}$ and satisfying the $(\bR^d)^N$-valued SDE \eqref{Eq:MV-SDE Propagation} (with $v$ given in \eqref{Eq:General MVSDE shape of v})
\begin{align*}
\dd {X}_{t}^{i,N} 
= \big( v (X_t^{i,N}, \mu^{X,N}_{t}  )+ b (t,{X}_{t}^{i,N}, \mu^{X,N}_{t}  )\big) \dd t 
+ \sigma (t,{X}_{t}^{i,N} , \mu^{X,N}_{t} ) \dd W_{t}^{i}
, \quad X^{i,N}_0=X_0^i, 
\end{align*}
where $\mu^{X,N}_{t}(\dd x) := \frac{1}{N} \sum_{j=1}^N \delta_{X_{t}^{j,N}}(\dd x)$ with $\delta_{{X}_{t}^{j,N}}$ being the Dirac measure at point ${X}_{t}^{j,N}$, and  $W^{i}, i\in \llbracket 1,N\rrbracket$ being independent Brownian motions (also independent of the BM $W$ appearing in \eqref{Eq:General MVSDE}; with a slight abuse of notation to  avoid re-defining the probability space's filtration).

\begin{remark}[The system through the lens of $\bR^{Nd}$] 
\label{remark:OSL for the whole function / system V}
We introduce the map $V$ to interpret \eqref{Eq:MV-SDE Propagation} as one system of equations in $\bR^{Nd}$ instead of $N$ dependent equations each in $\bR^d$. Namely, we define for $v$ given by \eqref{Eq:General MVSDE shape of v}, 
\begin{align}
\label{eq:Define-Function-V}
 &  V=(V_1,\cdots,V_N):(\bR^{d})^N\to (\bR^{d})^N 
 \ \textrm{ where for }\ i\in \llbracket 1,N\rrbracket \ \ V_i:(\bR^{d})^N\to \bR^{d},  
 \quad
 V_i(X^N)= v(X^{i,N}, \mu^{X,N}), 
\end{align}
and $X^N=(X^{1,N},\cdots,X^{N,N} )\in \bR^{Nd}$ where each $X^{i,N}$ solves \eqref{Eq:MV-SDE Propagation}, $\mu^{X,N}(\dd x) := \frac{1}{N} \sum_{j=1}^N \delta_{X^{j,N}}(\dd x) $.  

For $X^N,Y^N\in \bR^{Nd}$ with corresponding measure $\mu^{X,N},\mu^{Y,N} $ and letting Assumption \ref{Ass:Monotone Assumption} hold, the function $V$ also satisfies a one-sided Lipschitz condition 
\begin{align*}
    \langle& X^N-Y^N,V(X^N)-V(Y^N) \rangle 
    \\
    &
    =\frac{1}{2N}\sum_{i=1}^N\sum_{j=1}^N \Big\langle (X^{i,N}-X^{j,N})-(Y^{i,N}-Y^{j,N})
    , f(X^{i,N}- X^{j,N})-f(Y^{i,N}- Y^{j,N})\Big\rangle
    \\
    &\quad + 
    \sum_{i=1}^N \Big\langle X^{i,N}-Y^{i,N}, 
    u(X^{i,N}, \mu^{X,N})-u(Y^{i,N}, \mu^{X,N})+ u(Y^{i,N}, \mu^{X,N}) -u(Y^{i,N}, \mu^{Y,N})
    \Big\rangle
    \\
    & \qquad
    \le (2L_f^++L_u+\frac{1}{2}+\frac{L_{\tilde{u} }}{2})|X^N-Y^N|^2, \qquad L_f^+=\max\{0,L_f\}.
\end{align*}
In the last second step we changed the order of summation and used that $f$ is odd. 

\end{remark}

\textbf{Propagation of chaos (PoC).} 
In order to show that the particle approximation \eqref{Eq:MV-SDE Propagation} is of effective use to approximate the MV-SDE \eqref{Eq:General MVSDE}, we provide  a pathwise propagation of chaos result (convergence as the number of particles increases and with rate). 
We introduce the auxiliary system of non interacting particles  
\begin{align}
	\label{Eq:Non interacting particles}
	\dd X_{t}^{i} = 
	\big( 
	v(X_{t}^{i}, \mu^{X^{i}}_{t})+ b(t, X_{t}^{i}, \mu^{X^{i}}_{t}) \big)\dd t + \sigma(t,X_{t}^{i}, \mu^{X^{i}}_{t}) \dd W_{t}^{i}, \quad X_{0}^{i}=X_{0}^{i} \, ,\quad t\in [0,T] \, ,
\end{align}
which are just (decoupled) MV-SDEs with i.i.d. initial conditions  $X_{0}^{i}$. Since the $X^{i}$s are
independent, $\mu^{X^{i}}_{t}=\mu^{X}_{t}$ for all $
i$ (and $\mu^{X}_{t}$ the law of the solution to \eqref{Eq:General MVSDE} with $v$ given as \eqref{Eq:General MVSDE shape of v}). 

The Propagation of chaos result \eqref{eq:poc result} follows from \cite[Theorem 3.14]{adams2020large} under the assumption that the interacting particle system \eqref{Eq:MV-SDE Propagation} is well-posed. The first statement of Proposition \ref{Prop:Propagation of Chaos} establishes the well-posedness of the particle system hence closing the small gap left in  \cite[Theorem 3.14]{adams2020large}. 
\begin{proposition}
	\label{Prop:Propagation of Chaos}
	Let the assumptions of Theorem \ref{Thm:MV Monotone Existence} hold for some $m>2(q+1)$. 
	Then, for all $i\in \llbracket 1,N \rrbracket$ there exists a unique solution $X^{i,N}$ to \eqref{Eq:MV-SDE Propagation} in $\bS^m([0,T])$  and for any $1\leq p\leq m$ there exists $C>0$ independent of $N$ (but depending on $T$ and $m$) such that 
\begin{align}
\label{eq:momentboundParticiInteractingSystem}
    \sup_{t\in[0,T]}\sup_{ i\in \llbracket 1,N \rrbracket}  \bE\big[  |X^{i,N}_t|^p \big] \leq C\Big(1+\bE\big[\,|X_0^\cdot|^p \big]\Big).
\end{align}
For $i\in \llbracket 1,N \rrbracket$, let $X^{i}\in \bS^m([0,T])$ be the solution to \eqref{Eq:Non interacting particles}, ensured by Theorem \ref{Thm:MV Monotone Existence}. Suppose additionally that $m>\max\{ 2(q+1) , 4\} $. Then, there exists a constant $C>0$ independent of $N$ (but depending on $T$ and $m$) such that
	\begin{align}
	\label{eq:poc result}
	\sup_{ i\in \llbracket 1,N \rrbracket}\sup_{0 \le t \le T} 
	\bE\big[ |X_{t}^{i} - X_{t}^{i,N}|^{2}\big] 
	\le C  \begin{cases}N^{-1 / 2}, & d<4 \\ N^{-1 / 2} \log N, & d=4 \\ N^{\frac{-2}{d+4}}, & d>4\end{cases}.
	\end{align}
\end{proposition}

 The proof and further details are presented in Appendix \ref{appendix: proof of moment bound for PC}. 
This result shows that the particle scheme will converge to the MV-SDE with a given quantified rate. Therefore, to show convergence between our numerical scheme and the MV-SDE, we only need to show that the numerical version of the particle scheme converges to the ``true'' particle scheme in a way that is independent of $N$. We note that the PoC rate can be optimised for the case of constant diffusion \cite[Remark 2.5]{2021SSM}.

%
%
%
%
%
\subsection{The scheme for the interacting particle system and main results}
\label{section:SSM scheme and main results.}
The split-step method (SSM) here is inspired by that of \cite{2021SSM} and re-cast accordingly to the setup here. The critical difficulty arises from the convolution component in $v$  \eqref{Eq:General MVSDE}. This term is the main hindrance in proving moment bounds. Before continuing recall the definition of $V$ in Remark \ref{remark:OSL for the whole function / system V}. We now introduce the SSM numerical scheme. 
\begin{definition}[Definition of the SSM]
\label{def:definition of the ssm}
Let Assumption \ref{Ass:Monotone Assumption} hold. 
Define the uniform partition of $[0,T]$ as $\pi:=\{t_n:=nh : n\in \llbracket 0,M\rrbracket, h:=T/M \}$ for a prescribed $M\in \bN\setminus\{0\}$. 
Define recursively the SSM approximating \eqref{Eq:MV-SDE Propagation} as: set $\hx_{0}^{i,N}=X^i_0$ for $i\in \llbracket 1,N\rrbracket $; iteratively over  $n\in \llbracket 0,M-1\rrbracket$ for all $i\in \llbracket 1,N\rrbracket $ (recall Remark \ref{remark:OSL for the whole function / system V} and the definition of the map $V$ in \eqref{eq:Define-Function-V})\color{black}
\begin{align}
Y_{n}^{\star,N} &=\hat{ X}_{n}^{N}+h V (Y_{n}^{\star,N} ),
\quad  \hat{ X}_{n}^{N}=(\cdots,\hx_{n}^{i,N},\cdots),\quad Y_{n}^{\star,N}=(\cdots,Y_{n}^{i,\star,N},\cdots),
\label{eq:SSTM:scheme 0}
\\
\label{eq:SSTM:scheme 1}
&\textrm{where }~Y_{n}^{i,\star,N} =\hx_{n}^{i,N}+h v (Y_{n}^{i,\star,N},\hm^{Y,N}_n ),  
 \quad 
 \quad 
  \hm^{Y,N}_n(\dd x):= \frac1N \sum_{j=1}^N \delta_{Y_{n}^{j,\star,N}}(\dd x),
 \\
\label{eq:SSTM:scheme 2}
\hx_{n+1}^{i,N} &=Y_{n}^{i,\star,N}
            + b(t_n,Y_{n}^{i,\star,N},\hm^{Y,N}_n) h
            +\sigma(t_n,Y_{n}^{i,\star,N},\hm^{Y,N}_n) \Delta W_{n}^i,\qquad \Delta W_{n}^i=W_{t_{n+1}}^i-W_{t_n}^i.
\end{align}
The stepsize $h$ is chosen as to belong to the interval (this constraint is soft in the sense of Remark \ref{rem:Constraint on h is soft}) 
\begin{align}
\label{eq:h choice}
h\in \Big(0, \min\big\{1,\frac 1\zeta\big\} \Big)
\quad \textrm{for $\zeta$ defined as}\quad  
\zeta= \max\Big\{ 2(L_f+L_u)~,~4L_f^++2L_u+2 L_{\tilde{u} }	+1~,~0 \Big\}.
\end{align}
\end{definition}

In some cases where the original functions $f,u$ might cause trouble to find a suitable choice of $h$, and by the Remark below, we can use the addition and subtraction trick to bypass the constraint, see Remark \ref{remark: choice of h} and \cite[Section 3.4]{2021SSM} for more discussion.
\color{black}
\begin{remark}[The constraint on $h$ in \eqref{eq:h choice} is soft]
\label{rem:Constraint on h is soft}
Our framework allows to change $f,u,b$ in such a way as to have $\zeta=0$ in \eqref{eq:h choice} via addition and subtraction of linear terms to $f,u$ and $b$. Concretely, take $\theta,\gamma \in \bR$ and redefine $f,u,b$ into $\widehat f, \widehat u, \widehat b$ as follows: for any $t\in [0,\infty),x\in\bR^d, \mu\in \cP_2(\bR^d)$
\begin{align*}
\widehat f (x) = f(x) -\theta x, 
\qquad 
\widehat u(x,\mu) = u(x,\mu) -\gamma x - \theta \int_{\bR^d} z \mu(\dd z),
\quad \textrm{and}\quad 
\widehat b(t,x,\mu) = b(t,x,\mu) +(\gamma + \theta) x.
\end{align*}
For judicious choices of $\theta,\gamma$ it is easy to see that $\zeta$ can be set to be zero (we invite the reader to carry out the calculations). We remark that this operation increases the Lipschitz constant of $\widehat b$.
\end{remark}
Recall that the function $V$ satisfies a one-sided Lipschitz condition in $X\in\bR^{Nd}$ (Remark \ref{remark:OSL for the whole function / system V}), and hence (under \eqref{eq:h choice}) a unique solution $Y_{n}^{\star,N}$ to \eqref{eq:SSTM:scheme 0} as a function of $\hat{ X}_{n}^{N}$ exists (details in Lemma \ref{lemma:SSTM:new functions def and properties1}).  
After introducing the discrete scheme, we define its continuous extension and provide the main convergence results.
\begin{definition} [Continuous extension of the SSM]
\label{def:definition of the ssm continouse extension}
Under the same choice of $h$ and assumptions in Definition \ref{def:definition of the ssm}, for all $t\in[t_n,\tnp]$, $n\in \llbracket 0,M-1\rrbracket$ ,  $i\in\llbracket 1,N \rrbracket$, $\hx_{0}^{i,N}=X^i_0 \in L_0^m(\mathbb{R}^d)$, the continuous extension of the SSM is 
\begin{align}
\label{eq: scheme continous extension in SDE form}
    \dd \hat{X}_{t}^{i,N}
    &
    =
    \big( v (Y_{\kappa(t)}^{i,\star,N},\hm^{Y,N}_{\kappa(t)} )
    +b(\kappa(t),Y_{\kappa(t)}^{i,\star,N},\hm^{Y,N}_{\kappa(t)} )  \big) \dd t 
    + \sigma (\kappa(t),Y_{\kappa(t)}^{i,\star,N},\hm^{Y,N}_{\kappa(t)} ) \dd W_t^i,
    \\
    \nonumber
    \quad 
   \hm^{Y,N}_n(\dd x):&= \frac1N \sum_{j=1}^N \delta_{Y_{n}^{j,\star,N}}(\dd x),~\quad
    \kappa(t)=\sup\big\{t_n: t_n\le t,\ n\in \llbracket 0,M-1 \rrbracket \big\}\nonumber,~\quad
     \hm^{Y,N}_{t_n}=\hm^{Y,N}_n.
\end{align}
\end{definition}
The next result states our first strong convergence finding. It is a ``strong'' pointwise (non-path-space) convergence result that is not in the classical mean-square error form. 
\begin{theorem} [Non-path-space mean-square convergence]
\label{theorem:SSM: strong error 1}
Let Assumption \ref{Ass:Monotone Assumption} hold and choose $h$ as in \eqref{eq:h choice}.  
Let $i\in\llbracket 1,N\rrbracket$, take $X^{i,N}$ as the solution to \eqref{Eq:MV-SDE Propagation} and let $\hx^{i,N}$ be the continuous-time extension of the SSM given by \eqref{eq: scheme continous extension in SDE form}.  
If $m\ge 4q+4 >\max\{2(q+1),4\}$, where $X_0^i\in L^m_0(\bR^d)$ and $q$ is as defined in Assumption \ref{Ass:Monotone Assumption}, then   there exists a constant $C>0$ independent of $h,N,M$ (but depending on $T$ and $m$) such that
\color{black}
\begin{align}
\label{eq:convergence theroem term 1}
     \sup_{i\in \llbracket 1,N \rrbracket} \sup_{0\le t \le T}
  \bE\big[\,  |X_{t}^{i,N}-\hx_{t}^{i,N} |^2 \big]    &   \le C h.
\end{align} 
\end{theorem} 
The proof is presented in    Section \ref{subsection:Proof of strong error 1}.  
This result does not need $L^p$-moment bounds of the scheme for $p>2$. It needs \textit{only}  $L^p$-moments of the solution process of \eqref{Eq:MV-SDE Propagation} and $L^2$-moments for the scheme \cite{2015ssmBandC}. The proof takes advantage of the elegant structure induced by the SSM where Proposition \ref{prop:yi-yj leq xi-xj} and \ref{prop:sum y square leq sum x square} are the crucial intermediate results to deal with the convolution term.

The next moment bound result is necessary for the subsequent uniform convergence result. 
\begin{theorem}[Moment bounds]
\label{theorem：moment bound for the big theorm time extensions }
Let the setting of Theorem \ref{theorem:SSM: strong error 1} hold. 
Let $m\geq 2$ where $X^i_0\in L^m_0(\bR^d)$ for all $i\in\llbracket 1,N\rrbracket$ and let $\hx^{i,N}$ be the continuous-time extension of the SSM given by \eqref{eq: scheme continous extension in SDE form}. Let $2p\in [2,m]$, then there exists a constant $C>0$ independent of $h,N,M$   (but depending on $T$ and $m$)  such that  
\begin{align} 
\label{eq:SSTM:momentbound for split-step time extension process00}
         \sup_{i\in \llbracket 1,N \rrbracket} \sup_{0\le t \le T}\bE\big[ \,|\hx_{t}^{i,N}|^{2p}\big]\le C\big(  1+ \bE\big[\, |\hx_{0}^\cdot|^{2p}\big] \big) <\infty.
\end{align}
\end{theorem}
The proof is presented in Section \ref{subsection: Moment bound of the SSM} and builds around auxiliary Theorem \ref{theorem: discrete moment bound}. There, we expand  \eqref{eq:mmb:def of HXp} and \eqref{eq:mmb:def of HYp}, and leverage the properties of the SSM scheme stated in Proposition \ref{prop:yi-yj leq xi-xj} and \ref{prop:sum y square leq sum x square} to deal with the difficult convolution terms. 

Next we state the classic mean-square error convergence result.
\begin{theorem}[Classical path-space mean-square convergence]
\label{theorem:SSM: strong error 2}
Let the setting of Theorem \ref{theorem:SSM: strong error 1} hold. 
Assume there exists some $\varepsilon\in (0,1)$ such that $m \ge \max\{4q+4,2+q+q/\varepsilon \} >\max\{2(q+1),4\}$ with $X_0^i\in L^m_0(\bR^d)$ for $i\in\llbracket 1,N\rrbracket$ and $q$ given as in Assumption \ref{Ass:Monotone Assumption}. 
 Then there exists a constant $C>0$ independent of $h,N,M$ (but depending on $T$ and $m$) such that 
\begin{align}
    \label{eq:convergence theroem term 2}
  \sup_{i\in \llbracket 1,N \rrbracket} 
  \bE\big[ \sup_{0\le t \le T} |X_{t}^{i,N}-\hx_{t}^{i,N} |^2 \big]
  &\le Ch^{1-\varepsilon}.
\end{align}
\end{theorem}  
The proof is presented in Section \ref{subsection:Proof of strong error 2}. 
For this result we need both the $L^p$-moments of the scheme and solution process. 
This in contrast to the proof methodology of Theorem \ref{theorem:SSM: strong error 1} and the reason we introduce Theorem \ref{theorem：moment bound for the big theorm time extensions } as a main result. 
The {nearly optimal} error rate of $(1-\varepsilon)$ is a consequence of the estimation of \eqref{eq:se2:special term} (product of three unbounded random variables). 
The expectation is taken after the supremum and then we use Theorem \ref{theorem:SSM: strong error 1} and \ref{theorem：moment bound for the big theorm time extensions } -- this forces an $\varepsilon$ sacrifice of the rate. 
The {nearly optimal} error rate of $(1-\varepsilon)$ is also the present best one available even for higher-order differences $p>2$ (although we do not present these calculations). 
It is still open how to prove \eqref{eq:SSTM:momentbound for split-step time extension process00} with the $\sup_t$ inside the expectation --- the difficulty to be overcome relates to establishing \eqref{eq:sum y square leq sum x square} of Proposition \ref{prop:sum y square leq sum x square} under higher moments $p>2$ in a way that aligns with \textit{carr\'e-du-champs} type arguments and the convolution term (within the style of proof we provide, otherwise new arguments need be found). 
It remains an open problem to show \eqref{eq:convergence theroem term 2} when $\varepsilon=0$.

\subsubsection*{A particular result for granular media equation type models}

We recast the earlier results to granular media type models where the diffusion coefficient is constant and higher convergence rates can be established.
\begin{assumption}
\label{Ass:GM Assumption}
Consider the following MV-SDE 
\begin{align}
\label{Eq:GM type MVSDE}
\dd X_{t} &=  v(X_{t},\mu_{t}^{X}) \dd t + \sigma \dd W_{t},  \quad X_{0} \in L_{0}^{m}( \bR^{d}),\quad 
v(x,\mu)= \int_{\bR^{d}  } f(x-y) \mu(\dd y).
\end{align}
Let $f:\bR^d \to \bR^d$ be continuously differentiable satisfying  ($\mathbf{A}^f$) of Assumption \ref{Ass:Monotone Assumption}. There exist $ L_{f'},~L_{f''} >0$, $q\in \bN$ and $q>1$, with $q$ the same as in ($\mathbf{A}^f$), such that for all  $ x,~x'\in \bR^d$ 
\begin{align}
\label{eq:ass:gm:f 1}
    | \nabla f (x)| 
& 
\leq L_{f'}(1+ |x|^{q} )  
,\quad
    | \nabla f  (x)  
    -\nabla f  (x') 
    | 
\leq L_{f''}(1+ |x|^{q-1}+|x'|^{q-1} )|x-x'|  .
\end{align}
The function  $\sigma:[0,T] \times \bR^d \times \cP_2(\bR^d) \to \bR^{d\times l}$ is a constant matrix. 
\end{assumption}
In the language of the granular media equation, MV-SDE \eqref{eq:ass:gm:f 1} corresponds to the Fokker-Plank PDE $\partial_t \rho =\nabla \cdot[\nabla \rho+\rho \nabla W * \rho]$ where $\nabla W=f$ and $\rho$ is the probability measure \cite{Malrieu2003}. We have the following results.
\begin{theorem}
\label{theorem：gm strong error}
Let Assumption \ref{Ass:GM Assumption} hold and choose $h$ as in \eqref{eq:h choice}. Let $i\in\llbracket 1,N\rrbracket$, take $X^{i,N}$ to be the solution to \eqref{Eq:MV-SDE Propagation}, let $\hx^{i,N}$ be the continuous-time extension of the SSM given by \eqref{eq: scheme continous extension in SDE form} and $X^i_0\in L^m_0(\bR^d)$. Let $m\ge \max\{8q,~4q+4\}>\max\{2(q+1),4\}$ with $q$ as defined in Assumption \ref{Ass:GM Assumption}. 
Then   there exist a constant $C>0$ independent of $h,N,M$ (but depending on $T$ and $m$) such that 
\begin{align}
\label{eq:convergence theroem term 11 constant diffusion}
     \sup_{i\in \llbracket 1,N \rrbracket} \sup_{0\le t \le T}
  \bE\big[  |X_{t}^{i,N}-\hx_{t}^{i,N} |^2 \big]    &   \le Ch^2.
\end{align}
\end{theorem}

This result is proved in Section \ref{subsection:Proof of strong error 3 GM}. Supporting simulation results are presented in Section \ref{exam:StochGinzburg} and confirm the strong root mean square error rate of $1.0$. 

We note that one can use a proof methodology similar to that used for Theorem \ref{theorem:SSM: strong error 2} to obtain \eqref{eq:convergence theroem term 11 constant diffusion} with the $\sup_t$ inside the expectation. This would deliver a rate of $h^{2-\varepsilon}$, the key steps are similar to \eqref{eq:proof strong errr2:ep 1}-\eqref{eq:proof strong errr2:ep 2}. 
\color{black}

\section{Examples of interest}
\label{sec:examples}
 We illustrate the SSM on three numerical examples.\footnote{Implementation code in Python is available in 
 \href{url}{https://github.com/AnandaChen/Simulation-of-super-measure}
 } 
 The ``true'' solution in each case is unknown and the convergence rates for these examples are calculated in reference to a proxy solution given by the approximating scheme at a smaller timestep $h$ and higher number of particles $N$ (particular details are given below). The strong error between the proxy-true solution $X_T$ and approximation $\hat X_T$ is as follows
\begin{align*}
\textrm{root Mean-square error (Strong error)} 
= \Big( \bE\big[\, |X_T-\hat{X}_T|^2\big]\Big)^{\frac12}
\approx \Big(\frac1N \sum_{j=1}^N |X_T^j - \hat{X}_T^j|^2\Big)^\frac12.
\end{align*}

We also consider the path strong error define as follows, for $Mh=T~,t_n=nh$,
\begin{align*}
\textrm{  Strong error (Path) } 
= \Big( \bE\big[\,\sup_{ 0\leq t \leq T } |X_t-\hat{X}_t|^2\big]\Big)^{\frac12}
\approx \Big(\frac1N \sum_{j=1}^N \sup_{n\in \llbracket 0,M \rrbracket }|X_{t_n}^j - \hat{X}_{t_n}^j|^2\Big)^\frac12.
\end{align*}
\color{black}
The propagation of chaos (PoC) rate between different particle systems $\{\hx_T^{i,N_l}\}_{i,l}$ where $i$ denotes the $i$-th particle and $N_l$ denotes the size of the system, 
\begin{align}
\nonumber
\textrm{Propagation of chaos error (PoC error)} 
\approx \Big(\frac{1}{N_l} \sum_{j=1}^{N_l} |\hx_T^{j,N_l} - X_T^{j} |^2\Big)^\frac12.
\end{align}
\begin{remark}[`Taming' algorithm]
For comparative purposes we implement the `\textit{Taming}' algorithm \cite{2021SSM,reis2018simulation} -- any convergence analysis of the taming algorithm to the framework of this manuscript is an open question. 
Of the many variants of Taming possible, set the  terminal time $T$ with $Mh=T$, we implement as follows: $\int_{\bR^{d}  } f(\cdot-y) \mu(\dd y)$ is replaced by 
        $\int_{\bR^{d}  } f(\cdot-y) \mu(\dd y)/(1+M^{-\alpha}|\int_{\bR^{d}  } f(\cdot-y) \mu(\dd y)|)$, and $u$ is replaced by $u/(1+M^{-\alpha}|u|)$ with the { choice of $\alpha=1/2$ for non-constant diffusion and $\alpha=1$ for constant diffusion. }
\end{remark}
Within each example, the error rates of Taming and SSM are computed using the same Brownian motion paths.

Moreover, for the simulation study below, we fix the algorithmic parameters as follows:
\begin{enumerate}
    \item For the strong error, the proxy-true solution is calculated with $h=10^{-4}$ and the approximations are calculated with $h\in\{10^{-3},2\times10^{-3},\dots,10^{-1} \}$ with $N=1000$ at $T=1$ and using the same Brownian motion paths. We compare SSM and Taming with the proxy-true solutions provided by the same algorithm (SSM and Taming) respectively. 
    
     \item For the PoC error, the proxy-true solution is calculated with $N=2560$ and the approximations are calculated with $N\in\{40,80,\dots,1280\}$, with $h=0.001$ at $T=1$ and using the same Brownian motion paths. 
     
     \item The implicit step \eqref{eq:SSTM:scheme 0} of the SSM algorithm is solved, in our examples, via a Newton method iteration. We point the reader to Appendix \ref{appendix: discussion on Newton's method} for a full discussion. In practice, $2$ to $4$ Newton iterations are sufficient to ensure that the difference between two consecutive Newton iterates are not larger than $\sqrt{h}$ in $\|\cdot\|_\infty$-norm (in $\bR^{Nd}$). 
\end{enumerate}
Lastly, the symbols $\cN(\alpha,\beta)$ denote the normal distribution with mean $\alpha\in \bR$ and variance $\beta\in (0,\infty)$.
\color{black}

\subsection{Example: the granular media equation}
\label{exam:StochGinzburg}
The first example  is the granular media Fokker-Plank equation taking the form $\partial_t \rho =\nabla \cdot[\nabla \rho+\rho \nabla W * \rho]$ with $W(x)=\frac13|x|^3$  and $\rho$ is the correspondent probability density \cite{Malrieu2003,cattiaux2008probabilistic}. In MV-SDE form we have
\begin{align}
\label{eq:example:granular}
\dd X_{t} &= v(X_{t},\mu_{t}^{X}) \dd t + \sqrt{2}~\dd W_{t},  \quad X_{0} \in L_{0}^{m}( \bR^{d}), \quad 
v(x,\mu)= \int_{\bR^{d}  } \Big(-\sign(x-y)|x-y|^2  \Big)  \mu(\dd y),
\end{align}
where $\sign(\cdot)$ is the standard sign function, $\mu_t^X$ is the law of the solution process $X$ at time $t$.
\begin{figure}[h!bt]
    \centering
    \begin{subfigure}{.48\textwidth}
    \setlength{\abovecaptionskip}{-0.02cm}
     \setlength{\belowcaptionskip}{-0.05 cm}  
			\centering
 			\includegraphics[scale=0.25]{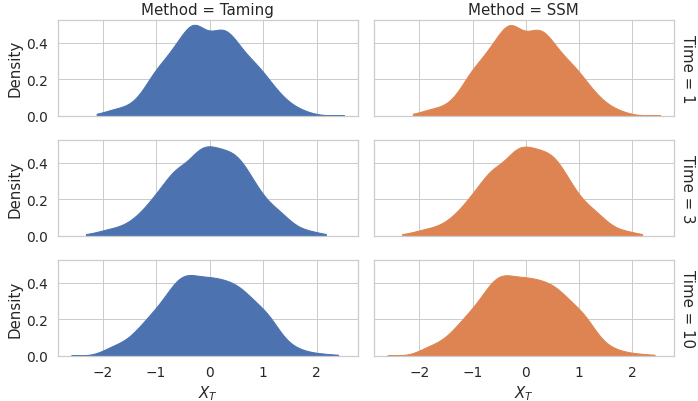}
			\caption{Density with $X_0\sim \cN(0,1)$}
			
		\end{subfigure}%
        \begin{subfigure}{.48\textwidth}
           \setlength{\abovecaptionskip}{-0.02cm}
           \setlength{\belowcaptionskip}{-0.05 cm} 
			\centering
 			\includegraphics[scale=0.25]{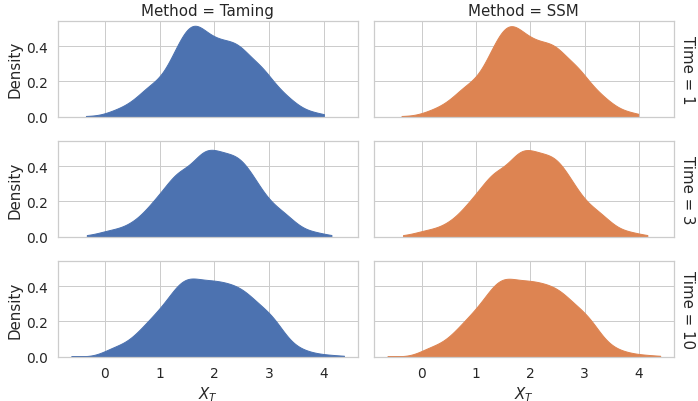}
			\caption{Density with $X_0\sim \cN(2,16)$}
		\end{subfigure}%
						\\
		\begin{subfigure}{.24\textwidth}
		\setlength{\abovecaptionskip}{-0.02cm}
          \setlength{\belowcaptionskip}{-0.2 cm} 
			\centering
 			\includegraphics[scale=0.23]{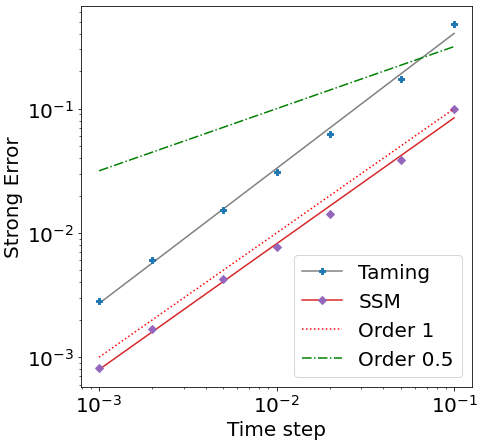}
			\caption{Strong Error}
		\end{subfigure}
		\begin{subfigure}{.24\textwidth}
		\setlength{\abovecaptionskip}{-0.02cm}
		\setlength{\belowcaptionskip}{-0.2 cm}
			\centering
 			\includegraphics[scale=0.23]{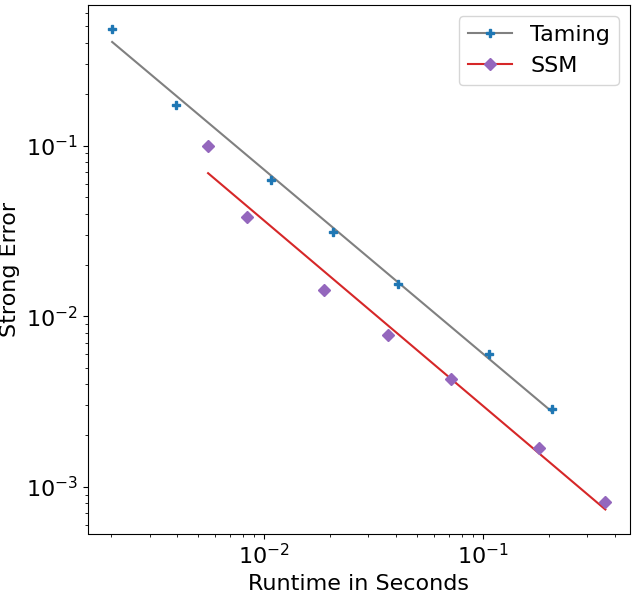}
			\caption{Strong error v.s Runtime }
		\end{subfigure}
		\begin{subfigure}{.24\textwidth}
		\setlength{\abovecaptionskip}{-0.02cm}
		\setlength{\belowcaptionskip}{-0.2 cm}
			\centering
 			\includegraphics[scale=0.23]{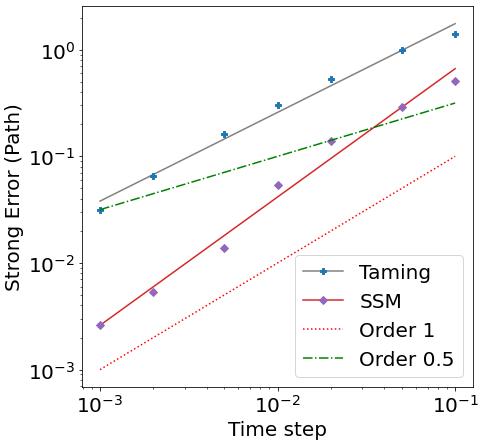}
			\caption{Strong error (Path) }
		\end{subfigure}
		\begin{subfigure}{.24\textwidth}
		\setlength{\abovecaptionskip}{-0.02cm}
		\setlength{\belowcaptionskip}{-0.2 cm}
			\centering
 			\includegraphics[scale=0.23]{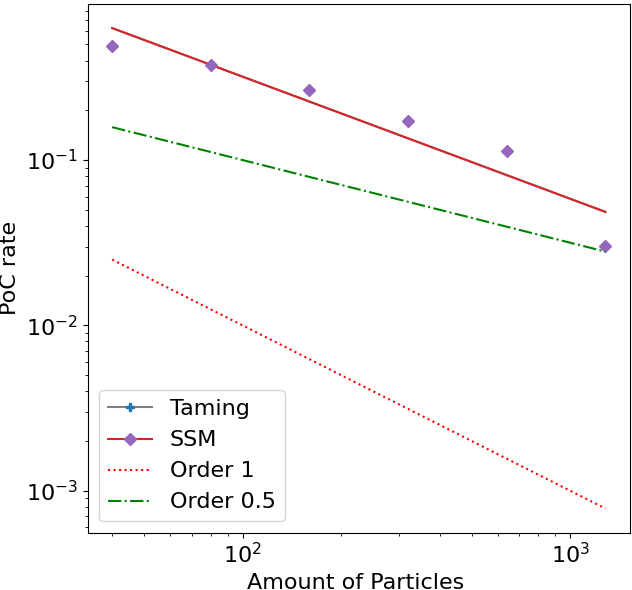}
			\caption{PoC Error }
		\end{subfigure}
	\setlength{\belowcaptionskip}{-0.2cm} 
    \caption{Simulation of the granular media equation \eqref{eq:example:granular} with $N=1000$ particles. (a) and (b) show the density map for Taming (blue, left) and SSM (orange, right) with $h=0.01$ at times $T=1,3,10$ seen top-to-bottom and with different initial distribution. (c) Strong error (rMSE) of SSM and Taming with $X_0 \sim \cN(2,16)$. (d)  Strong error (rMSE) of SSM and Taming w.r.t algorithm with $X_0 \sim \cN(2,16)$.(e) Strong error (Path) of SSM and Taming with $X_0 \sim \cN(2,16)$. (f) PoC error rate in $N$ of SSM and Taming with $X_0 \sim \cN(2,9)$ with perfect overlap of errors. }
    \label{fig:1-2:granular example}
\end{figure}
This granular media model has been well studied in \cite{Malrieu2003,cattiaux2008probabilistic} and is a reference model to showcase the numerical approximation. For this specific case, starting from a normal distribution, the particles concentrate and move around its initial mean value (also its steady state). 
In Figure \ref{fig:1-2:granular example} (a) and (b) one sees the evolution of the density map across time $T\in\{1,3,10\}$ for two initial initial distributions  $\mathcal{N}(0,1)$ and $\mathcal{N}(2,4)$ respectively, and $h=0.01$. For this case, both methods approximate well the solution without any apparent leading difference between Taming and SSM.  

Figure \ref{fig:1-2:granular example} (c)  shows  strong error of both methods, computed at $T=1$ across $h\in  \{ 10^{-3},2\times10^{-3},\dots, 10^{-1} \}$. The proxy-true solution for each method is taken at $h=10^{-4}$ and the baseline slopes for the ``order 1'' and ``order 0.5'' convergence rate are provided for comparison. 
\color{black}
The estimated rate of both method is  $1.0$ in accordance to Theorem \ref{theorem：gm strong error} (under constant diffusion coefficient).    
Figure \ref{fig:1-2:granular example} (d)  shows  strong error v.s algorithm runtime of both methods  under the same set up as in (c). The SSM perform slightly better than the Taming method. 

Figure \ref{fig:1-2:granular example} (e) shows the path type strong error of both method, compare to the results in (c), the SSM preserve the error rate of near $1.0$ and perform better than the Taming method.
\color{black}
Figure \ref{fig:1-2:granular example} (f) shows the PoC error of both methods. The two results coincide since the differences between two methods are within $0.001$. The PoC rates are near $0.5$ which is better than the theoretical result of $1/4$ after we take square root in Proposition \ref{Prop:Propagation of Chaos}. This result is similar to  \cite[Example 4.1]{reisinger2020adaptive}, and is explained theoretically by \cite[Lemma 5.1]{delarue2018masterCLT} but under stronger assumptions than ours.

\subsection{Example: Double-well model}
\label{section:exam:toy-toy2}
We consider a limit model of particles under a symmetric double-well confinement. We test a variant of the model studied in \cite{2013doublewell} but change its diffusion coefficient to a non-constant one (in opposition to the previous example). Concretely, we study the following McKean-Vlasov equation  
\begin{align}
\label{eq:example:toy2}
\dd X_{t} &= \big(  v(X_{t},\mu_{t}^{X})+X_{t} \big) \dd t + X_t \dd W_{t}, \quad 
v(x,\mu)= -\frac14 x^3+\int_{\bR^{d}  } -\big(x-y  \big)^3  \mu(\dd y).
\end{align}
The corresponding Fokker-Plank equation is 
    $ \partial_t \rho=\nabla \cdot [~\nabla ( \frac{\rho|x|^2}{2})+\rho \nabla V+\rho \nabla W * \rho ]$ 
with $W=\frac14|x|^4$, $V=\frac{1}{16}|x|^4-\frac{1}{2}|x|^2$,  $\rho$ is the corresponding density map. There are three stable states $\{-2,0,2\}$ for this model \cite{2013doublewell}.

\begin{figure}[h!bt]
    \centering
    \begin{subfigure}{.48\textwidth}
    \setlength{\abovecaptionskip}{-0.02cm}
     \setlength{\belowcaptionskip}{-0.05 cm}  
			\centering
 			\includegraphics[scale=0.25]{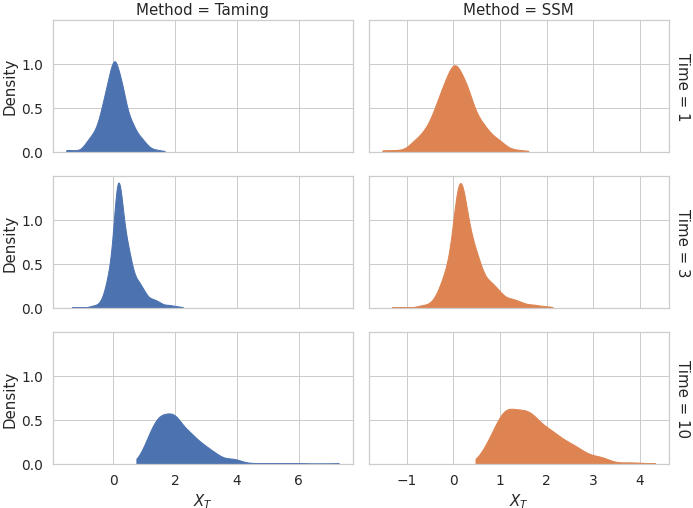}
			\caption{Density with $X_0\sim \cN(0,1)$}
			
		\end{subfigure}%
        \begin{subfigure}{.48\textwidth}
           \setlength{\abovecaptionskip}{-0.02cm}
           \setlength{\belowcaptionskip}{-0.05 cm} 
			\centering
 			\includegraphics[scale=0.25]{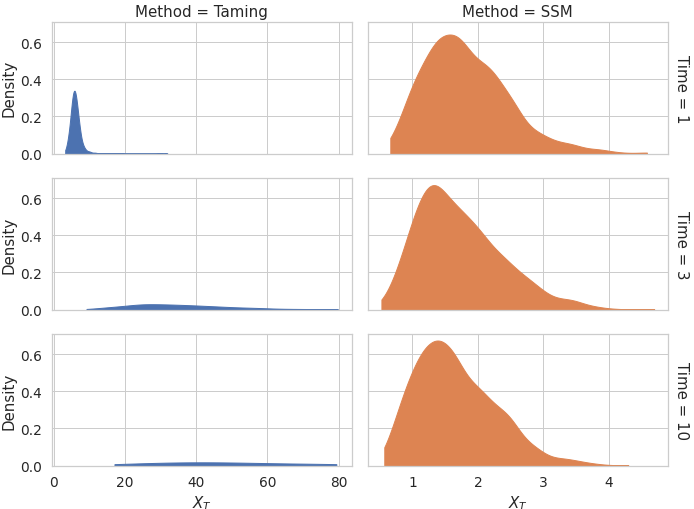}
			\caption{Density with $X_0\sim \cN(3,9)$}
			
		\end{subfigure}%
						\\
		\begin{subfigure}{.45\textwidth}
		\setlength{\abovecaptionskip}{-0.02cm}
          \setlength{\belowcaptionskip}{-0.03 cm} 
			\centering
 			\includegraphics[scale=0.4]{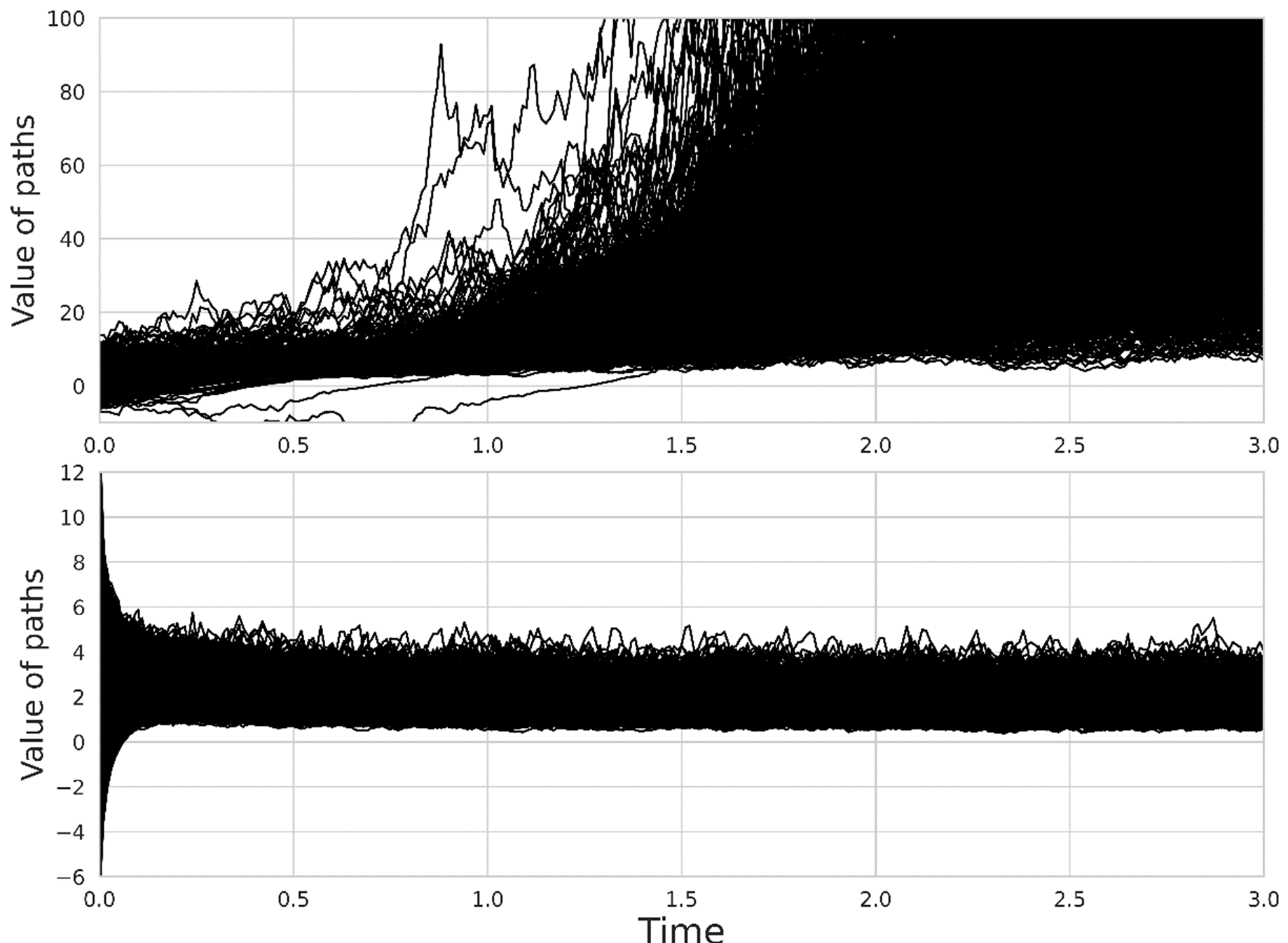}
			\caption{Simulated paths of Taming (top) and SSM (bottom)}
		\end{subfigure}
		\begin{subfigure}{.26\textwidth}
		\setlength{\abovecaptionskip}{-0.02cm}
		\setlength{\belowcaptionskip}{-0.2 cm}
			\centering
 			\includegraphics[scale=0.25]{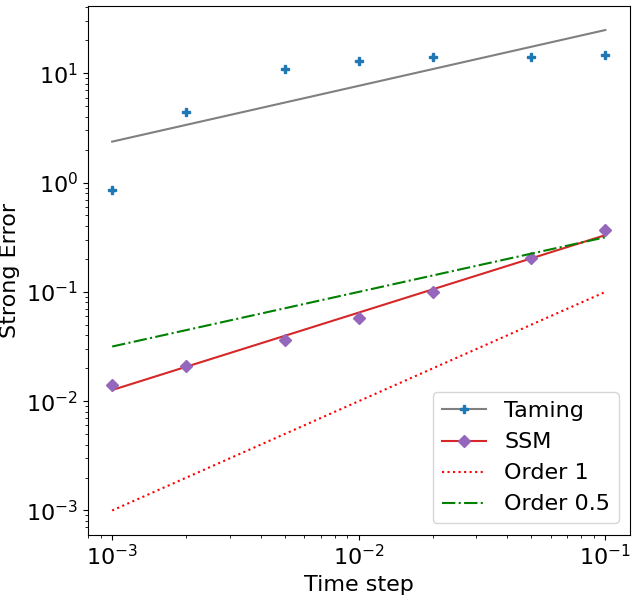}
			\caption{Strong error }
		\end{subfigure}
		\begin{subfigure}{.26\textwidth}
		\setlength{\abovecaptionskip}{-0.02cm}
		\setlength{\belowcaptionskip}{-0.2 cm}
			\centering
 			\includegraphics[scale=0.25]{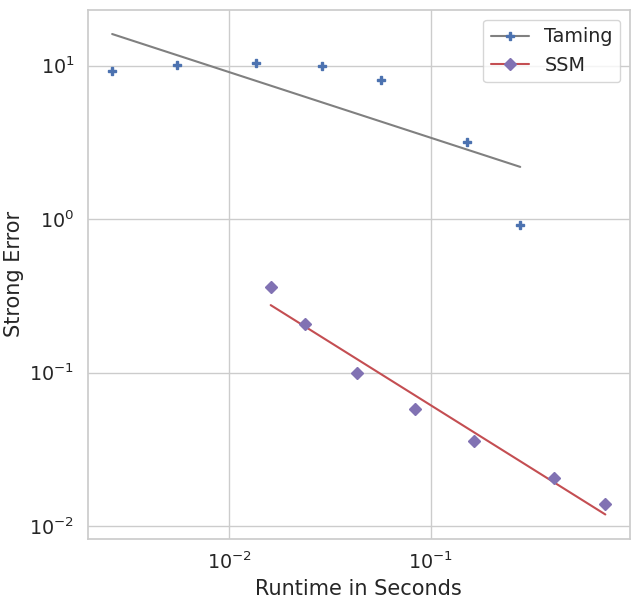}
			\caption{Strong error v.s Runtime }
		\end{subfigure}
	\setlength{\belowcaptionskip}{-0.2cm} 
    \caption{Simulation of the Double-Well model  \eqref{eq:example:toy2} with $N=1000$ particles. 
     (a) and (b) show the density map for Taming (blue, left) and SSM (orange, right) with $h=0.01$  at times $T=1,3,10$ seen top-to-bottom and with different initial distribution. 
     (c)  simulated paths by Taming (top) and SSM (bottom) with $h=0.01$ over  $t\in[0,3]$ and with $X_0\sim \cN(3,9)$.
     (d)  Strong error (rMSE) of SSM and Taming with $X_0 \sim \cN(2,4)$. 
     (e)  Strong error (rMSE) of SSM and Taming w.r.t algorithm Runtime with $X_0 \sim \cN(2,4)$.
            }
    \label{fig:3:the toy example3}
\end{figure}

The example of Section \ref{exam:StochGinzburg} was a relatively mild with additive noise and where both methods performed well. For this double-well model of \eqref{eq:example:toy2}, the drift includes super-linear growth components in both space and measure and a non-constant unbounded diffusion coefficient.  

In Figure \ref{fig:3:the toy example3} (a) and (b), Taming (blue, left) fails to produce acceptable results of any type -- Figure \ref{fig:3:the toy example3} (c) shows the simulated paths of both methods where it is noteworthy to see that Taming become unstable while the SSM paths remain stable. In respect to Figure \ref{fig:3:the toy example3} (a) and (b), the SSM (orange, right) depicts the distribution's evolution to one of the expected stable states ($x=2$) as time evolves. It is interesting to find out that for the SSM in (a), where $X_0\sim \mathcal{N}(0,1)$, the particles shift from the zero (unstable) steady state to the positive stable steady state $x=2$. However, in (b) with $X_0 \sim \mathcal{N}(3,9)$, we find that the particles remain within the basin of attraction of the stable state $x=2$.   
Figure \ref{fig:3:the toy example3} (d) displays under the same parameter choice for $h,~T$ as for the granular media example of Section \ref{exam:StochGinzburg} with $X_0\sim \cN(2,4)$ the estimated rate of convergence for the schemes. It  shows the taming method fails to converge (but does not explode). The strong error rate of the SSM is the expected $1/2$ in-line with Theorem \ref{theorem:SSM: strong error 1} (and Theorem \ref{theorem:SSM: strong error 2}).  

The ``order 1'' and ``order 0.5'' lines are baselines corresponding to the slope of $1$ and $0.5$ rate of convergence.

Figure \ref{fig:3:the toy example3} (e) shows that, to reach the same strong error level Taming shall takes far more (over 100 times) runtime than the SSM.

\subsection{Example: 2d  Van der Pol (VdP) oscillator}
\label{section:example:vdp}

We consider the  Van der Pol (VdP) model described in 
 \cite[Section 4.2 and 4.3]{HutzenthalerJentzen2015AMSbook}, with added super-linearity in measure and non-constant unbounded diffusion.  
 We study the following MV-SDE dynamics:  set $x=(x_1,x_2)\in\bR^2$, for \eqref{Eq:General MVSDE} define the functions $f,u,b,\sigma$ as 
\begin{align}
\label{eq:example:vdp}
f(x)=-x|x|^2
,
\quad
u(x)=
\left[\begin{array}{c}
-\frac43 x_1^3 \\
0
\end{array}\right],
\quad
 b(x)=
\left[\begin{array}{c}
4(x_1-x_2) \\
\frac{1}{4}x_1
\end{array}\right],
\quad
\sigma (x)=
\left[\begin{array}{ccc}
x_1 & 0 \\
0 &  x_2
\end{array}\right],
\end{align}
which satisfy the assumptions of this work.

Figure \ref{fig:1-3:vdp example}  (a)  shows  the strong error of both methods, the ``order 1'' and ``order 0.5'' lines are baselines with the slope of $1$ and $0.5$ for comparison. The estimated rate of the SSM is near $0.5$ while Taming failed to converge. 
Figure \ref{fig:1-3:vdp example} (b) shows the PoC error of both methods, Taming failed to converge while the estimated rate of the SSM is near $0.5$ (see discussion of previous Section \ref{exam:StochGinzburg}). 

Figure \ref{fig:1-3:vdp example} (c)  shows the system's phase-space portraits (i.e., the parametric plot of $t\mapsto (X_{1,t},X_{2,t})$ and $t\mapsto (\bE[X_{1,t}],\bE[X_{2,t}])$ over $t\in [0,20]$) of the SSM with respect to different choices of $N\in\{30,100,500,1000\}$. The impact of $N$ on the quality of simulation is apparent as is the ability of the SSM to capture the periodic behaviour of the true dynamics. 
Figure \ref{fig:1-3:vdp example}  (d)-(e)-(f)-(g) shows the expectation's fluctuation (of Figure \ref{fig:1-3:vdp example} (c)) and the system's phase-space path portraits of the SSM for different choices of $N$. The trajectory becomes smoother as $N$ becomes larger and the  paths are similar for  $N\ge 500$.

\begin{figure}[h!bt]
    \centering
		\begin{subfigure}{.26\textwidth}
		    \setlength{\abovecaptionskip}{-0.04cm}
     \setlength{\belowcaptionskip}{-0.1 cm}  
			\centering
 			\includegraphics[scale=0.26]{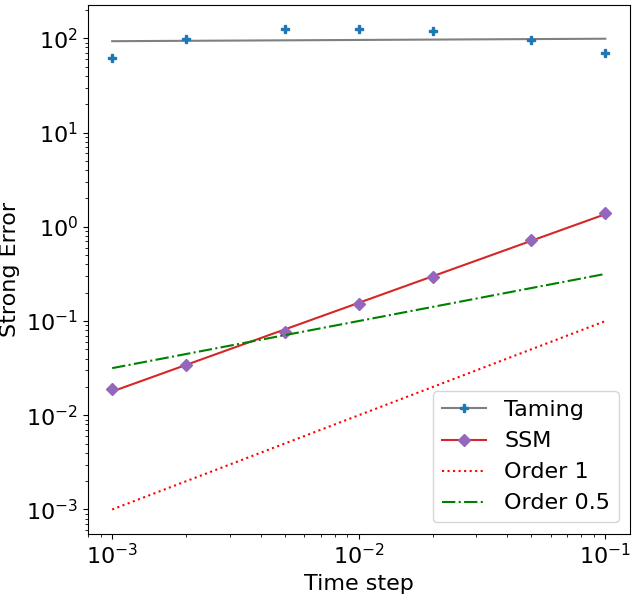}
			\caption{Strong error}
			\label{fig:13-2}
		\end{subfigure}
		\begin{subfigure}{.26\textwidth}
		    \setlength{\abovecaptionskip}{0.05cm}
            \setlength{\belowcaptionskip}{-0.05 cm}  
			\centering
			\includegraphics[scale=0.25]{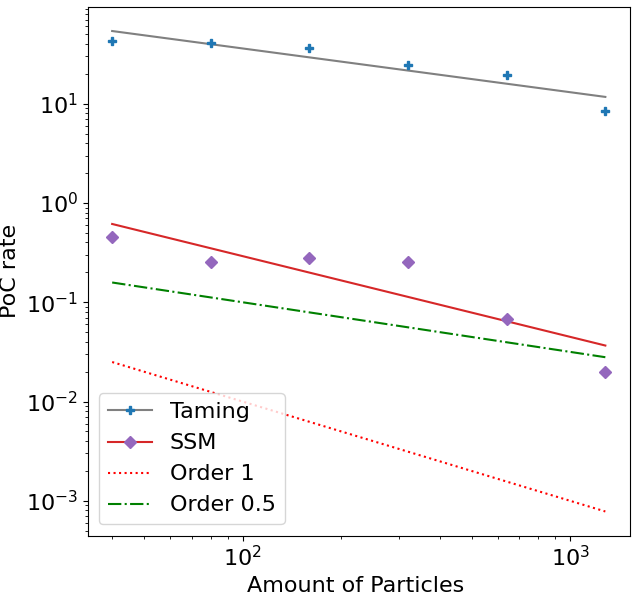}
			\caption{PoC error}
 			\label{fig:13-3}
		\end{subfigure}
		 \begin{subfigure}{.46\textwidth}
            \setlength{\abovecaptionskip}{-0.04cm}
     \setlength{\belowcaptionskip}{-0.1 cm}  
			\centering
 			\includegraphics[scale=0.46]{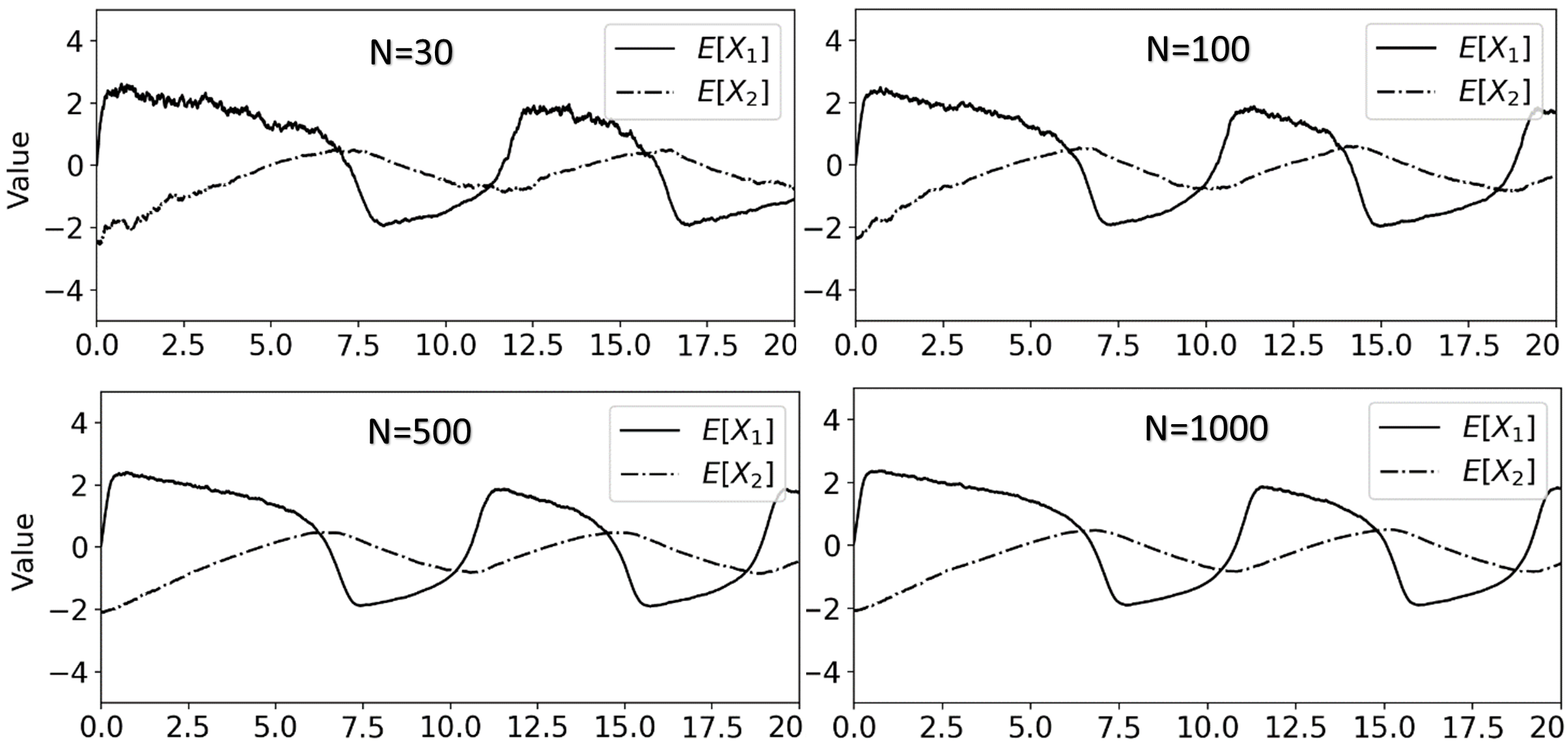}
			\caption{Expectation paths of SSM w.r.t $N$}
			\label{fig:13-1}
		\end{subfigure}%
		\\
		 \centering

		\begin{subfigure}{.24\textwidth}
		    \setlength{\abovecaptionskip}{-0.04cm}
     \setlength{\belowcaptionskip}{-0.1 cm}  
			\centering
 			\includegraphics[scale=0.58]{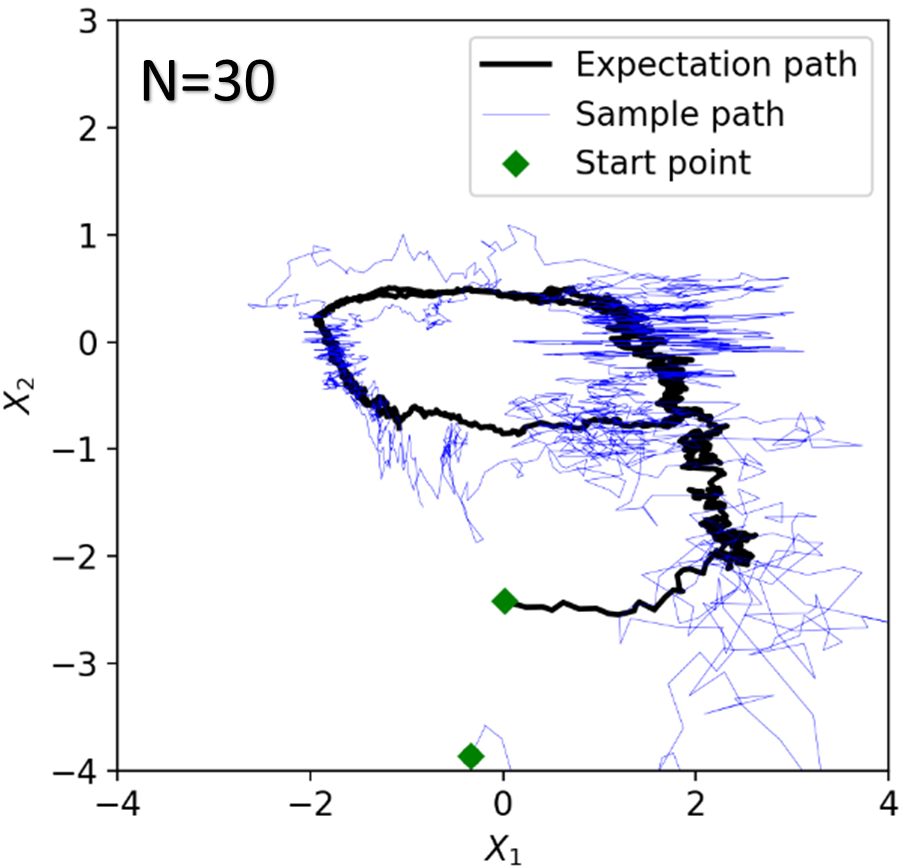}
			\caption{ $N=30$}
		\end{subfigure}
        		\begin{subfigure}{.24\textwidth}
		    \setlength{\abovecaptionskip}{-0.04cm}
     \setlength{\belowcaptionskip}{-0.1 cm}  
			\centering
 			\includegraphics[scale=0.58]{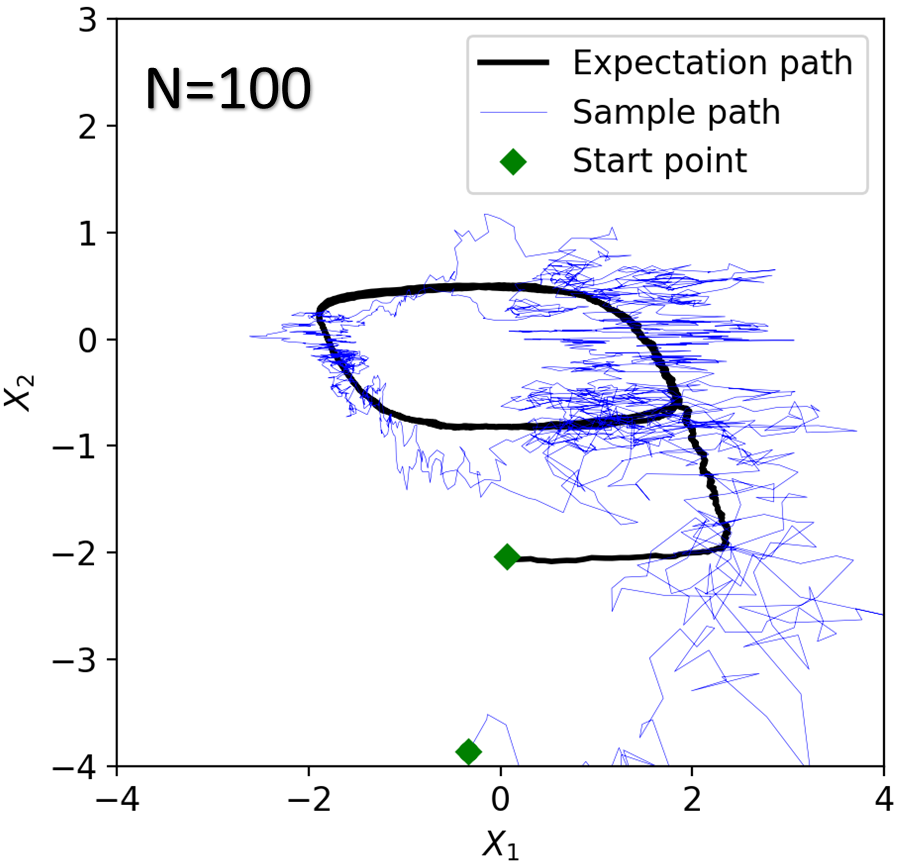}
			\caption{  $N=100$}
		\end{subfigure}		\begin{subfigure}{.24\textwidth}
		    \setlength{\abovecaptionskip}{-0.04cm}
     \setlength{\belowcaptionskip}{-0.1 cm}  
			\centering
 			\includegraphics[scale=0.58]{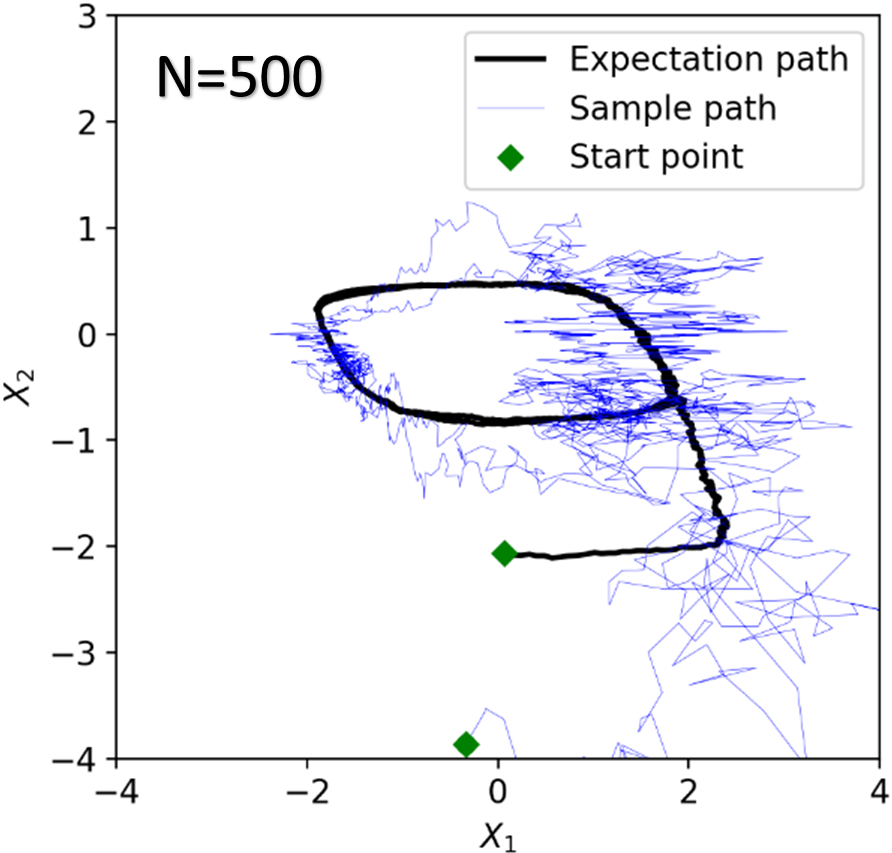}
			\caption{  $N=500$}
		\end{subfigure}		\begin{subfigure}{.24\textwidth}
		    \setlength{\abovecaptionskip}{-0.04cm}
     \setlength{\belowcaptionskip}{-0.1 cm}  
			\centering
 			\includegraphics[scale=0.58]{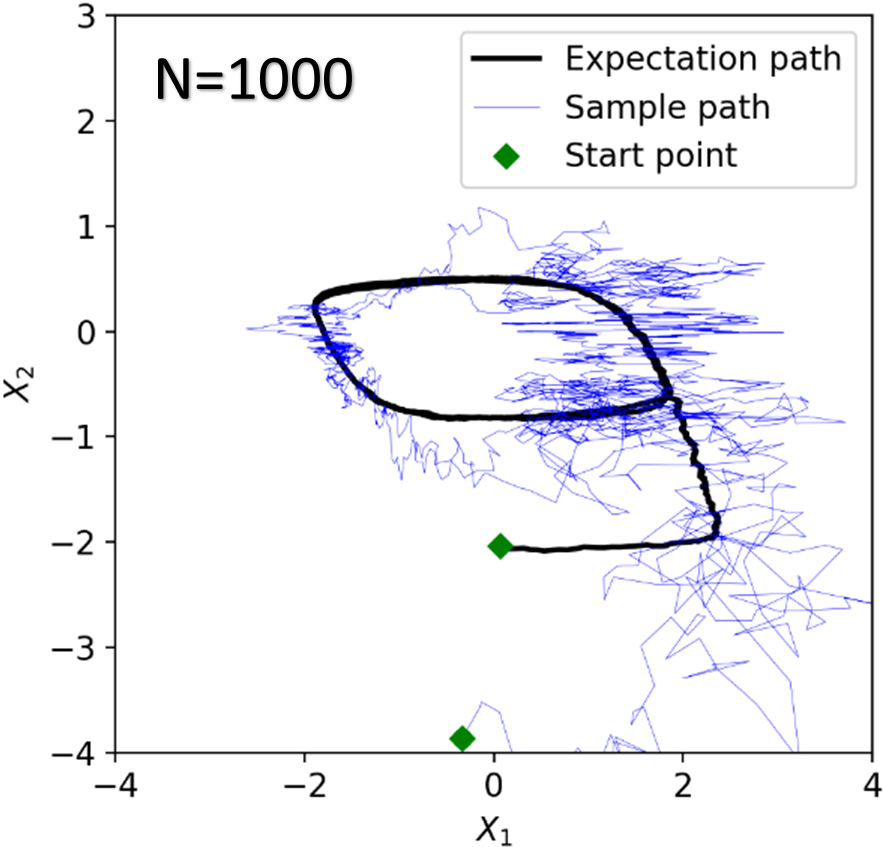}
			\caption{  $N=1000$}
		\end{subfigure}
	\setlength{\belowcaptionskip}{-0.2cm} 
    \caption{Simulation of the Vdp model \eqref{eq:example:vdp} with $X_1\sim \cN(0,4), X_2\sim \cN(-2,4)$.   (a) Strong error (rMSE) of the SSM and Taming with $T=1,~N=1000$. (b) PoC error of the SSM and Taming with $T=1,~h=0.001$.
    (c) the expectation overlays paths for the SSM with $T=20,h=0.01$ w.r.t different $N$. 
    (d)-(e)-(f)-(g) the corresponding phase-space portraits in (c) with  $N\in\{ 30,100,500,1000\} $.
    }
    \label{fig:1-3:vdp example}
\end{figure}

\subsection{Numerical complexity, discussion and various opens questions}
\label{sec:discussionOfNumerics}

Across the three examples the SSM converged and all examples recovered the theoretical convergence rate (of $1/2$ in general, and $1$ for the additive noise case). In the latter two examples, Taming failed to converge while on the first example the SSM and taming are mostly similar. The main difference between examples is the diffusion coefficient.

The SSM is robust in respect to small choices of $h$ and $N$. In all three examples, the SSM remains convergent for all choices of $h$ (even for $h=0.1$) while taming fails to converge at all. In the Van der Pol (VdP) oscillator example of Section \ref{section:example:vdp}, when comparing across different particle sizes $N$, the SSM provides a good approximation for all choices of $N$ (even for $N=30$) and the PoC result is as expected. In general, we found that the runtime of the SSM is nearly the double of Taming for the same choices of $h$, but on the other hand, Taming takes over 100-times more runtime to reach the same accuracy as the SSM (if one considers the strong error against runtime).

\subsubsection*{Computational costs and open questions for future research}

In the context of \eqref{Eq:MV-SDE Propagation}, assume one wants to simulate an $N$-particle system over a discretised finite time-domain with $M$ time points. 
Since we deal with  convolution type operator, the interaction term need to be computed for every single particle and thus, a standard explicit Euler scheme incurs a computational cost of $\cO(N^2M)$. Without the convolution component, the cost is simply $\cO(N M)$. 
For the SSM scheme in Definition \ref{def:definition of the ssm}, since it is has an implicit component there is an additional cost attached to it (more below). 

At this level, two strategies can be thought to reduce the complexity. 
The first is by controlling the cost of computing the interaction itself, these have been proposed for example in the projected particle method \cite{belomestny2018projected} or the Random Batch Method (RBM) \cite{jin2020rbm1}. To date there is no general proof of these outside Lipschitz conditions (and constant diffusion coefficient in the RBM case) for the efficacy of the method, also, it is not clear how to use these methods in combination with Newton to solve the SSM's implicit equation (more below). 
The second is to better address the competition between the number of particles $N$, as dictated by the PoC result Proposition \ref{Prop:Propagation of Chaos}, and the time-step parameter $M$ (or $1/h$). Our experimental work estimating the Propagation of chaos rate points to a convergence rate of order $1/2$ instead of the upper bound rate $1/4$ guaranteed by \eqref{eq:poc result} in Theorem \ref{Prop:Propagation of Chaos}. This result is not surprising in view of the theoretical result \cite[Lemma 5.1]{delarue2018masterCLT}; and numerically in \cite[Example 4.1]{reisinger2020adaptive}. To the best of our knowledge, no known PoC rate result covers the examples presented here and Theorem \ref{Prop:Propagation of Chaos} is presently the best known general result.

\subsubsection*{Solving the implicit step in SSM - Newton's method}

The SSM scheme contains an implicit Equation \eqref{eq:SSTM:scheme 0} that needs be solved at each timestep. It is left to the user to choose the most suitable method for given data and, in all generality, one needs an approximation scheme to solve \eqref{eq:SSTM:scheme 0}. Proposition \ref{prop: discussion of the approxi} below shows that as long as said approximation is uniformly controlled within a ball of radius $Ch$ of the true solution, then the SSM's convergence rate of Theorem \ref{theorem:SSM: strong error 1} is preserved. 

As mentioned in the initial part of Section \ref{sec:examples}, we use Newton's method (assuming extra differentiability of the involved maps) -- see Appendix \ref{appendix: discussion on Newton's method} for details where \cite[Section 4.3]{Suli2003NumericsBook} is used to guarantee convergence.  
The computation cost raises from $\cO(N^2M )$ to $\cO(\kappa N^2M)$, where $\kappa$ denotes the leading term cost of Newton after $\kappa$ iterations. In practice, we found that within $2$ to $4$ iterations (i.e., $\kappa\leq 4$) two consecutive Newton iteration are sufficiently close for the purposes of the scheme's accuracy: denoting Newton's $j^{th}$-iteration by $y^j \in \bR^{Nd}$, then $\|y^\kappa-y^{\kappa-1}\|_\infty<\sqrt{h}$ (which is the stop criteria used, see Appendix \ref{appendix: discussion on Newton's method}). 

Interacting particle systems like \eqref{Eq:MV-SDE Propagation} induce a certain structure to the associated Jacobian matrix when seen through the lens of $(\bR^{d})^N$. 
The closed form expressions provided in Appendix \ref{sec:NewtonMethodJacobian} point to a very sparse Jacobian matrix with a very specific block structure. For instance, the $\Gamma$ matrix (see Appendix \ref{sec:NewtonMethodJacobian}) is a symmetric one and is multiplied by $h/N$ making its entries very small: it stands to reason that $\Gamma$ can be removed from the Jacobian matrix as one solves the system (provided its entries can be controlled) and thus suggests that an inexact or quasi-Newton method might be computationally more efficient. 
In \cite[Section 3]{LiCheng2007MoreonNewtonMethod} the authors review \cite{SolodovSvaiter1999ConvInexactNewton} who address the case of using inexact Newton methods when the equation of interest \eqref{eq:SSTM:scheme 0} is a monotone map, which is indeed our case. The usage of Newton method is not a primary element of discussion and, as does \cite{LiCheng2007MoreonNewtonMethod}, we point the reader to the comprehensive review \cite{Martinez2000ReviewonNewton} on practical quasi-Newton methods for nonlinear equations. In conclusion, it remains to explore how different versions of Newton method for sparse systems can be used as way to reduce its computational cost but, in light of our study, we found Newton method very fast and efficient even comparatively with the Explicit Euler taming method in Section \ref{exam:StochGinzburg}.

\section{Proof of split-step method (SSM) for MV-SDEs and interacting particle systems: convergence and stability}
\label{sec:theSSMresults}

The proof appearing in Section \ref{subsection:Proof of strong error 1} depends in no way on Theorem \ref{theorem：moment bound for the big theorm time extensions } or its proof (in Section \ref{subsection: Moment bound of the SSM}). Nonetheless, Section \ref{subsection: Moment bound of the SSM} has a strong complementary effect to fully understanding the proof in Section \ref{subsection:Proof of strong error 1}.

\subsection{Some properties of the scheme}
\label{subsection:properties of the scheme }
Recall the SSM scheme of Definition \ref{def:definition of the ssm}. In this section we clarify further the choice of $h$ and then introduce two critical results arising from the SSM's structure. Note that throughout $C>0$ is a constant always independent of $h,N,M$.

\begin{remark}[Choice of $h$]
\label{remark: choice of h}
Let Assumption \ref{Ass:Monotone Assumption} hold, the constraint on $h$ in \eqref{eq:h choice} comes from 
\eqref{eq:prop:yi-yj leq xi-xj}, \eqref{eq:sum y square leq sum x square} and \eqref{eq:se1:y-y} below, where $L_f,L_u \in \bR$ and $L_{\tilde{u} }\ge 0$. Following the notation of those inequalities, under \eqref{eq:h choice}  
for $\zeta>0$, there exists $\xi \in(0,1)$ such that $h< {\xi}/{\zeta}$ and 
\begin{align*}
   &\max \Bigg\{ \frac{1}{1-2(L_f+L_u)h},~\frac{1}{1-(4  L_f^++2  L_u+2L_{\tilde{u} }+1)h},~\frac{1}{1-(4  L_f^++2  L_u+L_{\tilde{u} }+1)h}  \Bigg\}< \frac{1}{1-\xi}.
\end{align*}
For $\zeta=0$, the result is trivial and we conclude that there exist constants $C_1,C_2$ independent of $h$
\begin{align*}
   \max \Bigg\{ \frac{1}{1-2(L_f+L_u)h},~\frac{1}{1-(4  L_f^++2  L_u+2L_{\tilde{u} }+1)h},~\frac{1}{1-(4  L_f^++2  L_u+L_{\tilde{u} }+1)h}  \Bigg\}\le C_1\le 1+C_2h.
\end{align*}
As argued in Remark \ref{rem:Constraint on h is soft} the constraint on $h$ \textit{can be lifted}.
\end{remark}

\begin{lemma}
\label{lemma:SSTM:new functions def and properties1}
Choose $h$ as in \eqref{eq:h choice}. 
Then, given any $X\in \mathbb{R}^{Nd}$ there exists a unique solution $Y\in \mathbb{R}^{Nd}$ to 
\begin{align}\label{sstmmod func def}
    Y=X+h  V(Y).
\end{align}
The solution $Y$ is a measurable map of $X$. 
\end{lemma}
\begin{proof}
Recall Remark \ref{remark:OSL for the whole function / system V}. The proof is an adaptation of the proof \cite[Lemma 4.1]{2021SSM} to the $\bR^{Nd}$ case.
 

\end{proof}

\begin{proposition} [Differences relationship]
\label{prop:yi-yj leq xi-xj}
Let Assumption \ref{Ass:Monotone Assumption} hold and choose $h$ as in \eqref{eq:h choice}. For any $n\in \llbracket 0,M \rrbracket$ and $Y^{*,N}_n$ in \eqref{eq:SSTM:scheme 0}, there exists some constant $C>0$ such that for all $i,~j\in \llbracket 1,N \rrbracket$,  
\begin{align}\label{eq:prop:yi-yj leq xi-xj}
    |Y_{n}^{i,\star,N}-Y_{n}^{j,\star,N}|^2 
    &\le 
     | \hx_{n}^{i,N}-\hx_{n}^{j,N}|^2  \frac{1}{1-2(L_f+L_u)h} \le (1+Ch) | \hx_{n}^{i,N}-\hx_{n}^{j,N}|^2.
\end{align}
\end{proposition}

\begin{proof}
Take $n\in \llbracket 0,M \rrbracket$, $i,~j\in \llbracket 1,N \rrbracket$. Using Remark \ref{remark:ImpliedProperties} and Young's inequality we have
\begin{align*}
   &
   |Y_{n}^{i,\star,N}-Y_{n}^{j,\star,N}|^2
   \\
   &=
    \Big\langle  
    Y_{n}^{i,\star,N}-Y_{n}^{j,\star,N}
    ,
    \hx_{n}^{i,N}-\hx_{n}^{j,N}
    \Big\rangle
    +
    \Big\langle  
    Y_{n}^{i,\star,N}-Y_{n}^{j,\star,N}
    ,
     v\left(Y_{n}^{i,\star,N},\hm^{Y,N}_n\right)
     -
     v\left(Y_{n}^{j,\star,N},\hm^{Y,N}_n\right)
    \Big\rangle h
    \\
    &\le
    \frac{1}{2} |Y_{n}^{i,\star,N}-Y_{n}^{j,\star,N} |^2
    +
    \frac{1}{2} | \hx_{n}^{i,N}-\hx_{n}^{j,N}|^2
    +
    (L_f+L_u) |Y_{n}^{i,\star,N}-Y_{n}^{j,\star,N}|^2 h. 
\end{align*}
The argument regarding the uniformity of the constant $C$ in  regards to the parameters $h,N,M$    follows from Remark \ref{remark: choice of h}.
\end{proof}

\begin{proposition} [Summation relationship]
\label{prop:sum y square leq sum x square}
Let Assumption \ref{Ass:Monotone Assumption}  hold. Choose $h$ as in \eqref{eq:h choice}. For the process in \eqref{eq:SSTM:scheme 1} there exists a constant $C>0$ (independent of $h,N,M$) such that, for all $i\in \llbracket 1,N \rrbracket,~n\in \llbracket 0,M \rrbracket$,   
\begin{align}\label{eq:sum y square leq sum x square}
    \frac{1}{N}\sum_{i=1}^N |Y_{n}^{i,\star,N}|^2 
    &\le  
    Ch+(1+Ch) \frac{1}{N}\sum_{i=1}^N  |\hx_{n}^{i,N}|^2. 
\end{align}
\end{proposition}
\begin{proof}
From \eqref{eq:SSTM:scheme 2} we have 
 \begin{align}
&\frac{1}{N}\sum_{i=1}^N |Y_{n}^{i,\star,N}|^2 =\frac{1}{N}\sum_{i=1}^N \bigg\{\Big\langle Y_{n}^{i,\star,N},\hx_{n}^{i,N} \Big\rangle + \Big\langle  Y_{n}^{i,\star,N}, v(Y_{n}^{i,\star,N},\hm^{Y,N}_n) \Big\rangle   h \bigg\}
\nonumber\\
\label{eq:ystar2}
&\le\frac{1}{N}\sum_{i=1}^N  \bigg\{\frac{1}{2}|Y_{n}^{i,\star,N}|^2 + \frac{1}{2}|\hx_{n}^{i,N}|^2 
+\Big\langle  Y_{n}^{i,\star,N}, u(Y_{n}^{i,\star,N},\hm^{Y,N}_n) \Big\rangle   h
+ 
\frac{h}{N}\sum_{j=1}^N \Big\langle  Y_{n}^{i,\star,N}, f(Y_{n}^{i,\star,N}-Y_{n}^{j,\star,N}) \Big\rangle
\bigg\}.
\end{align} 
   By Assumption \ref{Ass:Monotone Assumption} and Young's inequality, we have  
\begin{align*}
    \frac{1}{N^2}\sum_{i=1}^N\sum_{j=1}^N \Big\langle  Y_{n}^{i,\star,N}, f(&Y_{n}^{i,\star,N}-Y_{n}^{j,\star,N}) \Big\rangle
    =\frac{1}{2 N^2}\sum_{i=1}^N\sum_{j=1}^N \Big\langle  Y_{n}^{i,\star,N}- Y_{n}^{j,\star,N}, f(Y_{n}^{i,\star,N}-Y_{n}^{j,\star,N}) \Big\rangle
    \\ 
    &\le \frac{1}{2 N^2}\sum_{i=1}^N\sum_{j=1}^N
    L_f |Y_{n}^{i,\star,N}-Y_{n}^{j,\star,N}|^2
    \le 
       \frac{2 L_f^+}{ N}\sum_{i=1}^N      |Y_{n}^{i,\star,N}|^2, \quad L_f^+=\max\{L_f,0\}.
\end{align*}
Plugging this into \eqref{eq:ystar2} and using Remark \ref{remark:ImpliedProperties} with $\Lambda=  4  L_f^++2  L_u+2L_{\tilde{u} }+1$, we have 
  \begin{align*}
\frac{1}{N}\sum_{i=1}^N |Y_{n}^{i,\star,N}|^2 
&
\le 
\frac{1}{N}\sum_{i=1}^N  \bigg\{|\hx_{n}^{i,N}|^2 
+
 2h \big(   
2 L_f^+ |Y_{n}^{i,\star,N}|^2
+C_u +\widehat L_u|Y_{n}^{i,\star,N}|^2+ L_{\tilde{u} } W^{(2)}(\hm^{Y,N}_n,\delta_0)^2 \big)
\bigg\}
\\
&\le \frac{1}{N}
\sum_{i=1}^N  \bigg\{|\hx_{n}^{i,N}|^2 
+
 2h \big(   
2 L_f^+ |Y_{n}^{i,\star,N}|^2
+C_u +\widehat L_u|Y_{n}^{i,\star,N}|^2 
+
\frac{L_{\tilde{u} } }{N}\sum_{j=1}^N |Y_{n}^{j,\star,N}|^2
\big) 
\bigg\}
\\
&\le \frac{1}{1-\Lambda h} \frac{1}{N} \sum_{i=1}^N  \bigg\{|\hx_{n}^{i,N}|^2  +2C_u h   \bigg\}
=
\frac{1}{N} \sum_{i=1}^N \bigg \{ |\hx_{n}^{i,N}|^2 (1+ h\frac{\Lambda }{1-\Lambda h}) +
\frac{2C_u h}{1-\Lambda h}
\bigg \} .
\end{align*} 
Remark \ref{remark: choice of h} yields the argument.
 \end{proof}

From Lemma \ref{lemma:SSTM:new functions def and properties1} we know a unique solution, $Y_{n}^{\star,N}$, to \eqref{eq:SSTM:scheme 0} as a function of $\hat{ X}_{n}^{N}$ exists. 
We next show that the scheme we proposed in \eqref{eq:SSTM:scheme 0}-\eqref{eq:SSTM:scheme 2}  is square integrable. 
\begin{proposition}[Second moment bounds of SSM] 
\label{prop: discrete second moment bound}
Let the setting of Theorem \ref{theorem:SSM: strong error 1} hold. Let $m\geq 2$ where $\hx^{i,N}_0\in L^m_0(\bR^d)$ for all $i\in\llbracket 1,N\rrbracket$, then there exists a constant $C>0$ independent of $h,N,M$ (but depending on $T$) such that    
\begin{align*} 
    \sup_{i\in \llbracket 1,N \rrbracket } \sup_{n\in \llbracket 0,M \rrbracket }
    \bE \big[ |\hx_{n}^{i,N}|^{2}  \big]
    +
     \sup_{i\in \llbracket 1,N \rrbracket } \sup_{n\in \llbracket 0,M-1 \rrbracket }
    \bE \big[ |Y_{n}^{i,\star,N}|^{2}  \big]
    &\le C \big(  1+ \bE\big[\, |\hx_{0}^{\cdot,N}|^{2}\big] \big) <\infty. 
\end{align*}
\end{proposition}
\begin{proof}
Let $i\in\llbracket 1,N\rrbracket$, $n\in\llbracket 0,M-1\rrbracket$, by Assumption \ref{Ass:Monotone Assumption}, from \eqref{eq:SSTM:scheme 0}-\eqref{eq:SSTM:scheme 2} and Proposition \ref{prop:sum y square leq sum x square}, since the particles are identically distributed, we have

\begin{align*}
    \bE \big[ 1+|Y_{n}^{i,\star,N}|^2| \big]  \leq 
    \bE\big[  1+|\hx_{n}^{i,N}|^2 \big] (1+Ch).
\end{align*}
Similar to \cite[Proposition 4.5]{2021SSM},  we have
\begin{align*}
     |\hx_{n+1}^{i,N}|^2 \leq 
     |\hx_{n}^{i,N}|^2
    +
    C \Big( 1+ |Y_{n}^{i,\star,N}|^2+\frac{1}{N} \sum_{j=1}^N |Y_{n}^{j,\star,N}|^2
    \Big)  ( h+ | \Delta W_{n}^i|^2 ) 
    + 2\Big\langle Y_{n}^{i,\star,N}, \sigma(t_n,Y_{n}^{i,\star,N},\hm^{Y,N}_n) \Delta W_{n}^i  \Big\rangle.
\end{align*}
Taking expectations and summing $1$ to both sides, Young's inequality yields 
\begin{align*}
   \bE \big[ 1+|\hx_{n+1}^{i,N}|^2\big]  
   \leq  
   \bE \big[1+ |\hx_{n}^{i,N}|^2 \big] (1+Ch).
\end{align*}
By induction and using that the particles are identically distributed, we conclude that 
\begin{align}
   \sup_{i\in \llbracket 1,N \rrbracket } \sup_{n\in \llbracket 0,M \rrbracket } \bE \big[ 1+ |\hx_{n}^{i,N}|^2\big]  
   \leq   
   \sup_{i\in \llbracket 1,N \rrbracket }
   \bE \big[ 1+|\hx_{0}^{i,N}|^2\big] (1+Ch)^M
   \leq  
   (1+\bE \big[ |\hx_{0}^{\cdot,N}|^2\big]) e^{CT}<\infty, 
\end{align}
where we used $Mh=T$ and that the $\{\hx_{0}^{i,N}\}_i$ are i.i.d. \\
The inequality for $ \sup_{i\in \llbracket 1,N \rrbracket } \sup_{n\in \llbracket 0,M-1 \rrbracket } \bE\big[ | Y_{n}^{i,\star,N} |^{2}  \big]  $ follows using similar argument.
\end{proof}

\color{black}

We provide the following auxiliary proposition to deal with the cross products terms in the later proofs.
\begin{proposition} 
\label{prop:auxilary for multipy of r.v.}
Take $N\in\bN$, for all $ i\in \llbracket 1,N \rrbracket$, for any given $p\in\bN $, sequences $\big\{ \{a_i\}_i: \sum_{i=1}^N a_i=p,~ a_i\in\bN \big\}$ and any collection of identically distributed $L^p$-integrable random variables $\{X_i\}_i$ we have 
\begin{align*}
    \bE \big[       \prod_{i=1}^N |X_i|^{a_i}  \big] 
    \le  \bE \big[   |X_1|^p     \big].
\end{align*}
\end{proposition}

\begin{proof}
Using the notation above, by Young's inequality, for any $ i,j\in \llbracket 1,N \rrbracket$ we have 
\begin{align*}
    |X_i|^{a_i}|X_j|^{a_j}   
    \le    
    \frac{a_i}{a_i+a_j}|X_i|^{a_i+a_j}+\frac{a_j}{a_i+a_j}|X_j|^{a_i+a_j}.
\end{align*}
Thus, by induction and using that the $\{X_i\}_i$ are identically distributed, the result follows.
\end{proof}

\subsection[Proof of the pointwise convergence result]{Proof of Theorem \ref{theorem:SSM: strong error 1}: the pointwise mean-square convergence result}\label{subsection:Proof of strong error 1}
We provide here the proof of Theorem \ref{theorem:SSM: strong error 1}. Throughout this section, we follow the notation introduced in Theorem \ref{theorem:SSM: strong error 1} and let Assumption \ref{Ass:Monotone Assumption} hold, $h$ is chosen as in \eqref{eq:h choice},  $m\ge 4q+4$, where $m$ is defined in \eqref{Eq:General MVSDE} and $q$ is defined in Assumption \ref{Ass:Monotone Assumption}.    Note that throughout $C>0$ is a constant always independent of $h,N,M$ but possibly depending on $T$ and $m$. 
\begin{proof}
Let $i\in\llbracket 1,N\rrbracket$, $n\in\llbracket 0,M-1\rrbracket$,  $s\in[0,h]$, $t_n=nh$ and $p\ge 2$ with $m\ge 4q+4$,  using same notation as in \eqref{Eq:MV-SDE Propagation}, define the following auxiliary process 
\begin{align*}
  X_{n}^{i,N}&=X_{t_n}^{i,N},\quad \Delta X^i_{t_n+s} =X_{t_n+s}^{i,N}-\hx_{t_n+s}^{i,N} ,\quad t_n=nh, \quad
    \Delta W_{n,s}^i=W_{t_n+s}^i-W_{t_n}^i,
  \\
    Y_{n}^{i,X,N} &=X_{n}^{i,N}+h v (Y_{n}^{i,X,N},\mu^{Y,X,N}_n ),  
     \quad 
 \quad 
  \mu^{Y,X,N}_n(\dd x):= \frac1N \sum_{j=1}^N \delta_{Y_{n}^{j,X,N}}(\dd x).
\end{align*} 
For all $n\in\llbracket 0,M-1\rrbracket,~i\in \llbracket 1,N \rrbracket,~r\in[0,h]$, from \eqref{eq: scheme continous extension in SDE form}, we have
\begin{align*}
    &|\Delta X^i_{t_n+r}|^2 
    = 
    \Big| 
    \Delta X^i_{t_n} 
    +
    \int_{t_n}^{t_{n}+r}  \big( v(X_{s}^{i,N},\mu_{s}^{X,N})-v(Y_n^{i,X,N},\mu^{Y,X,N}_{n})   
    \big) \dd s  
    \\
    &+ 
    \int_{t_n}^{t_{n}+r}  \big( v(Y_n^{i,X,N},\mu^{Y,X,N}_{n}) -v\left(Y_{n}^{i,\star,N},\hm^{Y,N}_n\right)
    \big) \dd s   
    + 
    \int_{t_n}^{t_{n}+r}  \big( b(s,X_{s}^{i,N},\mu_{s}^{X,N})-b(t_n,Y_n^{i,X,N},\mu^{Y,X,N}_{n})   
    \big) \dd s
    \\
        &+ \int_{t_n}^{t_{n}+r}  \big(  b(t_n,Y_n^{i,X,N},\mu^{Y,X,N}_{n})   
        - b(t_n,Y_{n}^{i,\star,N},\hm^{Y,N}_n)
    \big) \dd s
    \\
    &
    + 
    \int_{t_n}^{t_{n}+r}  \big( \sigma(s,X_{s}^{i,N},\mu_{s}^{X,N})
    -\sigma(t_n,Y_n^{i,X,N},\mu^{Y,X,N}_{n})   
    \big) \dd W^i_s
    \\
        &
        + \int_{t_n}^{t_{n}+r}  \big(  \sigma(t_n,Y_n^{i,X,N},\mu^{Y,X,N}_{n})   
        - \sigma(t_n,Y_{n}^{i,\star,N},\hm^{Y,N}_n)
    \big)  \dd W^i_s \Big|^2 .
\end{align*} 
Taking expectations on both side, using Jensen's inequality and It\^o's isometry, we have  
\begin{align}
\label{eq:se1:ultimate eqaution IIII}
    \bE \big[  |\Delta X^i_{t_n+r}|^2 \big]
    &\le
    (1+h) I_1+(1+\frac{1}{h}) I_2+2 I_3+2 I_4,
\end{align}        
where the terms $I_1,I_2,I_3,I_4$ are defines as follows 
\begin{align}
     I_1= \bE \Big[        \Big|& 
    \Delta X^i_{t_n}
    + 
    \int_{t_n}^{t_{n}+r}  \big( v(Y_n^{i,X,N},\mu^{Y,X,N}_{n}) -v\left(Y_{n}^{i,\star,N},\hm^{Y,N}_n\right)
    \big) \dd s  
    \\
    &+ \int_{t_n}^{t_{n}+r}  \big(  b(t_n,Y_n^{i,X,N},\mu^{Y,X,N}_{n})   
    - b(t_n,Y_{n}^{i,\star,N},\hm^{Y,N}_n)
    \big) \dd s
     \Big|^2\Big],
 \end{align}
 \begin{align}
     I_2=  \bE \Big[ \Big| & \int_{t_n}^{t_{n}+r}  \big( v(X_{s}^{i,N},\mu_{s}^{X,N})-v(Y_n^{i,X,N},\mu^{Y,X,N}_{n})   
    \big) \dd s  
    \\
    &+ 
    \int_{t_n}^{t_{n}+r}  \big( b(s,X_{s}^{i,N},\mu_{s}^{X,N})-b(t_n,Y_n^{i,X,N},\mu^{Y,X,N}_{n})   
    \big) \dd s     
     \Big|^2\Big],
 \end{align}
 \begin{align}
     I_3=&  \bE \Big[  \Big| 
     \int_{t_n}^{t_{n}+r}  \big( \sigma(s,X_{s}^{i,N},\mu_{s}^{X,N})
    -\sigma(t_n,Y_n^{i,X,N},\mu^{Y,X,N}_{n})   
    \big) \dd W^i_s
     \Big|^2\Big],
 \end{align}
 \begin{align}
     I_4=&  \bE \Big[  \Big| 
     \int_{t_n}^{t_{n}+r}  \big(  \sigma(t_n,Y_n^{i,X,N},\mu^{Y,X,N}_{n})   
        - \sigma(t_n,Y_{n}^{i,\star,N},\hm^{Y,N}_n)
    \big)  \dd W^i_s
     \Big|^2\Big].
\end{align}

For $I_1$, Young's inequality yields 
\begin{align}
    \nonumber
    I_1& =
    \bE \Big[       
     \Big|
   X_{n}^{i,N}+
    \big(  
    V_n^{Y,i}
    +b(t_n,Y_n^{i,X,N},\mu^{Y,X,N}_{n}) \big) r 
    -
     \hx_{n}^{i,N}-
    \big(  
    V_n^{*,i}
    +b(t_n,Y_n^{i,\star,,N},\hm^{Y,N}_{n}) \big) r
     \Big|^2 \Big]
     \\
     \label{eq:I1:term1}
     &\le \bE \Big[        
      \Big|
   X_{n}^{i,N}-
     \hx_{n}^{i,N}
     +
    \big(  
    V_n^{Y,i}-V_n^{*,i}\big)
    r 
     \Big|^2\Big](1+\frac{h}{2}) 
    + \bE \Big[        \Big|
    b(t_n,Y_n^{i,X,N},\mu^{Y,X,N}_{n})  
    -b(t_n,Y_n^{i,\star,,N},\hm^{Y,N}_{n})  
     \Big|^2 \Big](\frac{h}{2}+h),
\end{align}   
where  $ V_n^{Y,i}$ and $V_n^{*,i}$ stand for   
     $V_n^{Y,i}=    v (Y_n^{i,X,N},\mu^{Y,X,N}_{n} )$ and   
 $V_n^{*,i} = v (Y_n^{i,\star,N},\hm^{Y,N}_{n} )$ respectively. 

For the first term of \eqref{eq:I1:term1}, recall the SSM defined in \eqref{eq:SSTM:scheme 1}. We have
\begin{align*}
    \bE \Big[        
      \Big|
   X_{n}^{i,N}-
     \hx_{n}^{i,N}
    + & 
    \big( V_n^{Y,i}
    -
     V_n^{*,i} \big)r 
     \Big|^2\Big]
     \\
     =&\bE \Big[ 
     \Big\langle X_{n}^{i,N}-
     \hx_{n}^{i,N}
     +
    \big(  V_n^{Y,i}  -  V_n^{*,i} \big)r,
    Y_n^{i,X,N}-
     Y_n^{i,\star,N}
     +
    \big(  V_n^{Y,i}  -  V_n^{*,i} \big)(r-h)
    \Big\rangle\Big]
    \\
    =&\bE \Big[ 
     \Big\langle X_{n}^{i,N}-
     \hx_{n}^{i,N},
    Y_n^{i,X,N}-
     Y_n^{i,\star,N}
    \Big\rangle\Big]
    +
    \bE \Big[ 
     \Big\langle X_{n}^{i,N}-
     \hx_{n}^{i,N},
    \big(  V_n^{Y,i}  -  V_n^{*,i} \big)
    \Big\rangle\Big](r-h)
    \\
    &
    +
    \bE \Big[ 
     \Big\langle  Y_n^{i,X,N}-
     Y_n^{i,\star,N},
    \big(  V_n^{Y,i}  -  V_n^{*,i} \big)
    \Big\rangle\Big]r
    -r(h-r)
     \bE \Big[ 
    \Big|  V_n^{Y,i}  -  V_n^{*,i} \Big|^2\Big].
\end{align*}
Using the relationship that \eqref{eq:SSTM:scheme 1} induces, we have 
\begin{align*}
    V_n^{Y,i}  -  V_n^{*,i} =\frac{  Y_n^{i,X,N}-X_{n}^{i,N}
     +Y_n^{i,\star,N}-
     \hx_{n}^{i,N}}{h}.
\end{align*}
We first deduce that
\begin{align}
 \nonumber
    \bE \Big[        
      \Big|
      &
   X_{n}^{i,N}-
     \hx_{n}^{i,N}
     +
    \big( V_n^{Y,i}
    -
     V_n^{*,i} \big)r 
     \Big|^2\Big] 
    =
    \bE \big[ 
     | X_{n}^{i,N}-
     \hx_{n}^{i,N}|^2
   \big]  
    +
    \bE \Big[ 
     \Big\langle X_{n}^{i,N}-
     \hx_{n}^{i,N},
   V_n^{Y,i} -
     V_n^{*,i}
    \Big\rangle\Big] 2r
    \\ \nonumber
    & \qquad
    +
    \bE \Big[ 
     \Big\langle  (Y_n^{i,X,N}-
     Y_n^{i,\star,N})-(X_{n}^{i,N}-
     \hx_{n}^{i,N} ),
       V_n^{Y,i}  -  V_n^{*,i}  
    \Big\rangle\Big]\frac{r^2}{h}
    \\ 
    &~= 
    \bE \big[ 
     | X_{n}^{i,N}-
     \hx_{n}^{i,N}|^2
   \big]  (1-C_{h,r})
    +
    \bE \big[ 
     | 
   Y_n^{i,X,N}-
     Y_n^{i,\star,N}
    |^2\big] C_{h,r}
     \label{eq:se1: y-y v-v source 0}
     +
    \bE \Big[ 
     \Big\langle  Y_n^{i,X,N}-
     Y_n^{i,\star,N},V_n^{Y,i}  -  V_n^{*,i} 
    \Big\rangle\Big]\frac{r^2}{h}.
\end{align}
Where $C_{h,r}=(2hr-r^2)/{2h^2}$. Also, for the second term of \eqref{eq:I1:term1}, using Assumption \ref{Ass:Monotone Assumption} and that the particles are identically distributed 
\begin{align}
\nonumber
    \bE \Big[   \Big|b(t_n,Y_n^{i,X,N},&~\mu^{Y,X,N}_{n})
    -b(t_n,Y_{n}^{i,\star,N},\hm^{Y,N}_n)\Big|^2 \Big]
    \\
    \nonumber
    \le&
    C\bE \Big[   |Y_n^{i,X,N}-Y_n^{i,\star,N}|^2+ W^{(2)}(\mu^{Y,X,N}_{n},\hm^{Y,N}_n)\Big]
    \\
    \label{eq:se1:by-by}
    \le&
    C\bE \big[   |Y_n^{i,X,N}-Y_n^{i,\star,N}|^2\big]
    + C \bE \Big[  \frac{1}{N} \sum_{j=1}^N |Y_n^{j,X,N}-Y_n^{j,\star,N}|^2\Big]
    \le C\bE \big[   |Y_n^{i,X,N}-Y_n^{i,\star,N}|^2\big].
\end{align}
By Assumption \ref{Ass:Monotone Assumption} and  using Young's inequality once again
\begin{align}
\label{eq:se1:Yi-Yi star}
   \bE \big[   |Y_n^{i,X,N} &-Y_n^{i,\star,N}|^2 \big] 
    \le 
    \bE \Big[  \Big\langle   Y_n^{i,X,N}-Y_n^{i,\star,N},
    X_{n}^{i,N}-\hx_{n}^{i,N}+V_n^{Y,i}- V_n^{*,i}
    \Big\rangle \Big]h
   \\ 
   &
   \le 
    \bE \Big[  
    \frac{1}{2} |Y_n^{i,X,N}-Y_n^{i,\star,N}|^2 +\frac{1}{2} |X_{n}^{i,N}-\hx_{n}^{i,N}|^2 \Big]
    \label{eq:se1: y-y v-v}
    +\bE \Big[  
    \Big\langle   Y_n^{i,X,N}-Y_n^{i,\star,N}, 
    V_n^{Y,i}- V_n^{*,i}
    \Big\rangle \Big] h. 
\end{align}
For the last term \eqref{eq:se1: y-y v-v}, since the particles are identically distributed, Assumption \ref{Ass:Monotone Assumption} and Remark \ref{remark:OSL for the whole function / system V} yield 
\begin{align}
\nonumber
\label{eq:se1:y-y v-v}
    \bE \Big[   
    \Big\langle   Y_n^{i,X,N}-Y_n^{i,\star,N}, 
    V_n^{Y,i}- V_n^{*,i}
    \Big\rangle \Big] 
     &
    \le
        \bE \Big[  \frac{1}{N} \sum_{j=1}^N \Big\langle   Y_n^{j,X,N}-Y_n^{j,\star,N}, 
    V_n^{Y,j}- V_n^{*,j}
    \Big\rangle \Big]
    \\  
     &\le \Big(2L_f^++L_u+\frac{1}{2}+\frac{L_{\tilde{u} }}{2}\Big) \bE \big[  | Y_n^{i,X,N}-Y_n^{i,\star,N}|^2 \big]. 
\end{align}
Thus, injecting \eqref{eq:se1:y-y v-v} back into  \eqref{eq:se1: y-y v-v} and \eqref{eq:se1:Yi-Yi star}, set $\Gamma_2=4L_f^++2L_u+L_{\tilde{u} }+1$, then by Remark \ref{remark: choice of h},  
\begin{align}
\label{eq:se1:y-y}
    &\bE \big[   |Y_n^{i,X,N} -Y_n^{i,\star,N}|^2 \big]
   \le \frac{1}{1-\Gamma_2 h}
    \bE \big[  |X_{n}^{i,N}-\hx_{n}^{i,N}|^2 \big]
    \le \bE \big[  |X_{n}^{i,N}-\hx_{n}^{i,N}|^2 \big] (1+Ch).
\end{align}
Plug \eqref{eq:se1:y-y} and  \eqref{eq:se1:y-y v-v}  back into \eqref{eq:se1: y-y v-v source 0}, \eqref{eq:se1:by-by} and \eqref{eq:I1:term1}. We then conclude that  
\begin{align}
\label{eq:se1:I1 final result}
    I_1\le \bE \big[  |X_{n}^{i,N}-\hx_{n}^{i,N}|^2 \big] (1+Ch).
\end{align}
For $I_2$, by Young's and Jensen's inequality, we have
\begin{align}
\label{eq:se1:I2 term 1}
    I_2\le  h~\bE \Big[  & \int_{t_n}^{t_{n}+h}  \Big|  v(X_{s}^{i,N},\mu_{s}^{X,N})-v(Y_n^{i,X,N},\mu^{Y,X,N}_{n})   
     \Big|^2 \dd s  
    \\
    \label{eq:se1:I2 term 2}
    &+ 
    \int_{t_n}^{t_{n}+h}  \Big|   b(s,X_{s}^{i,N},\mu_{s}^{X,N})-b(t_n,Y_n^{i,X,N},\mu^{Y,X,N}_{n})   
      \Big|^2 \dd s     
     \Big].
\end{align} 
For \eqref{eq:se1:I2 term 1}, from Assumption \ref{Ass:Monotone Assumption},  using Young's, Jensen's, and Cauchy-Schwarz inequality
\begin{align}
    \nonumber
     &\bE \Big[        \Big| 
    v( X_{s}^{i,N},\mu_s^{X,N})
    - v(  Y_n^{i,X,N},\mu_n^{Y,X,N}) \Big|^2\Big]
    \\  \label{eq:se1:f-f things}
    &\le C
     \bE \Big[\Big|
    u( X_{s}^{i,N},\mu_s^{X,N})
    - u(  Y_n^{i,X,N},\mu_n^{Y,X,N})
     \Big|^2
     +
     \frac{1}{N} \sum_{j=1}^N 
     \Big|
    f(X_{s}^{i,N}-X_{s}^{j,N} )-
    f(Y_{n}^{i,X,N}-Y_{n}^{j,X,N} )
     \Big|^2
     \Big]
     \\ \nonumber
     &\le \frac{C}{N} \sum_{j=1}^N
       \bE \Big[  \Big| \Big( 
       1+|X_{s}^{i,N}-X_{s}^{j,N}|^{q}    +|Y_{n}^{i,X,N}-Y_{n}^{j,X,N}|^{q}         \Big)
        |X_{s}^{i,N}-Y_{n}^{i,X,N}-(X_{s}^{j,N}-Y_{n}^{j,X,N})|   \Big|^2  \Big]
     \\ \nonumber
     &~ 
     +
      C \bE \Big[   (1+| X_{s}^{i,N}|^{2q} +| Y_n^{i,X,N}|^{2q})(|X_{s}^{i,N}-Y_n^{i,X,N}|^{2})+   \frac{1}{N} \sum_{j=1}^N | X_{s}^{j,N}- Y_n^{j,X,N} |^2   \Big]
    \\ 
      \label{eq:se1: v-v inequality 1}
     &\le C 
     \sqrt{ \bE \Big[ 1+| X_{s}^{i,N}|^{4q} +| Y_n^{i,X,N}|^{4q}  \Big]  \bE \Big[ |X_{s}^{i,N}-Y_n^{i,X,N}|^{4}  \Big] }
     +
     \bE \Big[   \frac{1}{N} \sum_{j=1}^N | X_{s}^{j,N}- Y_n^{j,X,N} |^2   \Big]
     \\ 
     \label{eq:se1: v-v inequality 2}
     &~  + \frac{C}{N} \sum_{j=1}^N
       \sqrt{
       \bE \Big[   1+|X_{s}^{i,N}-X_{s}^{j,N}|^{4q}    +|Y_{n}^{i,X,N}-Y_{n}^{j,X,N}|^{4q}         \Big]
       \bE \Big[ |X_{s}^{i,N}-Y_{n}^{i,X,N}|^{4}    +|X_{s}^{j,N}-Y_{n}^{j,X,N}|^{4}     \Big]
      }.
\end{align}
Using the structure of the SSM, Young's and Jensen's inequality,  and Proposition \ref{prop:yi-yj leq xi-xj} we have 
\begin{align}
    \label{eq:se1:Xs-Yn 2 estimate}
     |X_{s}^{i,N}-Y_n^{i,X,N} |^2 
     &
     \le
     2  |X_{s}^{i,N}-X_{n}^{i,N} |^2+
     2 |X_{n}^{i,N}-Y_n^{i,X,N} |^2,
     \\
     \nonumber
    |X_{n}^{i,N}-Y_n^{i,X,N} |^2
     & =\Big|v(  Y_n^{i,X,N},\mu_n^{Y,X,N}) h \Big|^2
     \le
     2\Big|u(  Y_n^{i,X,N},\mu_n^{Y,X,N}) h \Big|^2
     +
     \frac{2h^2}{N} \sum_{j=1}^N \Big|f(  Y_n^{i,X,N}- Y_n^{j,X,N})\Big|^2
     \\\nonumber 
     &\le 
     C \Big( 1+  |Y_n^{i,X,N}|^{2q+2} + \frac{1 }{N} \sum_{j=1}^N |Y_n^{j,X,N}|^{2}
     \Big) h^2
     +
     \frac{C h^2   }{N} \sum_{j=1}^N 
     \Big( 1+  |Y_n^{i,X,N}- Y_n^{j,X,N}|^{2q+2}\Big)
     \\\nonumber
     &\le
     C \Big( 1+  |Y_n^{i,X,N}|^{2q+2}
      + \frac{1 }{N} \sum_{j=1}^N |Y_n^{j,X,N}|^{2}
     \Big) h^2
     +
     \frac{C h^2  }{N} \sum_{j=1}^N 
     \Big( 1+  |X_n^{i,N}-X_n^{j,N}|^{2q+2}\Big).
\end{align}
Similarly, we have 
\begin{align}
    \label{eq:se1:Xs-Yn 4 estimate}
     |X_{s}^{i,N}-Y_n^{i,X,N} |^4 
     &\le
     16  |X_{s}^{i,N}-X_{n}^{i,N} |^4+
     16 |X_{n}^{i,N}-Y_n^{i,X,N} |^4,
     \\
     \nonumber
     |X_{n}^{i,N}-Y_n^{i,X,N} |^4
     &
     \le
      C \Big( 1+  |Y_n^{i,X,N}|^{4q+4}  + \frac{1 }{N} \sum_{j=1}^N |Y_n^{j,X,N}|^{4}\Big) h^4
      +
     \frac{C h^4  }{N} \sum_{j=1}^N 
     \Big( 1+  |X_n^{i,N}-X_n^{j,N}|^{4q+4}\Big).
\end{align}
From \eqref{Eq:MV-SDE Propagation} and using \eqref{eq:momentboundParticiInteractingSystem} (since $m\ge 4q+4$) alongside Young's inequality and It\^o's isometry, we have  
\begin{align*}
    \bE \big[   & |X_{s}^{i,N}-X_n^{i,N} |^2     \big] 
    \le \bE \Big[ \Big|
    \int_{t_n}^{s} v( X_{u}^{i,N},\mu_u^{X,N})
    +b(u,X_{u}^{i,N},\mu_u^{X,N}) \dd u
    +  
    \int_{t_n}^{s} \sigma(u,X_{u}^{i,N},\mu_u^{X,N}) \dd W_u^i  \Big|^2 \Big]~\le C h,
    \\
    \bE \big[    &|X_{s}^{i,N}-X_n^{i,N} |^4     \big] 
    \le
     \bE \Big[ \Big|
    \int_{t_n}^{s} v( X_{u}^{i,N},\mu_u^{X,N})
    +b(u,X_{u}^{i,N},\mu_u^{X,N}) \dd u
    +  \int_{t_n}^{s} \sigma(u,X_{u}^{i,N},\mu_u^{X,N}) \dd W_u^i  \Big|^4 \Big]~\le C h^2.
\end{align*}
Also, using \eqref{eq:momentboundParticiInteractingSystem}, Jensen's and Young's inequality (since $m\ge 4q+4$) we have  
\begin{align*}
    \bE \Big[     \frac{C h^2  }{N} \sum_{j=1}^N 
     \Big( 1+  |X_t^{i,N}-X_t^{j,N}|^{2q+2}\Big)    \Big]
     \le&
     Ch^2
     \quad\textrm{and}\quad  
     \bE \Big[   \Big|  \frac{C h^2  }{N} \sum_{j=1}^N 
     \Big( 1+  |X_t^{i,N}-X_t^{j,N}|^{2q+2}  \Big) \Big|^2   \Big]
     \le
     Ch^4.
\end{align*}
This next argument uses steps similar to those used in \eqref{eq:mmb:def of HXp} and \eqref{eq:mmb:def of HYp} (appearing in the proof of Theorem \ref{theorem: discrete moment bound}). Since $X^{\cdot,N}$ has bounded moments via \eqref{eq:momentboundParticiInteractingSystem} (this refers to the true interacting particle system), we have for any $m\ge p\ge 2$ that 
\begin{align*}
    \bE\big[|Y_n^{i,X,N}|^{p}\big] 
    &\le\Big( 
     4^{p} \bE \Big[       \frac{1}{N} \sum_{j=1}^N |X_{n}^{i,N}- X_{n}^{j,N} |^{p} \Big]
    +
    4^p \bE \Big[     \Big|\frac{1}{N} \sum_{j=1}^N ( 1+  |X_{n}^{j,N}|^2) \Big|^{p/2}     \Big]+1\Big)(1+Ch)
    \le C.
\end{align*}
Collecting all the terms above, using that the particles are identically distributed, we have 
\begin{align}
    \label{eq:se1: expec diff order 1}
     \bE \big[   |X_{s}^{i,N}-Y_n^{i,X,N} |^2     \big]
     &\le Ch,
     \qquad
     \bE \big[   |X_{s}^{i,N}-Y_n^{i,X,N} |^4     \big]
     \le Ch^2,
     \qquad
      \bE\big[|Y_n^{i,X,N}|^{p}\big]
    \le C ,
     \\
     \label{eq:se1: expec diff order 2}
     \bE \Big[   \Big| W^{(2)}(\mu_s^{X,N},\mu_n^{Y,X,N} )  \Big|^2     \Big]
     &\le \bE \Big[ \frac{1}{N} \sum_{j=1}^N |X_{s}^{j,N}- Y_n^{j,X,N} |^2     \Big] \le Ch.
\end{align}
Plugging all the above inequalities back into \eqref{eq:se1: v-v inequality 1} and \eqref{eq:se1: v-v inequality 2}, we conclude that 
\begin{align}
\label{eq:se1:I2 result for term v}
     \bE \Big[        \Big| 
    v( X_{s}^{i,N},\mu_s^{X,N})
    - v(  Y_n^{i,X,N},\mu_n^{Y,X,N}) \Big|^2\Big]
    \le Ch.
\end{align}
We now consider term \eqref{eq:se1:I2 term 2} of $I_2$. By Assumption \ref{Ass:Monotone Assumption}, using \eqref{eq:se1: expec diff order 1} and \eqref{eq:se1: expec diff order 2} 
\begin{align}
    \bE \Big[ \Big|& b(s,X_{s}^{i,N},\mu_{s}^{X,N})-b(t_n,Y_n^{i,X,N},\mu^{Y,X,N}_{n})\Big|^2    \Big]  
    \label{eq:se1:I2 result for term b}
   \le
   C\bE\Big[        h+   |X_{s}^{i,N}-Y_n^{i,X,N} |^2
    +
     \Big| W^{(2)}(\mu_s^{X,N},\mu_{n}^{Y,X,N} )  \Big|^2 \Big] \le Ch.
\end{align}
Thus, plugging \eqref{eq:se1:I2 result for term v}, \eqref{eq:se1:I2 result for term b} back into \eqref{eq:se1:I2 term 1} and \eqref{eq:se1:I2 term 2}, we have \begin{align}
\label{eq:se1:I2 final result}
    I_2\le Ch^3.
\end{align}
For $I_3$,  by It\^o's isometry, the results in \eqref{eq:se1: expec diff order 1} and \eqref{eq:se1: expec diff order 2}, and using similar argument as in \eqref{eq:se1:I2 result for term b} we have 
\begin{align}
I_3=&  \bE \Big[  \Big| 
     \int_{t_n}^{t_{n}+r}  \Big( \sigma(s,X_{s}^{i,N},\mu_{s}^{X,N})
    -\sigma(t_n,Y_n^{i,X,N},\mu^{Y,X,N}_{n})   
    \Big) \dd W^i_s
     \Big|^2\Big]    
     \nonumber
     \\
     \label{eq:se1:I3 final result}
     \le& \bE \Big[  
     \int_{t_n}^{t_{n}+h}  \Big| \Big( \sigma(s,X_{s}^{i,N},\mu_{s}^{X,N})
    -\sigma(t_n,Y_n^{i,X,N},\mu^{Y,X,N}_{n})   
    \Big)\Big|^2\ \dd s
     \Big]  \le Ch^2.
\end{align}
Similarly for $I_4$, by It\^o's isometry,     Proposition \ref{prop: discrete second moment bound},   Equation \eqref{eq:se1:y-y} and using similar argument in \eqref{eq:se1:by-by} 
\begin{align}
I_4=&  \bE \Big[  \Big| 
     \int_{t_n}^{t_{n}+r}  \Big(  \sigma(t_n,Y_n^{i,X,N},\mu^{Y,X,N}_{n})   
        - \sigma(t_n,Y_{n}^{i,\star,N},\hm^{Y,N}_n)
    \Big)  \dd W^i_s
     \Big|^2\Big]    
     \nonumber
     \\
     \label{eq:se1:I4 final result}
     \le& \bE \Big[  
     \int_{t_n}^{t_{n}+h}  \Big| \Big( \sigma(t_n,Y_n^{i,X,N},\mu^{Y,X,N}_{n})   
    - \sigma(t_n,Y_{n}^{i,\star,N},\hm^{Y,N}_n)   
    \Big)\Big|^2\ \dd s
     \Big]  \le  \bE \big[  |X_{n}^{i,N}-\hx_{n}^{i,N}|^2 \big] Ch.
\end{align}
Plugging \eqref{eq:se1:I1 final result}, \eqref{eq:se1:I2 final result} \eqref{eq:se1:I3 final result}  and \eqref{eq:se1:I4 final result} back to \eqref{eq:se1:ultimate eqaution IIII}, we have, for all $n\in \llbracket 0,M-1 \rrbracket$,  $i\in \llbracket 1,N \rrbracket$ and $r\in[0,h]$  
\begin{align*} 
  \bE \big[  |\Delta X^i_{t_n+r}|^2 \big] 
    &\le 
    (1+h)  \bE \big[  |X_{n}^{i,N}-\hx_{n}^{i,N}|^2 \big] (1+Ch)+(1+\frac{1}{h})Ch^3+Ch^2+  \bE \big[  |X_{n}^{i,N}-\hx_{n}^{i,N}|^2 \big] Ch
    \\
    &\le
    \bE \big[  |X_{n}^{i,N}-\hx_{n}^{i,N}|^2 \big] (1+Ch)+Ch^2.
\end{align*}
By backward induction, the discrete Gr\"onwall's lemma delivers the result of \eqref{eq:convergence theroem term 1}.  
\end{proof}

\subsection[Proof of the moment bound result]{Proof of Theorem \ref{theorem：moment bound for the big theorm time extensions }: the moment bound result}
\label{subsection: Moment bound of the SSM}

In this section prove Theorem \ref{theorem：moment bound for the big theorm time extensions }. 
Throughout this section we follow the notation introduced in Theorem \ref{theorem：moment bound for the big theorm time extensions } and let: Assumption \ref{Ass:Monotone Assumption} hold, $h$ is chosen as in \eqref{eq:h choice} and  $m\ge 2p$ with $m$ as defined in \eqref{Eq:General MVSDE}.   

We first prove a moment bounds result across the timegrid then extend it to the continues process as stated in Theorem \ref{theorem：moment bound for the big theorm time extensions }.
\begin{theorem}[Moment bounds of SSM] 
\label{theorem: discrete moment bound}
Let the setting of Theorem \ref{theorem:SSM: strong error 1} hold. Let $m\geq 2$ where $\hx^{i,N}_0\in L^m_0(\bR^d)$ for all $i\in\llbracket 1,N\rrbracket$ and let $\hx^{i,N}$ be the continuous-time extension of the SSM given by \eqref{eq: scheme continous extension in SDE form}. Let $2p\in [2,m]$,   then there exists a constant $C>0$ independent of $h,N,M$ (but depending on $T$ and $m$) such that     
\begin{align*} 
    \sup_{i\in \llbracket 1,N \rrbracket } \sup_{n\in \llbracket 0,M \rrbracket }
    \bE \big[ |\hx_{n}^{i,N}|^{2p}  \big]
    +
     \sup_{i\in \llbracket 1,N \rrbracket } \sup_{n\in \llbracket 0,M-1 \rrbracket }
    \bE \big[ |Y_{n}^{i,\star,N}|^{2p}  \big]
    &\le C \big(  1+ \sup_{i\in \llbracket 1,N \rrbracket } \bE\big[\, |\hx_{0}^{i,N}|^{2p}\big] \big) <\infty. 
\end{align*}
\end{theorem}
\begin{proof}
The next inequality introduces the quantities $H^{X,p}_{n}$ and $H^{Y,p}_{n}$. For any $ i\in \llbracket 1,N \rrbracket ,\  n\in \llbracket 0,M \rrbracket$, by Young's and Jensen's inequality 
\begin{align}
\nonumber 
     \bE\big[ |\hx_{n}^{i,N} |^{2p}  \big] 
     &=
     \bE \Big[   \Big|    \frac{1}{N} \sum_{j=1}^N  (   \hx_{n}^{i,N}- \hx_{n}^{j,N} )  
    +\frac{1}{N} \sum_{j=1}^N \hx_{n}^{j,N} 
    \Big|^{2p} 
    \Big]
    \\
    \label{eq:mmb:def of HXp} 
    &\le 4^p \bE \Big[       \frac{1}{N} \sum_{j=1}^N |\hx_{n}^{i,N}- \hx_{n}^{j,N} |^{2p} \Big]
    +
    4^p \bE \Big[     \Big|\frac{1}{N} \sum_{j=1}^N ( 1+  |\hx_{n}^{j,N}|^2) \Big|^{p}     \Big]+1=H^{X,p}_{n},
    \\ 
     \bE\big[ |Y_{n}^{i,\star,N} |^{2p}  \big]   
    \label{eq:mmb:def of HYp}
    &\le 4^p \bE \Big[       \frac{1}{N} \sum_{j=1}^N |Y_{n}^{i,\star,N}- Y_{n}^{j,\star,N} |^{2p} \Big]
    +
    4^p \bE \Big[     \Big|\frac{1}{N} \sum_{j=1}^N  (   1+ |Y_{n}^{j,\star,N}|^{2} )   \Big|^{p}  \Big]+1=H^{Y,p}_{n}. 
\end{align}
Using the following inequalities from Proposition \ref{prop:yi-yj leq xi-xj} and \ref{prop:sum y square leq sum x square}, we have $H^{Y,p}_{n} \le  H^{X,p}_{n}(1+Ch)$, 
\begin{align*}
|Y_{n}^{i,\star,N}-Y_{n}^{j,\star,N}|^2 
    &\le 
     | \hx_{n}^{i,N}-\hx_{n}^{j,N}|^2 (1+Ch)
     \quad \textrm{and}\quad  
  \frac{1}{N} \sum_{j=1}^N    (1+|Y_{n}^{j,\star,N}|^2) 
    \le \Big[        
      \frac{1}{N} \sum_{j=1}^N
    (1+| \hx_{n}^{j,N}|^2)\Big]  (1+Ch) .
\end{align*} 
We now prove that $H^{X,p}_{n+1} \le H^{Y,p}_{n}(1+Ch) $. 
For the first element composing $H^{X,p}_{n+1} $ we have 
\begin{align}
     \bE\Big[       \frac{1}{N} \sum_{j=1}^N |\hx_{n+1}^{i,N}- \hx_{n+1}^{j,N} |^{2p} \Big]
     = &
     \bE \Big[       \frac{1}{N} \sum_{j=1}^N  \Big|    
    \Big( Y_{n}^{i,\star,N}
            + b(t_n,Y_{n}^{i,\star,N},\hm^{Y,N}_n) h
            +\sigma(t_n,Y_{n}^{i,\star,N},\hm^{Y,N}_n) \Delta W_{n}^i  \Big)
    \nonumber
    \\
    \label{eq: proof of p-moment bound, aux 2}
    &\quad 
    -\Big( Y_{n}^{j,\star,N}
            + b(t_n,Y_{n}^{j,\star,N},\hm^{Y,N}_n) h
            +\sigma(t_n,Y_{n}^{j,\star,N},\hm^{Y,N}_n) \Delta W_{n}^j  \Big)\Big|^{2p}  \Big].
\end{align}
Introduce the extra (local) notation for $G^{i,j,n}_1$, $G^{i,j,n}_2$ and $G^{i,j,n}_3$ as 
\begin{align*}
    G^{i,j,n}_1&=
    Y_{n}^{i,\star,N}-Y_{n}^{j,\star,N},
    \quad
    G^{i,j,n}_2=  b(t_n,Y_{n}^{i,\star,N},\hm^{Y,N}_n) h-b(t_n,Y_{n}^{j,\star,N},\hm^{Y,N}_n) h, 
    \\
    G^{i,j,n}_3&= \sigma(t_n,Y_{n}^{i,\star,N},\hm^{Y,N}_n) \Delta W_{n}^i-\sigma(t_n,Y_{n}^{j,\star,N},\hm^{Y,N}_n) \Delta W_{n}^j .
\end{align*}
For $a+b+c=2p$, $a< 2p$, $a,b,c \in \mathbb{N}$, by Assumption \ref{Ass:Monotone Assumption},  Young's inequality, Jensen's inequality, Proposition \ref{prop:auxilary for multipy of r.v.} and the fact that the Brownian increments are independent, the particles are conditionally independent and identically distributed, for \eqref{eq: proof of p-moment bound, aux 2}, we have 
\begin{align*}
    \bE \Big[ \frac{C}{N} \sum_{j=1}^N  |G^{i,j,n}_1|^a |G^{i,j,n}_2|^b |G^{i,j,n}_3|^c        \Big]
    \le
        \bE \big[ | Y_{n}^{i,\star,N} |^{2p}          \big]  Ch 
        \le H^{Y,p}_{n} Ch.
\end{align*}
Thus, for the first term of $H^{X,p}_{n+1} $, we conclude that 
\begin{align}
     4^p\bE\Big[       \frac{1}{N} \sum_{j=1}^N |\hx_{n+1}^{i,N}- \hx_{n+1}^{j,N} |^{2p} \Big]
      \label{eq:moment bounde: first term estimate}
    &\le 
    4^p\bE\Big[       \frac{1}{N} \sum_{j=1}^N |Y_{n}^{i,\star,N}-Y_{n}^{j,\star,N} |^{2p} \Big]+  H^{Y,p}_{n} Ch. 
\end{align}
For the second term of $H^{X,p}_{n+1} $ we have 
\begin{align*}
    &\bE \Big[     \Big|\frac{1}{N} \sum_{j=1}^N (1+|\hx_{n+1}^{j,N}|^2)\Big|^{p}    \Big] 
    =
    \bE \bigg[   \Big|    \frac{1}{N} \sum_{j=1}^N     \Big[    1+   
    \Big( Y_{n}^{j,\star,N}
            + b(t_n,Y_{n}^{j,\star,N},\hm^{Y,N}_n) h
            +\sigma(t_n,Y_{n}^{j,\star,N},\hm^{Y,N}_n) \Delta W_{n}^j  \Big)^2 \Big] \Big|^{p} \bigg]. 
\end{align*}
Set the following (extra local) notation
\begin{align*}
    G_4^{n}&=
     \frac{1}{N} \sum_{j=1}^N     (1+|Y_{n}^{j,\star,N}|^2),
     \quad
     G_5^{n}= \frac{1}{N} \sum_{j=1}^N 
    \Big\langle 2 Y_{n}^{j,\star,N}+
    \sigma(t_n,Y_{n}^{j,\star,N},\hm^{Y,N}_n) \Delta W_{n}^j  ,
    \sigma(t_n,Y_{n}^{j,\star,N},\hm^{Y,N}_n) \Delta W_{n}^j\Big\rangle, 
    \\
    G_6^{n}&=  \frac{1}{N} \sum_{j=1}^N   
      \Big \langle 2Y_{n}^{j,\star,N}+
    b(t_n,Y_{n}^{j,\star,N},\hm^{Y,N}_n)h+2\sigma(t_n,Y_{n}^{j,\star,N},\hm^{Y,N}_n) \Delta W_{n}^j  ,b(t_n,Y_{n}^{j,\star,N},\hm^{Y,N}_n)h\Big\rangle. 
\end{align*} 
We have once again using similar arguments as before, by Young's inequality, Jensen's inequality,  Proposition \ref{prop:auxilary for multipy of r.v.}, that the particles are conditionally independent and identically distributed, the independence property of the Brownian increments, the Lipschitz property for $b$ and $\sigma$, and using the fact that for $l_1>l_2>1,~|x|^{l_2}\le 1+|x|^{l_1}$  we have 
\begin{align*}
    \bE \big[ |G_4^{n}|^a |G_5^{n}|^b |G_6^{n}|^c        \big]
    \le
        \bE \big[ | Y_{n}^{i,\star,N} |^{2p}+1          \big]  Ch
        \le H^{Y,p}_{n} Ch.
\end{align*}
Thus, for the second term of $~H^{X,p}_{n+1} $, we conclude that  
\begin{align}
\label{eq:moment bounde: second term estimate}
     4^p \bE \Big[     \Big|\frac{1}{N} \sum_{j=1}^N ( 1+  |\hx_{n+1}^{j,N}|^2) \Big|^{p}     \Big]
    &\le 
    4^p \bE \Big[     \Big|\frac{1}{N} \sum_{j=1}^N  (   1+ |Y_{n}^{j,\star,N}|^{2} )   \Big|^{p}  \Big]+  H^{Y,p}_{n} Ch .
\end{align}
Plug \eqref{eq:moment bounde: first term estimate} and \eqref{eq:moment bounde: second term estimate} into $H^{X,p}_{n+1} $ we have 
\begin{align*}
    H^{X,p}_{n+1}&
    =4^p \bE \Big[       \frac{1}{N} \sum_{j=1}^N |\hx_{n+1}^{i,N}- \hx_{n+1}^{j,N} |^{2p} \Big]
    +
   4^p \bE \Big[     \Big|\frac{1}{N} \sum_{j=1}^N ( 1+  |\hx_{n+1}^{j,N}|^2) \Big|^{p}     \Big]+1
    \\
    &\le
    4^p\bE\Big[       \frac{1}{N} \sum_{j=1}^N |Y_{n}^{i,\star,N}-Y_{n}^{j,\star,N} |^{2p} \Big]
    +
    4^p \bE \Big[     \Big|\frac{1}{N} \sum_{j=1}^N  (   1+ |Y_{n}^{j,\star,N}|^{2} )   \Big|^{p}  \Big]+1+  H^{Y,p}_{n} Ch  
    \le H^{Y,p}_{n}(1+Ch).
\end{align*}
Thus finally, for all $n\in \llbracket 0,M-1 \rrbracket,~i\in \llbracket 1,N \rrbracket $, by backward induction collecting all the results above, since $m\ge 2p$, where $m$ is defined in \eqref{Eq:General MVSDE}, we have (for some $C>0$ independent of $h,N,M$)
\begin{align*}
    \bE\big[ |\hx_{n+1}^{i,N} |^{2p}  \big] 
    &\le
    H^{X,p}_{n+1}
    \le H^{Y,p}_{n}(1+Ch) \le  H^{X,p}_{n}(1+Ch)^2 
    \le
    \cdots \le  H^{X,p}_{0} e^{CT} \le C \bE\big[ |\hx_{0}^{i,N} |^{2p}  \big] +C <\infty.
\end{align*}
Similar argument yields the result for $ \bE\big[ | Y_{n}^{i,\star,N} |^{2p}  \big]  $.
\end{proof}

\subsubsection*{Proof of the Theorem \ref{theorem：moment bound for the big theorm time extensions }} 
\begin{proof}[Proof of the Theorem \ref{theorem：moment bound for the big theorm time extensions }]
Under the same assumptions and notations of Theorem \ref{theorem: discrete moment bound}, one can apply arguments similar to those used in \cite[Proposition 4.6]{2021SSM} to obtain the result.
\end{proof}

The final result of this section concerns the incremental (in time) moment bounds of $\hx^{i,N}$. This result is in preparation for the next section.
\begin{proposition}
\label{prop:SSTM: local increment error00}
Under same assumptions and notations of Theorem \ref{theorem：moment bound for the big theorm time extensions }, there exists a constant $C>0$     independent of $h,N,M$ (but depending on $T$ and $m$)    such that for any $p\geq 2$ satisfy $m\ge (q+1)p$, where $m$ is defined in \eqref{Eq:General MVSDE}, $q$ is defined in Assumption \ref{Ass:Monotone Assumption}, we have 
\begin{align}
    \sup_{i\in \llbracket 1,N \rrbracket}
    \sup_{0\le t \le T} \bE\big[ |\hx_{t}^{i,N}-\hx_{\kt}^{i,N} |^p \big]
    \le C h^{\frac p2}.
\end{align}
\end{proposition}
\begin{proof}
Under Assumption \ref{Ass:Monotone Assumption}, and carefully applying Young's and Jensen's inequality, one can argue similarly as to \cite[Proposition 4.7]{2021SSM} and obtain the result (we omit further details).
\end{proof}

\subsection{Proof of Theorem \ref{theorem:SSM: strong error 2}, the uniform convergence result in path-space}
\label{subsection:Proof of strong error 2}

We now prove Theorem \ref{theorem:SSM: strong error 2}. 

\begin{proof}[Proof of Theorem \ref{theorem:SSM: strong error 2}]Let Assumption \ref{Ass:Monotone Assumption} hold. Let $i\in\llbracket 1,N\rrbracket$, $t\in[0,T]$, suppose  $m \ge \max\{4q+4,2+q+q/\varepsilon \}$, where $X^i_0\in L^m_0(\bR^d)$, $q$ is as given in Assumption \ref{Ass:Monotone Assumption}. From \eqref{eq:momentboundParticiInteractingSystem} and \eqref{eq:SSTM:momentbound for split-step time extension process00}, both process $X^{i,N}$ and $\hx^{i,N}$ have sufficient bounded moments for the following proof. 
Define $\Delta X^i:=X^{i,N}-\hx^{i,N}$. It\^o's formula applied to $|X_{t}^{i,N}-\hx_{t}^{i,N}|^2=|\Delta X^i_t|^2$ yields  
\begin{align}
    |\Delta X^i_t|^2 \label{eq:Conv of big theom:term 11 } 
    =&2\int_0^t \Big\langle  v(X_{s}^{i,N},\mu_{s}^{X,N})-v(Y_{\kappa(s)}^{i,\star,N},\hm^{Y,N}_{\kappa(s)}) ,\Delta X^i_s \Big\rangle \dd s  
    \\\label{eq:Conv of big theom:term 22 }
        &+2  \int_0^t \Big\langle  b(s,X_{s}^{i,N},\mu_{s}^{X,N})-b(\kappa(s),Y_{\kappa(s)}^{i,\star,N},\hm^{Y,N}_{\kappa(s)}),\Delta X^i_s\Big\rangle \dd s
    \\    \label{eq:Conv of big theom:term 33 }
    &+  \int_0^t \Big| \sigma(s,X_{s}^{i,N},\mu_{s}^{X,N})-\sigma(\kappa(s),Y_{\kappa(s)}^{i,\star,N},\hm^{Y,N}_{\kappa(s)})\Big|^2 \dd s
    \\\label{eq:Conv of big theom:term 44 }
    &+2\int_0^t \Big\langle \Delta X^i_s,\Big(\sigma(s,X_{s}^{i,N},\mu_{s}^{X,N}) - \sigma(\kappa(s),Y_{\kappa(s)}^{i,\star,N},\hm^{Y,N}_{\kappa(s)})\Big) \dd W^i_s\Big\rangle.
\end{align}
We analyse the above terms one by one and will take supremum over time with expectation. For \eqref{eq:Conv of big theom:term 11 }, 
\begin{align}
\nonumber
     \big\langle  v(X_{s}^{i,N},\mu_{s}^{X,N})&-v(Y_{\kappa(s)}^{i,\star,N},\hm^{Y,N}_{\kappa(s)}) ,\Delta X^i_s \big\rangle
    \\
    \label{eq:se2:v-v term}
    =&
  \big\langle  v(X_{s}^{i,N},\mu_{s}^{X,N})-v(\hx_{s}^{i,N},\hm^{X,N}_{s}),\Delta X^i_s \big\rangle
     +    
    \big\langle v(\hx_{s}^{i,N},\hm^{X,N}_{s})-v(Y_{\kappa(s)}^{i,\star,N},\hm^{Y,N}_{\kappa(s)}),\Delta X^i_s \big\rangle. 
\end{align}
For the first term above, by  Assumption \ref{Ass:Monotone Assumption} and using Remark \ref{remark:ImpliedProperties} 
\begin{align}
\nonumber
    \bE\Big[  \sup_{0\le t \le T} 
    &
    \int_0^t  \Big\langle  v(X_{s}^{i,N},\mu_{s}^{X,N})-v(\hx_{s}^{i,N},\hm^{X,N}_{s}),\Delta X^i_s \Big\rangle \dd s \Big] 
    \\\nonumber
    \le&
    \bE\Big[  \int_0^T    
    \frac{C}{N} \sum_{j=1}^N 
    \Big|
     f(X_{s}^{i,N}-X_{s}^{j,N})-f(\hx_{s}^{i,N}-\hx_{s}^{j,N}) \Big| |\Delta X^i_s |
    \dd s \Big]
    \\\nonumber
     &  + 
    \bE\Big[  \sup_{0\le t \le T} 
    \int_0^t  \Big\langle  u(X_{s}^{i,N},\mu_{s}^{X,N})-u(\hx_{s}^{i,N},\hm^{X,N}_{s}),\Delta X^i_s \Big\rangle \dd s \Big]
    \\
    \label{eq:se2:special term}
    \le&
    \bE\Big[  \int_0^T      \frac{C}{N} \sum_{j=1}^N 
    \Big\{
    \Big( 1+|X_{s}^{i,N}-X_{s}^{j,N} |^q +|\hx_{s}^{i,N}-\hx_{s}^{j,N} |^q   \Big) 
     | \Delta X^i_s-\Delta X^j_s |    |\Delta X^i_s | \Big\} \dd s \Big]
     \\\nonumber
     &+
     \bE\Big[  \int_0^T    \Big(  \widehat L_u |\Delta X^i_s |^2 
     +    \frac{ L_{\tilde{u} }	}{2N}  \sum_{j=1}^N  |\Delta X^j_s |^2  \Big)   \dd s \Big].
\end{align}
To deal with \eqref{eq:se2:special term}, using the following notations, for all $i,j \in \llbracket 1,N \rrbracket$, 
\begin{align}
\nonumber
    G_7^{i,j,s}=\Big( 1+|X_{s}^{i,N}-X_{s}^{j,N} |^q +|\hx_{s}^{i,N}-\hx_{s}^{j,N} |^q   \Big)
    \quad\textrm{and}\quad 
    G_8^{i,j,s}= | \Delta X^i_s-\Delta X^j_s |    |\Delta X^i_s |.
\end{align}
 The combination of $G_7^{i,j,s}$ and $G_8^{i,j,s}$ makes it difficult to obtain a domination via  $ |\Delta X^i_s |^2$, we overcome this by applying Chebyshev's inequality as follows.
\color{black}
The indicator function is denoted as $\1_{\{\Omega\}}$. Recall the moment  bound results on $X,\hx$ in \eqref{eq:momentboundParticiInteractingSystem} and \eqref{eq:SSTM:momentbound for split-step time extension process00} respectively. Now, using Theorem \ref{theorem:SSM: strong error 1}, Proposition \ref{prop:auxilary for multipy of r.v.} and Young's inequality, we have 
\begin{align}
\label{eq:proof strong errr2:ep 1}
\bE \big[   G_7^{i,j,s}G_8^{i,j,s}   \big]
&=
\bE \Big[   G_7^{i,j,s}G_8^{i,j,s} (\1_{\{ G_7^{i,j,s}<M^{\varepsilon} \}}) \Big]
+\bE \Big[   G_7^{i,j,s}G_8^{i,j,s} (\1_{\{ G_7^{i,j,s}\ge M^{\varepsilon} \}}) \Big]
\\
\nonumber 
&\le
\bE \big[ M^{\varepsilon}  G_8^{i,j,s}  \big]
+
\bE \Big[ \frac{ |  G_7^{i,j,s}|^{1/\varepsilon} }{ M}  G_7^{i,j,s} G_8^{i,j,s}  \Big]  
\le
2 \bE \big[ M^{\varepsilon} |\Delta X^i_s |^2 \big]
+
 h \bE \big[  |  G_7^{i,j,s}|^{1/\varepsilon}  G_7^{i,j,s} G_8^{i,j,s}  \big] 
\\
\label{eq:proof strong errr2:ep 2}
&
\le C h^{1-\varepsilon} +h C\Big(1+\bE \Big[ |X_{s}^{i,N} |^{2+q+q/\varepsilon}  +|\hx_{s}^{i,N} |^{2+q+q/\varepsilon} \Big]  \Big)
\le C h^{1-\varepsilon},
\end{align}
where  for the last inequality, we used that the particles are identically distributed and there are sufficiently high  bounded moments available for the process since $m\ge 2+q+q/\varepsilon$.  



Thus, for the first term in \eqref{eq:se2:v-v term} and using that the particles are identically distributed, we conclude that 
\begin{align}
      \bE\Big[  \sup_{0\le t \le T} \int_0^t  \Big\langle  v(X_{s}^{i,N},\mu_{s}^{X,N})-v(\hx_{s}^{i,N},\hm^{X,N}_{s}),\Delta X^i_s \Big\rangle \dd s \Big]
    \le 
    C\bE\Big[ \int_0^T |\Delta X^i_s|^2  \dd s \Big]+
     C h^{1-\varepsilon}.
\end{align}
For the second term in \eqref{eq:se2:v-v term}, under Assumption \ref{Ass:Monotone Assumption}, using Young's inequality, Jensen's inequality, and Proposition \ref{prop:SSTM: local increment error00} we have 
\begin{align}
\label{eq:se2: v-v term 2 sup form}
\bE \Big[& \sup_{0\le t \le T} \int_0^t    
    \Big\langle v(\hx_{s}^{i,N},\hm^{X,N}_{s})-v(Y_{\kappa(s)}^{i,\star,N},\hm^{Y,N}_{\kappa(s)}),\Delta X^i_s \Big\rangle \dd s \Big]
    \\ \label{eq:se2:I2 source}
    =&
    \bE\Big[  \sup_{0\le t \le T} \int_0^t    
    \Big\langle u(\hx_{s}^{i,N},\hm^{X,N}_{s})-u(Y_{\kappa(s)}^{i,\star,N},\hm^{Y,N}_{\kappa(s)}),\Delta X^i_s \Big\rangle \dd s \Big]
    \\ \label{eq:se2:I3 source}
    &
    \quad + \bE\Big[ \sup_{0\le t \le T} \int_0^t        \frac{1}{N} \sum_{j=1}^N \langle  f(\hx_{s}^{i,N}-\hx_{s}^{j,N})-f(Y_{\kappa(s)}^{i,\star,N}-Y_{\kappa(s)}^{j,\star,N}),\Delta X^i_s \Big\rangle \dd s \Big]
    \\
    \nonumber 
    \le 
    & 
    \bE\Big[ \int_0^T |\Delta X^i_s|^2  \dd s \Big] +I_2+I_3. 
\end{align}
For $I_2$ (given by the domination of \eqref{eq:se2:I2 source}), by Assumption \ref{Ass:Monotone Assumption}, Young's inequality and Cauchy-Schwarz inequality  
 \begin{align*}
    I_2=&L_{\hat{u}}  \bE\Big[
     \int_0^T
     \Big(1+ |\hx_{s}^{i,N}|^{q} + |Y_{\kappa(s)}^{i,\star,N}|^{q}\Big)^2
     |\hx_{s}^{i,N}-Y_{\kappa(s)}^{i,\star,N}|^2 \Big]
     \dd s   
    \\\nonumber
    \le
    &
    C \int_0^T
    \sqrt{   \bE\Big[  \Big(1+|\hx_{s}^{i,N}|^{2q}+|Y_{\kappa(s)}^{i,\star,N}|^{2q}\Big)^2 
     \Big]
     \bE\big[ 
     |\hx_{s}^{j,N}-Y_{\kappa(s)}^{j,\star,N}|^4
     \big]
     } \dd s.
\end{align*}
For $I_3$ (given by the domination of \eqref{eq:se2:I3 source} after extracting the $|\Delta X^i|$ term), by Assumption \ref{Ass:Monotone Assumption}, Young's inequality and Cauchy-Schwarz inequality 
\begin{align*}
    I_3=&\frac{C L_{\hat{f}}}{N} \sum_{j=1}^N \bE\Big[
     \int_0^T
     \Big(1+ |\hx_{s}^{i,N}-\hx_{s}^{j,N}|^{q} + |Y_{\kappa(s)}^{i,\star,N}-Y_{\kappa(s)}^{j,\star,N}|^{q}\Big)^2~
     \big|(\hx_{s}^{i,N}-\hx_{s}^{j,N})-(Y_{\kappa(s)}^{i,\star,N}-Y_{\kappa(s)}^{j,\star,N})\big|^2 \Big]
     \dd s   
    \\\nonumber
    \le&
     \frac{C}{N} \sum_{j=1}^N \int_0^T
    \sqrt{   \bE\Big[  \Big(1+|\hx_{s}^{j,N}|^{2q}+|\hx_{s}^{i,N}|^{2q}+|Y_{\kappa(s)}^{i,\star,N}|^{2q}+|Y_{\kappa(s)}^{j,\star,N}|^{2q}\Big)^2 
     \Big]
     \bE\big[ 
     |\hx_{s}^{j,N}-Y_{\kappa(s)}^{j,\star,N}|^4
     \big]
     } \dd s.
\end{align*}
By \eqref{eq:SSTM:scheme 1}, Assumption \ref{Ass:Monotone Assumption}, Young's inequality, Jensen's inequality, since $m\ge 4q+4$, and by     Theorem \ref{theorem: discrete moment bound},  we have 
\begin{align*}
    \bE\big[ &| \hx_{\kappa(s)}^{i,N}-Y_{\kappa(s)}^{i,\star,N}|^4
     \big]
     =
     \bE\big[ | h v(Y_{\kappa(s)}^{i,\star,N},\hm^{Y,N}_\kappa(s)) |^4
     \big] 
     \\
     &
     \le 
     C h^4\bE\big[ | u(Y_{\kappa(s)}^{i,\star,N},\hm^{Y,N}_\kappa(s)) |^4
     \big]
     +
     \frac{Ch^4}{N}
     \sum_{j=1}^N\bE\big[ |f(Y_{\kappa(s)}^{i,\star,N}-Y_{\kappa(s)}^{j,\star,N}) |^4
     \big] 
     \\
     &\le
      C h^4\bE\Big[ 1+ |Y_{\kappa(s)}^{i,\star,N}|^{4q+4}
      + \frac{1}{N} \sum_{j=1}^N |Y_{\kappa(s)}^{j,\star,N}|^{4}
     \Big]
     +
     \frac{Ch^4}{N} \sum_{j=1}^N \bE\Big[ (1+ |Y_{\kappa(s)}^{i,\star,N}-Y_{\kappa(s)}^{j,\star,N}|^{4q} ) |Y_{\kappa(s)}^{i,\star,N}-Y_{\kappa(s)}^{j,\star,N}|^4 \Big]
     \\
     &\le  \frac{Ch^4}{N} \sum_{j=1}^N \bE\big[ 1+ |Y_{\kappa(s)}^{j,\star,N}|^{4q+4}  \big]
     \le C h^4    .
\end{align*}  
Using this inequality in combination with Proposition \ref{prop:SSTM: local increment error00} allows us to conclude that 
\begin{align}
\label{eq:diff between X hat t and Y star kttt}
     \bE\big[ |\hx_{s}^{j,N}-Y_{\kappa(s)}^{j,\star,N}|^4
     \big]
     \le
      C\bE\Big[ |\hx_{s}^{j,N}-\hx_{\kappa(s)}^{j,N}|^4+| \hx_{\kappa(s)}^{j,N}-Y_{\kappa(s)}^{j,\star,N}|^4
     \Big]\le Ch^2.
\end{align}
Thus, for \eqref{eq:se2:v-v term} injected back in \eqref{eq:Conv of big theom:term 11 }, take supremum and expectation, and collecting all the necessary results above, we reach  
\begin{align}
\label{eq:se2:term 1 final result}
    &\bE \Big[ \sup_{0\le t \le T} \int_0^t     
    \Big\langle v(X_{s}^{i,N},\mu^{N}_{s})-v(Y_{\kappa(s)}^{i,\star,N},\hm^{Y,N}_{\kappa(s)}),\Delta X^i_s \Big\rangle \dd s \Big]
    \le
     C\bE\Big[ \int_0^T |\Delta X^i_s|^2 \dd s \Big]+Ch^{1-\varepsilon} .
\end{align}
For the second term \eqref{eq:Conv of big theom:term 22 }, the calculation is similar as in \cite[Proof of Proposition 4.9]{2021SSM}, we conclude that %
\begin{align}
    \label{eq:se2:term 2 final result}
    &\bE \Big[ \sup_{0\le t \le T} \int_0^t     \Big\langle  b(s,X_{s}^{i,N},\mu_{s}^{X,N})-b(\kappa(s),Y_{\kappa(s)}^{i,\star,N},\hm^{Y,N}_{\kappa(s)}) ,\Delta X^i_s \Big\rangle  \dd s  \Big] 
    \le
    C h +  C  \bE \Big[  \int_0^T  |\Delta X^i_s|^2 \dd s \Big].
\end{align}
Similarly, for the third term \eqref{eq:Conv of big theom:term 33 } (these are just Lipschitz terms), we have 
\begin{align}
 \label{eq:se2:term 3 final result}
    \bE\Big[ \sup_{0\le t \le T} \int_0^t  \Big| \sigma(s,X_{s}^{i,N},\mu_{s}^{X,N})-\sigma(\kappa(s),Y_{\kappa(s)}^{i,\star,N},\hm^{Y,N}_{\kappa(s)})\Big|^2 \dd s \Big]
    \le
    C h + C  \bE \Big[  \int_0^T  |\Delta X^i_s|^2 \dd s \Big] .
\end{align}
Consider the last term \eqref{eq:Conv of big theom:term 44 } -- this is a Lipschitz term and dealt with similarly to \cite[Proof of Proposition 4.9]{2021SSM}. 
Using the Burkholder-Davis-Gundy's, Jensen's and Cauchy-Schwarz inequality, and the above results, 
\begin{align}
\label{eq:se2:term 4 final result}
&\bE\Big[ \sup_{0\le t \le T} \int_0^t   \Big\langle \Delta X^i_s,\Big(\sigma(s,X_{s}^{i,N},\mu_{s}^{X,N}) - \sigma(\kappa(s),Y_{\kappa(s)}^{i,\star,N},\hm^{Y,N}_{\kappa(s)})\Big) \dd W^i_s\Big\rangle  ~\Big] 
\\
& \nonumber
\le 
\frac{1}{4}\bE\big[~ \sup_{0\le t \le T }  |\Delta X^i_t|^2 \big]
+
\bE\Big[\int_0^T  \Big|\sigma(s,X_{s}^{i,N},\mu_{s}^{X,N}) - \sigma(\kappa(s),Y_{\kappa(s)}^{i,\star,N},\hm^{Y,N}_{\kappa(s)})\Big|^2\dd s~\Big]  .
\end{align}
Again, gathering all the above results  \eqref{eq:se2:term 1 final result},
\eqref{eq:se2:term 2 final result},
\eqref{eq:se2:term 3 final result}, 
and 
\eqref{eq:se2:term 4 final result}, plugging them back into \eqref{eq:Conv of big theom:term 11 }, after taking supremum over $t\in[0,T]$ and expectation, for all $i\in \llbracket 1,N \rrbracket $ we have 
\begin{align*}
  \bE\big[ \sup_{0\le t \le T}|\Delta X^i_t|^2  ~\big] 
&\le    
Ch^{1-\varepsilon}+ C\bE \Big[~ \int_0^T  \sup_{0\le u \le s}|\Delta X^i_u|^2   \dd s    ~\Big]+\frac{1}{2} \bE \big[\sup_{0\le t \le T }  |\Delta X^i_t|^2\big]
\\
&
\le    
Ch^{1-\varepsilon} +C\int_0^T   \bE \big[\sup_{0\le u \le s} |\Delta X^i_u|^2 ~\big] \dd s .
\end{align*}
Gr\"onwall's lemma delivers the final result after taking supremum over $i\in \llbracket 1,N \rrbracket $.
\end{proof}


\subsection{Discussion on the granular media type equation}\label{subsection:Proof of strong error 3 GM}

Throughout $C>0$ denotes a constant always independent of $h,N,M$ but possibly depending on $T$ and $m$. 

\begin{proof}[Proof of Proposition \ref{Prop:Propagation of Chaos}]
Recall the proof of \eqref{eq:Conv of big theom:term 11 } in Section \ref{subsection:Proof of strong error 2}. Under Assumption \ref{Ass:GM Assumption}, for all $i\in\llbracket 1,N\rrbracket$, $t\in[0,T]$, 
and using arguments similar to those of \eqref{eq:se2:v-v term} we have  
\begin{align}
    \nonumber
    \Delta X^i_t=&~X_{t}^{i,N}-\hx_{t}^{i,N} =\int_0^t         v(X_{s}^{i,N},\mu_{s}^{X,N})-v(Y_{\kappa(s)}^{i,\star,N},\hm^{Y,N}_{\kappa(s)})~  \dd s,
    \\
    \Rightarrow \bE \big[       |\Delta X^i_t|^2 \big] \label{eq:gm:proof: delta x term 1 } 
    \le&~
    2\int_0^t \bE \Big[       \Big\langle  v(X_{s}^{i,N},\mu_{s}^{X,N})-v(\hx_{s}^{i,N},\hm^{X,N}_{s}) ,\Delta X^i_s \Big\rangle \Big] \dd s  
    \\
    \label{eq:gm:proof: delta x term 2 }
    & ~+
    2\int_0^t \bE \Big[        \Big\langle   v(\hx_{s}^{i,N},\hm^{X,N}_{s})-v(Y_{\kappa(s)}^{i,\star,N},\hm^{Y,N}_{\kappa(s)}) ,\Delta X^i_s \Big\rangle \Big] \dd s .
\end{align}
For \eqref{eq:gm:proof: delta x term 1 }, arguing as in \eqref{eq:se1:y-y v-v}, Remark \ref{remark:OSL for the whole function / system V} and using that the particles are identically distributed, we have 
\begin{align}
    \bE &\Big[       \Big\langle  v(X_{s}^{i,N},\mu_{s}^{X,N})-v(\hx_{s}^{i,N},\hm^{X,N}_{s}) ,\Delta X^i_s \Big\rangle \Big]  
    \label{eq:gm:proof: delta x term 1 results} 
    \le 2L_f^+  \bE \big[       |\Delta X^i_s|^2 \big].
\end{align}
For \eqref{eq:gm:proof: delta x term 2 }, it is similar to the above, we have 
\begin{align}
      2\int_0^t \bE &\Big[        \Big\langle   v(\hx_{s}^{i,N},\hm^{X,N}_{s})-v(Y_{\kappa(s)}^{i,\star,N},\hm^{Y,N}_{\kappa(s)}) ,\Delta X^i_s \Big\rangle \Big] \dd s 
    \label{eq:gm:proof: delta x term 21} 
    =
      \frac{2}{N} \sum_{j=1}^N\int_0^t 
    \bE\Big[       
    \Big\langle f(\Delta^{X,i,j}_s) -f( \Delta^{Y,i,j}_\ks)
    ,\Delta X^i_s \Big\rangle \Big] \dd s, 
\end{align}
where we introduce the following handy notation (recall \eqref{eq:SSTM:scheme 1} and \eqref{eq: scheme continous extension in SDE form}) 
\begin{align}
\nonumber 
 \Delta^{X,i,j}_t&=\hx_{s}^{i,N}-\hx_{s}^{j,N},
    \qquad
    \Delta^{Y,i,j}_\ks=Y_{\kappa(s)}^{i,\star,N}-Y_{\kappa(s)}^{j,\star,N},
\\ \nonumber
     \Delta^{X,i,j}_s
    &=  \Delta^{X,i,j}_\ks+G_9^{i,j,s}(s-\ks)+G_{10}^{i,j,s},
    \qquad
    \label{eq:gm:G7 comes}
     \Delta^{Y,i,j}_\ks
    = \Delta^{X,i,j}_\ks+G_{9}^{i,j,s} h,
\\  
    G_{9}^{i,j,s}
    &= \Big(  v(Y_{\kappa(s)}^{i,\star,N},\hm^{Y,N}_{\kappa(s)})-v(Y_{\kappa(s)}^{j,\star,N},\hm^{Y,N}_{\kappa(s)})\Big)
    \quad\textrm{and}\quad 
    G_{10}^{i,j,s}
    =\sigma\Big(  (W_s^i- W_{\ks}^i)-  (W_s^j- W_{\ks}^j)  \Big).
\end{align}
We now proceed to estimate \eqref{eq:gm:proof: delta x term 21}. 
By the mean value theorem under Assumption \ref{Ass:GM Assumption}, for \eqref{eq:gm:proof: delta x term 21}, there exist $\rho_1,\rho_2 \in [0,1]$ such that \begin{align*}
      f(\Delta^{X,i,j}_s) 
    =& 
    f (\Delta^{X,i,j}_\ks)
    +
     \nabla f (\Delta^{X,i,j}_\ks) \Big( G_{9}^{i,j,s}(s-\ks)+G_{10}^{i,j,s} \Big) 
    + \int_{\Delta^{X,i,j}_\ks}^{\Delta^{X,i,j}_s} \Big( \nabla f(u) -\nabla f (\Delta^{X,i,j}_\ks)  \Big)  du
     \\
     =&
    f (\Delta^{X,i,j}_\ks)
    +
     \nabla f (\Delta^{X,i,j}_\ks) \Big( G_{9}^{i,j,s}(s-\ks)+G_{10}^{i,j,s} \Big) 
     \\
     &+ \Big( \nabla f\big(\Delta^{X,i,j}_\ks +\rho_1 ( G_{9}^{i,j,s}(s-\ks)+G_{10}^{i,j,s})  \big) -\nabla f (\Delta^{X,i,j}_\ks)  \Big)    
     \Big(     \Delta^{X,i,j}_s-\Delta^{X,i,j}_\ks       \Big)  , 
     \\
     f(\Delta^{Y,i,j}_\ks) 
    =& 
    f (\Delta^{X,i,j}_\ks)
    +
     \nabla f (\Delta^{X,i,j}_\ks) \Big( G_{9}^{i,j,s} h \Big)
        +  
     \Big( \nabla f\big(\Delta^{X,i,j}_\ks +\rho_2 ( G_{10}^{i,j,s}h)  \big) -\nabla f (\Delta^{X,i,j}_\ks)  \Big) 
     \Big(    \Delta^{Y,i,j}_\ks-\Delta^{X,i,j}_\ks       \Big). 
\end{align*}
 
Note that only $G_{10}$ contains the Brownian increments. From the above, there exists $\rho_{1,s},~\rho_{2,s}\in[0,1]$ for all $s\in[0,T]$, and by Young's inequality, we have 
\begin{align}
    \label{eq:gm:f delta term 00}
    \int_0^t 
    \bE\Big[   &   \Big\langle  f(\Delta^{X,i,j}_s) 
    -f(\Delta^{Y,i,j}_\ks)
    ,\Delta X^i_s \Big\rangle \Big] \dd s
    \\
    \label{eq:gm:f delta term 11}
    \le&
    \int_0^t 
    \bE\Big[      \Big\langle  \nabla f
    (\Delta^{X,i,j}_\ks) 
    \Big( G_{9}^{i,j,s}(s-h-\ks)+G_{10}^{i,j,s} \Big) 
    ,\Delta X^i_s \Big\rangle \Big] \dd s
    + 
    C \int_0^t 
    \bE\Big[     |\Delta X^i_s|^2  \Big] \dd s
    \\
    \label{eq:gm:f delta term 22}
    &+ C
    \int_0^t 
     \bE\Big[    \Big| \nabla f\Big(\Delta^{X,i,j}_\ks +\rho_{1,s} ( G_{9}^{i,j,s}(s-\ks)+G_{10}^{i,j,s})  \Big) -\nabla f (\Delta^{X,i,j}_\ks)  \Big|^2  
     \Big|    \Delta^{X,i,j}_s-\Delta^{X,i,j}_\ks      \Big|^2   \Big] \dd s
     \\
     \label{eq:gm:f delta term 33}
    &+ C
    \int_0^t 
     \bE\Big[    \Big|\nabla f\big(\Delta^{X,i,j}_\ks +\rho_{2,s} ( G_{9}^{i,j,s}h)  \big) -\nabla f (\Delta^{X,i,j}_\ks) \Big|^2  
     \Big|   \Delta^{Y,i,j}_\ks-\Delta^{X,i,j}_\ks         \Big|^2   \Big] \dd s.
\end{align}
For the first term of \eqref{eq:gm:f delta term 11}, by Young's inequality 
\begin{align}
    &\int_0^t 
    \bE\Big[      \Big\langle  \nabla f
    (\Delta^{X,i,j}_\ks) 
    \Big( G_{9}^{i,j,s}(s-h-\ks)+G_{10}^{i,j,s} \Big) 
    ,\Delta X^i_s \Big\rangle \Big] \dd s
    \\
    \label{eq:gm:f delta term 11-11}
    &\le
    C \int_0^t 
    \bE\Big[     |\Delta X^i_s|^2  \Big] \dd s
    + 
    C \int_0^t 
    \bE\Big[      \Big| \nabla f
    (\Delta^{X,i,j}_\ks) 
     G_{9}^{i,j,s}(s-h-\ks)\Big|^2 
     \Big] \dd s
     \\
     \label{eq:gm:f delta term 11-22}
          & \quad +
     \int_0^t 
    \bE\Big[      \Big\langle  
    \nabla f
    (\Delta^{X,i,j}_\ks)~ G_{10}^{i,j,s},\Delta X^i_s -\Delta X^i_{\ks} \Big\rangle \Big] \dd s
    +
    \int_0^t 
    \bE\Big[      \Big\langle  
    \nabla f
    (\Delta^{X,i,j}_\ks)~ G_{10}^{i,j,s},\Delta X^i_{\ks} \Big\rangle \Big] \dd s. 
\end{align} 
For the second term of \eqref{eq:gm:f delta term 11-11}, since $m\ge4q+2$, by Assumption \ref{Ass:GM Assumption} and Theorem \ref{theorem：moment bound for the big theorm time extensions }, using calculations similar to those in \eqref{eq:se1:f-f things} and Proposition \ref{prop:auxilary for multipy of r.v.},  we have 
\begin{align*}
    C \int_0^t 
    \bE\Big[      \Big| \nabla f
    (\Delta^{X,i,j}_\ks) 
     G_{9}^{i,j,s}(s-h-\ks)\Big|^2 
     \Big] \dd s
     \le C h^2
      \int_0^t 
    \bE\Big[   1+| \hx_{\ks}^{i,N}|^{4q+2} 
    + | Y_{\ks}^{i,\star,N}|^{4q+2}\Big] \dd s
    \le Ch^2.
\end{align*}
By Jensen's inequality and calculations close to those for $I_3$ in \eqref{eq:se2:I3 source}, since $m\ge4q+2$, we have 
\begin{align}
   \bE \big[ |\Delta X^i_t &-\Delta X^i_{\kt}  |^2 \big] 
    =\bE \Big[       \Big|\int_{\kt}^t \Big(  v(X_{s}^{i,N},\mu_{s}^{X,N})-v(Y_{\kappa(s)}^{i,\star,N},\hm^{Y,N}_{\kappa(s)})\Big) ~ \dd s  \Big|^2 \Big]
    \\ 
     \le& h \int_{\kt}^t 
     \frac{1}{N} \sum_{j=1}^N
     \bE \Big[        \Big| f(X_{s}^{i,N}-X_{s}^{i,N})-f(Y_{\kappa(s)}^{i,\star,N}-Y_{\kappa(s)}^{i,\star,N})\Big|^2\Big] ~ \dd s
     \le Ch^3.
\end{align}
Thus, for the first term of \eqref{eq:gm:f delta term 11-22}, by Cauchy-Schwarz inequality and the properties of the Brownian increment 
\begin{align*} 
   \int_0^t 
    \bE\Big[   &   \Big\langle  
    \nabla f
    (\Delta^{X,i,j}_\ks)~ G_{10}^{i,j,s},\Delta X^i_s -\Delta X^i_{\ks} \Big\rangle \Big] \dd s 
     \leq
    \int_0^t  
    \sqrt{\bE \big[  \big| 
    \nabla f     (\Delta^{X,i,j}_\ks)~ G_{10}^{i,j,s} 
    \big|^2     \big]  }
    \sqrt{\bE \big[  \big|    
    \Delta X^i_s -\Delta X^i_{\ks}
    \big|^2 \big] } 
     \dd s
     \le 
     Ch^2.
\end{align*}
For the second term of \eqref{eq:gm:f delta term 11-22}, since $ G_{10}^{i,j,s}$ of \eqref{eq:gm:G7 comes} is conditionally independent of $\Delta^{X,i,j}_\ks$ and $\Delta X^{i}_\ks $ (and contains the Brownian increments), the tower property yields
\begin{align}
    \label{eq:gm:f delta term 11-22 result 2 }
    \int_0^t 
    \bE\Big[   &   \Big\langle  
    \nabla f
    (\Delta^{X,i,j}_\ks)~ G_{10}^{i,j,s},\Delta X^i_{\ks} \Big\rangle \Big] \dd s
    = 0. 
\end{align}
Thus, plugging the above results back into \eqref{eq:gm:f delta term 11}, we conclude that 
\begin{align}
     \label{eq:gm:f delta term 11-result}
     \int_0^t 
    \bE\Big[      \Big\langle  \nabla f
    (\Delta^{X,i,j}_\ks) 
    \Big( G_{9}^{i,j,s}(s-h-\ks)+G_{10}^{i,j,s} \Big) 
    ,\Delta X^i_s \Big\rangle \Big] \dd s \le Ch^2.
\end{align}
For  \eqref{eq:gm:f delta term 22}, by Assumption \ref{Ass:GM Assumption}, Cauchy-Schwarz inequality and the properties of the Brownian increment, and the condition $m\ge \max\{8q,~4q+4\}$  
\begin{align*}
    &\bE\Big[    \big| \nabla f\big(\Delta^{X,i,j}_\ks +\rho_{1,s} ( G_{9}^{i,j,s}(s-\ks)+G_{10}^{i,j,s})  \big) -\nabla f \big(\Delta^{X,i,j}_\ks\big)  \big|^4  \Big]
    \\
    &\le C \bE\Big[ \big| \Big(  1
    + \big|\Delta^{X,i,j}_\ks 
    + \rho_{1,s} \big( G_{9}^{i,j,s}(s-\ks)+G_{10}^{i,j,s}\big)  \big|^{q-1} +  \big|\Delta^{X,i,j}_\ks\big|^{q-1}\Big) 
    \big| \rho_{1,s} \big( G_{9}^{i,j,s}(s-\ks)+G_{10}^{i,j,s}\big)   \big|^4  \Big] \le  Ch^2,
 \end{align*}
and 
 \begin{align*}
    &\bE\big[    \big|    \Delta^{X,i,j}_s-\Delta^{X,i,j}_\ks \big|^4 \big]
    \le 
    C \bE\Big[ \big| \big( G_{9}^{i,j,s}(s-\ks)+G_{10}^{i,j,s}\big)\big|^4 \Big] 
    \le  Ch^2.
\end{align*}
Thus, using  Cauchy-Schwarz inequality again and the results above we conclude that 
\begin{align}
    \label{eq:gm:f delta term 22-result}
    \int_0^t 
     \bE\Big[    \big| \nabla f\Big(\Delta^{X,i,j}_\ks 
     +\rho_{1,s} \big( G_{9}^{i,j,s}(s-\ks)+G_{10}^{i,j,s}\big)  \Big) -\nabla f (\Delta^{X,i,j}_\ks)  \big|^2  
     \big|    \Delta^{X,i,j}_s-\Delta^{X,i,j}_\ks      \big|^2   \Big] \dd s
     \le Ch^2  .
\end{align}
For \eqref{eq:gm:f delta term 33}, recall \eqref{eq:gm:G7 comes}. Similarly to above, by assumption $m \ge 4q+2$  and hence 
\begin{align}
    \label{eq:gm:f delta term 33-result}
     \int_0^t 
     \bE\Big[  ~  \big|\nabla f\big( \Delta^{X,i,j}_\ks +\rho_{2,s}  G_{9}^{i,j,s}h  \big) -\nabla f \big(\Delta^{X,i,j}_\ks\big) \big|^2  ~
     \big|  G_{9}^{i,j,s} h \big|^2  ~ \Big] \dd s  \le Ch^2.
\end{align}
Thus, plugging \eqref{eq:gm:f delta term 11-result}, \eqref{eq:gm:f delta term 22-result} and \eqref{eq:gm:f delta term 33-result} back into \eqref{eq:gm:f delta term 00}, yields
\begin{align}
    \label{eq:gm:proof: delta x term 2 results}
    \int_0^t 
    \bE\Big[      \Big\langle  f(\Delta^{X,i,j}_s) 
    -f(\Delta^{Y,i,j}_\ks)
    ,\Delta X^i_s \Big\rangle \Big] \dd s
    &\le
    Ch^2+
    C \int_0^t 
    \bE\big[     |\Delta X^i_s|^2  \big] \dd s.
\end{align}
Plug the above result and \eqref{eq:gm:proof: delta x term 1 results} back to \eqref{eq:gm:proof: delta x term 1 }, we conclude that, for  all $i\in\llbracket 1,N\rrbracket$, $t\in[0,T]$
\begin{align}
    \bE \big[       |\Delta X^i_t|^2 \big] 
    \le&
    C\int_0^t \bE \big[   |\Delta X^i_s|^2     \big] \dd s  +Ch^2.
\end{align}
Gr\"onwall's lemma delivers the final result after taking supremum over $i\in \llbracket 1,N \rrbracket $.
\end{proof}

\appendix

%
%
%
%
%
%
\section{Well-posedness of the particle system and the PoC --  Proposition \ref{Prop:Propagation of Chaos} }
\label{appendix: proof of moment bound for PC}

The Propagation of chaos result \eqref{eq:poc result} follows directly from \cite[Theorem 3.14]{adams2020large}. The gap we close is the well-posedness result for the interacting particle system and the moment bound result.    Note that throughout $C>0$ is a constant always independent of $h,N,M$ but possibly depending on $T$ and $m$.

\begin{proof}[Proof of Proposition \ref{Prop:Propagation of Chaos}]  
We start by interpreting the interacting particle system \eqref{Eq:MV-SDE Propagation} as a single SDE in $\bR^{Nd}$. In Remark \ref{remark:OSL for the whole function / system V} we show that, as a system in $\bR^{Nd}$, the function $V$ (see \eqref{eq:Define-Function-V} and \eqref{Eq:General MVSDE shape of v}) satisfies a one-sided Lipschitz condition (as a map in $\bR^{Nd}$). Thus: (i) the drift term of the whole system also satisfies one-sided Lipschitz condition as $b$ satisfies a uniformly Lipschitz condition by $(\mathbf{A}^b)$; (ii) the diffusion coefficient satisfies a Lipschitz condition (by $(\mathbf{A}^\sigma)$). In conclusion, the well-posedness of the interacting particle SDE $\bR^{Nd}$-system is ensured by standard SDE results \cite[Theorem 3.5 (p.58)]{mao2008stochastic}. 

The moment bound result of the $\bR^{Nd}$-system that follows from \cite[Theorem 3.5 (p.58)]{mao2008stochastic} does not lead to \eqref{eq:momentboundParticiInteractingSystem} as the constant appearing on the right-hand side \textit{depends on $N$} and explode as $N\nearrow \infty$. Nonetheless, with well-posedness at hand, we are able to improve the bound and show \eqref{eq:momentboundParticiInteractingSystem}. 
 
The strategy of the proof is the same as that in Section \ref{subsection: Moment bound of the SSM}. For all $m\ge 2p\ge 2$,  $i\in \llbracket 1,N \rrbracket$, $t\in[0,T]$, we have 
\begin{align}
\nonumber 
     \bE \big[ |X_{t}^{i,N} |^{2p}   \big] 
     &=
     \bE \Big[   \Big|    \frac{1}{N} \sum_{j=1}^N \big(   X_{t}^{i,N}- X_{t}^{j,N}\big)  
    +\frac{1}{N} \sum_{j=1}^N X_{t}^{j,N} 
    \Big|^{2p} 
    \Big]
    \\
    \nonumber
    &\le 4^p \bE \Big[       \frac{1}{N} \sum_{j=1}^N |X_{t}^{i,N}- X_{t}^{j,N} |^{2p} \Big]
    +
    4^p \bE \Big[     \Big|\frac{1}{N} \sum_{j=1}^N   |X_{t}^{j,N}|^2 \Big|^{p}     \Big]
    \\
    \label{eq:apdx:two terms}
    &\le 4^p \bE \Big[       |X_{t}^{i,N}- X_{t}^{j,N} |^{2p} \Big]_{i\neq j}
    +
    4^p \bE \Big[     \Big|\frac{1}{N} \sum_{j=1}^N   |X_{t}^{j,N}|^2 \Big|^{p}     \Big].
\end{align}

For the first term in \eqref{eq:apdx:two terms}, by It\^o's formula, for $i,j\in \llbracket 1,N \rrbracket, i\neq j $,

\begin{align*}
     |X_{t}^{i,N}- &X_{t}^{j,N} |^{2p}
    = 
    |X_{0}^{i,N}- X_{0}^{j,N} |^{2p}
    \\
    &
    + 2p \int_0^t |X_{s}^{i,N}- X_{s}^{j,N} |^{2p-2} \Big\langle  X_{s}^{i,N}- X_{s}^{j,N}, v(X_s^{i,N}, \mu^{X,N}_{s} )- v(X_s^{j,N}, \mu^{X,N}_{s} )  \Big\rangle \dd s
    \\
    &
    +
    2p \int_0^t |X_{s}^{i,N}- X_{s}^{j,N} |^{2p-2} \Big\langle  X_{s}^{i,N}- X_{s}^{j,N}, b(s,X_s^{i,N}, \mu^{X,N}_{s} )- b(s,X_s^{j,N}, \mu^{X,N}_{s} )  \Big\rangle \dd s
     \\
    &
    +
     2p \int_0^t |X_{s}^{i,N}- X_{s}^{j,N} |^{2p-2} \Big\langle  X_{s}^{i,N}- X_{s}^{j,N}, \sigma(s,X_s^{i,N}, \mu^{X,N}_{s} )\dd W_s^i- \sigma(s,X_s^{j,N}, \mu^{X,N}_{s} )\dd W_s^j  \Big\rangle 
    \\
    &
    +
    \frac{2p(2p-1)}{2} \int_0^t |X_{s}^{i,N}- X_{s}^{j,N} |^{2p-2} \Big(  |\sigma(s,X_s^{i,N}, \mu^{X,N}_{s} )|^2+ |\sigma(s,X_s^{j,N}, \mu^{X,N}_{s} )|^2 \Big)   \dd s. 
\end{align*}
By Assumption \ref{Ass:Monotone Assumption}, Remark \ref{remark:ImpliedProperties}, Jensen's inequality,   Proposition \ref{prop:auxilary for multipy of r.v.}, take expectation on both side, by the particles are identically distributed and  Burkholder-Davis-Gundy (BDG) inequality, we have   
\begin{align*}
     \bE\big[ |X_{t}^{i,N}- &X_{t}^{j,N} |^{2p}\big]
    \le 
    \bE\big[|X_{0}^{i,N}- X_{0}^{j,N} |^{2p}\big]
    + C \int_0^t  \bE\big[|X_{s}^{i,N}- X_{s}^{j,N} |^{2p}\big]  \dd s
    + C \int_0^t  \bE\big[|X_{s}^{i,N} |^{2p}\big]  \dd s.
\end{align*}

For the second term in \eqref{eq:apdx:two terms}, similarly, and notice that,
\begin{align*}
    \frac{1}{N}& \sum_{j=1}^N   |X_{t}^{j,N}|^2
    = 
   \frac{1}{N} \sum_{j=1}^N   |X_{0}^{j,N}|^2
    + \frac{1}{N} \sum_{j=1}^N   \int_0^t   \Big\langle   X_{s}^{j,N}, v(X_s^{j,N}, \mu^{X,N}_{s} )  \Big\rangle \dd s 
    +
    \frac{1}{2N} \sum_{j=1}^N  \int_0^t   |\sigma(s,X_s^{j,N}, \mu^{X,N}_{s} )|^2    \dd s
    \\
    &\qquad\qquad\qquad 
    +
      \frac{1}{N} \sum_{j=1}^N \int_0^t  \Big\langle  X_{s}^{j,N},  b(s,X_s^{j,N}, \mu^{X,N}_{s} )  \Big\rangle \dd s
    +
     \frac{1}{N} \sum_{j=1}^N  \int_0^t  \Big\langle   X_{s}^{j,N},  \sigma(s,X_s^{j,N}, \mu^{X,N}_{s} )\dd W_s^j  \Big\rangle 
    \\
    &
    \le 
     \frac{1}{N} \sum_{j=1}^N  \Big( |X_{0}^{j,N}|^2 + 
     \int_0^t  | X_{s}^{j,N}|^2 \dd s
     +
     \int_0^t  \Big\langle   X_{s}^{j,N},  \sigma(s,X_s^{j,N}, \mu^{X,N}_{s} )\dd W_s^j  \Big\rangle
     \Big)
    + 
    \frac{C}{N^2} \sum_{i=1}^N\sum_{j=1}^N   \int_0^t    | X_{s}^{i,N}- X_{s}^{j,N}|^2 \dd s.
\end{align*}
Take power of $p$ on both side and expectations. By Jensen's inequality, BDG inequality, Proposition \ref{prop:auxilary for multipy of r.v.}, Assumption \ref{Ass:Monotone Assumption}, the Lipschitz properties on $\sigma$, we can conclude with the highest order up to $2p$, we have  
\begin{align*}
   \bE\Big[  \Big|\frac{1}{N}  \sum_{j=1}^N   |X_{t}^{j,N}|^2 \Big|^{p}  \Big] 
    &
    \le 
    C+
    C
     \bE \Big[        \frac{1}{N} \sum_{j=1}^N   |X_{0}^{j,N}|^{2p} \Big] 
    + 
     C  \int_0^t    \bE[|X_{s}^{i,N} |^{2p}]  \dd s
    +
       C \int_0^t  \bE\Big[  \Big|\frac{1}{N}  \sum_{j=1}^N   |X_{s}^{j,N}|^2 \Big|^{p}  \Big] \dd s,
\end{align*}
where we used that the particles are identically distributed to deal with the third term on the righ-hand side. 

Collecting all the above results and using \eqref{eq:apdx:two terms} again, we have
\begin{align*}
    \bE \big[ &|X_{t}^{i,N} |^{2p}   \big] 
    \le
   \bE\big[ |X_{t}^{i,N}- X_{t}^{j,N} |^{2p}\big]
   +
   \bE\Big[  \Big| \frac{1}{N} \sum_{j=1}^N   |X_{t}^{j,N}|^2 \Big|^{p}  \Big] 
    \\
    &
    \le 
    \bE\big[|X_{0}^{i,N}- X_{0}^{j,N} |^{2p}\big]
    + 
    C
     \bE \big[        |X_{0}^{i,N}|^{2p} \big] 
    +  
     C    \int_0^t   \Big(\bE\Big[|X_{s}^{i,N}- X_{s}^{j,N} |^{2p}\Big]_{i\neq j}  
    +
     \bE\Big[  \Big| \frac{1}{N} \sum_{j=1}^N   |X_{s}^{j,N}|^2 \Big|^{p}  \Big] \Big)  \dd s. 
\end{align*}

Gr\"onwall's lemma delivers the final result after taking supremum over $i\in \llbracket 1,N \rrbracket $ and $t\in[0,T]$.
\end{proof}
\color{black}

%
%
%
%
%
%
 

\section{Solving the implicit equation of the SSM and a deployment of Newton's method}
\label{appendix: discussion on Newton's method}

In this section we address solving the implicit Equation \eqref{eq:SSTM:scheme 0} in the SSM. We first present a general result stating the level of precision on needs to solve \eqref{eq:SSTM:scheme 0} such that the final convergence rate of the SSM method is preserved (e.g., Theorem \ref{theorem:SSM: strong error 1} and \ref{theorem:SSM: strong error 2}). Proposition \ref{prop: discussion of the approxi} is understood as a requirement of an adequate approximation method. In the subsequent section, we describe a deployment of Newton's method as one such method (among many) with the simulation results in Section \ref{sec:examples} showing its efficiency.

\subsection{Approximation scheme to the SSM}
Recall the SSM from Definition \ref{def:definition of the ssm}. For any timestep $n\in \llbracket 0,M-1\rrbracket$, for any particle $i\in \llbracket 1,N\rrbracket $, define $\hat{\Psi}_i:\bR^d\times\bR^{Nd}\times[0,T]\rightarrow\bR^d$ be the measurable map associating the unique solution $Y_{n}^{i,\star,N}$ of \eqref{eq:SSTM:scheme 0} to its data $\hat{ X}_{n}^{i,N}$, $\hat{ X}_{n}^{N}$ and $h$, i.e., 
\begin{align}
    \label{eq: theory generator of the SSM}
    \hp_i(\hat{ X}_{n}^{i,N},\hat{ X}_{n}^{N},h)=Y_{n}^{i,\star,N},
    \quad
    \hp=(\hp_1,\dots,\hp_N).
\end{align}
The existence of such a map $\hp$ is guaranteed by Lemma \ref{lemma:SSTM:new functions def and properties1} (see also Proposition \ref{prop:yi-yj leq xi-xj} and \ref{prop:sum y square leq sum x square} for some of its good properties). 
We next introduce a version SSM of Definition \ref{def:definition of the ssm} where the implicit equation is solved approximately only. 
\begin{definition}[Approximation scheme to the SSM] 
\label{def:definition of the approxi ssm}
We follow the notation of Definition \ref{def:definition of the ssm} hold. Denote the approximation mapping at each SSM step \eqref{eq:SSTM:scheme 0} as a measurable map  $\bp_i:\bR^d\times\bR^{Nd}\times[0,T]\rightarrow\bR^d$. The SSM variant is then, corresponding to \eqref{eq:SSTM:scheme 0}-\eqref{eq:SSTM:scheme 1}: 
set $\bx_{0}^{i,N}=X^i_0$ for $i\in \llbracket 1,N\rrbracket $;  then for all $i\in \llbracket 1,N\rrbracket $ and $n\in \llbracket 0,M-1\rrbracket$
\begin{align}
\label{eq:SSM approxi:scheme 0}
& \by_{n}^{i,\star,N} =\bp_i(\bx_{n}^{i,N},\bx_{n}^{N},h),
\quad
\bx_{n}^{N}=(\bx_{n}^{1,N},\dots,\bx_{n}^{N,N}),
 \quad 
  \bm^{Y,N}_n(\dd x):= \frac1N \sum_{j=1}^N \delta_{\by_{n}^{j,\star,N}}(\dd x),
 \\
 \label{eq:SSM approxi:scheme 1}
\bx_{n+1}^{i,N} &=\by_{n}^{i,\star,N}
            + b(t_n,\by_{n}^{i,\star,N},\bm^{Y,N}_n) h
            +\sigma(t_n,\by_{n}^{i,\star,N},\bm^{Y,N}_n) \Delta W_{n}^i,\qquad \Delta W_{n}^i=W_{t_{n+1}}^i-W_{t_n}^i,
\end{align}
where for any $i$ the map $\bp_i$ is an approximation to $\hp_i$ solving \eqref{eq: theory generator of the SSM}.
\end{definition}
We emphasise that at this point, our assumption is that the maps $\bp_i$ can be found. We discuss how to find them in the next section.
\begin{proposition} 
\label{prop: discussion of the approxi}
Let the assumptions of Theorem \ref{theorem：moment bound for the big theorm time extensions } hold. Recall the notation of Definition \ref{def:definition of the ssm} and \eqref{def:definition of the approxi ssm}. For the $\hp_i$ and $\bp_i$ defined in \eqref{eq: theory generator of the SSM} and \eqref{eq:SSM approxi:scheme 0} respectively, if $\sup_i \bE[|\hp_i(x_i,x,h)-\bp_i(x_i,x,h) |^2] \leq C h$ for all $x=(x_1,\dots,x_N)\in L_{0}^{2}( \bR^{Nd})$ and some constant $C$   (independent of $h,N,M$ but depending on $T$), then 
\begin{align}
    \label{eq:approxi numerical strong error}
    \sup_{n\in \llbracket 1,M\rrbracket} \sup_{i \in \llbracket 1,N\rrbracket } \bE\big[   |\hx_{n}^{i,N}-\bx_{n}^{i,N} |^2 \big]
  &\le Ch.
\end{align}
\end{proposition}
The main interpretation is that as long as the implicit Equation \eqref{eq:SSTM:scheme 0} is solved approximately up to an accuracy of size $h$ (the time-step increment) in $L^2$-norm, then the final order of convergence of the numerical scheme is preserved. 

\begin{proof}
We proceed by induction since for all $i \in \llbracket 1,N\rrbracket $, by definition, we have  $\hx_{0}^{i,N}=\bx_{0}^{i,N}=X^i_0$.

\textit{Step: The initial case.} We prove that 
$\sup_{i \in \llbracket 1,N\rrbracket } \bE\big[   |\hx_{1}^{i,N}-\bx_{1}^{i,N} |^2 \big] \le Ch $. By the assumptions of Proposition \ref{prop: discussion of the approxi} we have
\begin{align*}
    \sup_{i \in \llbracket 1,N\rrbracket }\bE\big[   |Y_{0}^{i,\star,N}-\by_{1}^{i,\star,N} |^2 \big]
    \le
    \sup_{i \in \llbracket 1,N\rrbracket }
    \bE\big[
    |\hp_i(X^i_0,X_0,h)-\bp_i(X^i_0,X_0,h) |^2
    \big]
    \le C h.
\end{align*}
For all  $i \in \llbracket 1,N\rrbracket $, since function $b$ and $\sigma$ are Lipschitz, by similar arguments in \eqref{eq:se1:I2 result for term b}, 
\begin{align}
\nonumber
    \sup_{i \in \llbracket 1,N\rrbracket }\bE\big[   |\hx_{1}^{i,N}-\bx_{1}^{i,N} |^2 \big]
    &\le 
    C \sup_{i \in \llbracket 1,N\rrbracket }\bE\Big[ 
    |Y_{0}^{i,\star,N}-\by_{1}^{i,\star,N} |^2
    + \big| W^{(2)}(\bm_0^{Y,N},\hm_{0}^{Y,N} )  \big|^2 h
    \Big]
    \\\label{eq:base1}
    &
    \le
     \sup_{i \in \llbracket 1,N\rrbracket }\bE\big[   |Y_{0}^{i,\star,N}-\by_{1}^{i,\star,N} |^2 \big]
     \le Ch.
\end{align}
\textit{Step: The inductive case.} For $n \in \llbracket 1,M-1\rrbracket$, given $\sup_{i \in \llbracket 1,N\rrbracket } \bE\big[   |\hx_{n}^{i,N}-\bx_{n}^{i,N} |^2 \big] \le Ch $,
we need to proof $\sup_{i \in \llbracket 1,N\rrbracket } \bE\big[   |\hx_{n+1}^{i,N}-\bx_{n+1}^{i,N} |^2 \big] \le Ch $, similarly, we first proof the result for the first step, from the assumption of Proposition \ref{prop: discussion of the approxi},
\begin{align}
\nonumber
    \sup_{i \in \llbracket 1,N\rrbracket }&\bE\big[   |Y_{n}^{i,\star,N}-\by_{n}^{i,\star,N} |^2 \big]
    =
   \sup_{i \in \llbracket 1,N\rrbracket } 
   \bE\big[
    |\hp_i(\hx_{n}^i,\hx_{n},h)-\bp(\bx_{n}^i,\bx_{n},h) |^2
    \big]
    \\
    \nonumber
    &
    \le
    2 \sup_{i \in \llbracket 1,N\rrbracket } 
   \bE\big[
    |\hp_i(\hx_{n}^i,\hx_{n},h)-\hp_i(\bx_{n}^i,\bx_{n},h) |^2
    \big]
    +
     2\sup_{i \in \llbracket 1,N\rrbracket } 
   \bE\big[
    |\hp_i(\bx_{n}^i,\bx_{n},h)-\bp_i(\bx_{n}^i,\bx_{n},h) |^2
    \big]
    \\
    \label{eq:indc1}
    &
    \le  2 \sup_{i \in \llbracket 1,N\rrbracket } 
   \bE\big[
    |\hp_i(\hx_{n}^i,\hx_{n},h)-\hp_i(\bx_{n}^i,\bx_{n},h) |^2
    \big]
    +2h.
\end{align}
Recall the results in Section \ref{subsection:Proof of strong error 1}, the arguments in \eqref{eq:se1:y-y} are satisfied for all $i\in \llbracket 1,N\rrbracket$, thus,
\begin{align*}
    \sup_{i \in \llbracket 1,N\rrbracket } 
   \bE\big[
    |\hp_i(\hx_{n}^i,\hx_{n},h)-\hp_i(\bx_{n}^i,\bx_{n},h) |^2
    \leq  
    \sup_{i \in \llbracket 1,N\rrbracket } 
   \bE\big[    | \hx_{n}^i - \bx_{n}^i |^2 (1+Ch)
    \leq Ch.
\end{align*}
 Plug the result above into \eqref{eq:indc1} to conclude 
 \begin{align*}
      \sup_{i \in \llbracket 1,N\rrbracket }&\bE\big[   |Y_{n}^{i,\star,N}-\by_{n}^{i,\star,N} |^2 \big]
      \leq Ch.
 \end{align*}
And, by similar argument in \eqref{eq:base1}, we have
\begin{align*}
    \sup_{i \in \llbracket 1,N\rrbracket }\bE\big[   |\hx_{n+1}^{i,N}-\bx_{n+1}^{i,N} |^2 \big]
    \leq Ch.
\end{align*}
\end{proof}

\subsection{Deploying Newton's method}
\label{sec:NewtonMethodJacobian}
We now provide a discussion on using Newton's method to solve \eqref{eq:SSTM:scheme 0} in the scope of the SSM. 
We first introduce Newton's method for high dimensions. Recall the functions $V,u,f$ in \eqref{Eq:General MVSDE shape of v}, \eqref{eq:Define-Function-V}, and the SSM in Definition \ref{def:definition of the ssm}. 

For simplicity of presentation, we assume that the function $u$ only depends on the space-components (this is inline with the numerical examples section) and $f$ has continuous second order derivative. Fix $x\in\bR^{Nd}$, for $y=(y_1,y_2,\dots,y_N)\in(\bR^d)^N$, for the functions $V,F:\bR^{Nd} \rightarrow \bR^{Nd}$ and $u,f:\bR^{d} \rightarrow \bR^{d}$, we want to find a solution of $y\mapsto F(y)$ (given by \eqref{eq:SSTM:scheme 0}) defined as 
\begin{align*}
    \bR^{Nd}\ni y \mapsto F(y)=y-x-hV(y)=0,\quad V=(V_1,V_2,\dots,V_N)
    \quad\textrm{and}\quad 
    V_i(y)=u(y_i)+\frac{1}{N}\sum_{j=1}^{N}f(y_i-y_j).
\end{align*}
For a fixed $x\in\bR^{Nd}$, Lemma \ref{lemma:SSTM:new functions def and properties1} ensures that a unique $y^\star$ exists satisfying $F(y^\star)=0$. 
Setting as initial guess of $y^0=x$, we denote the $\kappa^{th}$-iteration of the Newton method by $y^{\kappa}$ and define it as 
\begin{align*}
    y^0=x, \quad
    y^{\kappa+1}=y^{\kappa}-[\nabla  F]^{-1}(y^{\kappa}) F(y^{\kappa}), 
\end{align*}
where $\nabla  F$ stands for the Jacobian matrix of $F$. 

Denoting $I_{Nd}$ as the identity matrix in $Nd$-dimensions, we express the Jacobian of $F$ in closed form as 
\begin{align*}
     [\nabla F] (y ) &=I_{Nd} -h A(y)+\frac{h}{N} \Gamma(y) \quad \textrm{where for } y=(y_1,y_2,\dots,y_N)\in(\bR^d)^N \textrm{ we have }
     \\  
     A(y)&=
    \begin{bmatrix}
    \nabla u(y_1) & \cdots & 0 \\
    \vdots & \ddots & \vdots \\
    0 & \cdots &  \nabla u(y_N)
    \end{bmatrix} 
    +
    \begin{bmatrix}
    \frac{1}{N}\sum_{j=1}^{N} \nabla f(y_1-y_j) & \cdots & 0 \\
    \vdots & \ddots & \vdots \\
    0 & \cdots &  \frac{1}{N}\sum_{j=1}^{N} \nabla f(y_N-y_j)
    \end{bmatrix}
    \\
     \Gamma(y)&=
     \begin{bmatrix}
    \nabla f(y_1-y_1) & \cdots & \nabla f(y_1-y_n) \\
    \vdots & \ddots & \vdots \\
    \nabla f(y_n-y_1) & \cdots & \nabla f(y_n-y_n)
    \end{bmatrix}.
\end{align*}     
The matrix $A(y)$ is a block diagonal matrix, and $\Gamma$ is a symmetric matrix since $f$ is odd and its main diagonal is equal to $\nabla f(\mathbf{0})$. We stop the Newton's iteration at step $\kappa$ when the error tolerance rule $\| y^{\kappa }- y^{\kappa-1}\|_\infty<\sqrt{h}$ is satisfied. We note that since $\Gamma(\cdot)$ is a symmetric matrix weighted by $\frac{h}{N}$ which is an order $1/N$ smaller that $I_{Nd}$ and $hA(\cdot)$ one can think of ignoring it in favour of an approximate Newton's method.

\textit{Theoretical foundation for methodological choices.} As mentioned, Lemma \ref{lemma:SSTM:new functions def and properties1} ensures a unique $y^\star$ exists solving $F(y^\star)=0$. 
Proposition \ref{prop:yi-yj leq xi-xj} and \ref{prop:sum y square leq sum x square} ensure continuous dependence of $y^\star$ on $x$, and hence assuming $h$ small enough the choice of $y^0=x$ as the initial guess for $y^\star$ in the Newton method is justified. 
From \cite[Theorem 4.4]{Suli2003NumericsBook}, under the extra assumption that $F$ is twice differentiable with continuous derivatives, we have that the Newton iteration converges quadratically to the unique solution $y^\star$. In fact, given $h$ small enough and complementing with the trick highlighted in Remark \ref{rem:Constraint on h is soft} one can show that $V$ in \eqref{eq:Define-Function-V} has a strictly negative one-sided Lipschitz constant and hence $\nabla V$ is strict negative definite matrix (see \cite{Lionnetetal2015}) and hence so is $\nabla F$ -- this ensures that $\nabla F$ is nonsingular (also at $y^\star$) and thus \cite[Theorem 4.4]{Suli2003NumericsBook} applies guaranteeing convergence.  

In the scope of the examples presented in Section \ref{sec:examples}, with the choices above, we found that the condition $\| y^{\kappa }- y^{\kappa-1}\|_\infty<\sqrt{h}$ is attained within two to four Newton method iterations, i.e., with $\kappa \leq 4$.

\color{black}

\bibliographystyle{abbrv}


\end{document}